\documentclass{amsart}

\usepackage{amsmath,amssymb,amsthm}
\usepackage{fourier}
\usepackage{a4wide,bbm}

\newtheorem{theorem}{Theorem}[section]
\newtheorem{lemma}[theorem]{Lemma}
\newtheorem{ass}[theorem]{Assumption}
\usepackage{hyperref}

\newtheorem{remark}[theorem]{Remark}
\allowdisplaybreaks
\theoremstyle{definition}

\usepackage{xcolor,graphicx,subfigure,caption}
\theoremstyle{remark}
\newtheorem{proposition}[theorem]{Proposition}
\newtheorem{corollary}[theorem]{Corollary}

\numberwithin{equation}{section}

\begin{document}

\title[low regularity time approximation of  stochastic KdV equation]{Strong error bounds at low regularity for a time approximation of the stochastic KdV equation driven by small additive noise
}

\author{Jianbo Cui}
\address{Department of Applied Mathematics, The Hong Kong Polytechnic University, Hung Hom, Hong Kong}
\email{jianbo.cui@polyu.edu.hk}

\author{Raffaele D'Ambrosio}
\address{Department of Information Engineering, Computer Science and Mathematics, University of L'Aquila, L'Aquila, Italy.}
\email{raffaele.dambrosio@univaq.it}

\author{Stefano Di Giovacchino}
\address{Department of Information Engineering, Computer Science and Mathematics, University of L'Aquila, L'Aquila, Italy.}
\email{stefano.digiovacchino@univaq.it}

\author{Liying Sun}
\address{Academy for Multidisciplinary Studies, Capital Normal University, Beijing 100048, China}
\email{liyingsun@lsec.cc.an.cn}

\date{}


\begin{abstract}

We study strong temporal approximation of  periodic stochastic Korteweg--de Vries equation driven by small additive \(Q\)-Wiener noise of amplitude \(\mathcal O(\varepsilon)\), \(0<\varepsilon\ll1\) under $H^s$-regularity, $s=1,2$. This is a challenging problem, due to the additional derivative term in the nonlinearity and the lack of suitable exponential moment bounds for the exact solutions.
Exploiting the small-noise regime, we first decompose the solution into a deterministic KdV flow and a stochastic component; then we linearize the obtained stochastic equation and approximate the resulting equation by means of Fourier analytic techniques. 
Combining the small-noise linearization error, the discretization error of the linearized equation, and the deterministic temporal approximation error, and assuming suitable conditions on the deterministic solver, we get strong errors of size $\mathcal{O}(\tau+\varepsilon^2)$ 
 under \(H^1\)-regularity and of size $\mathcal{O}(\tau^2+\varepsilon^2)$ under $H^2$-regularity. To the best of our knowledge, these are the first strong error bounds achieved for numerical time approximations of the stochastic KdV at low regularity.
\end{abstract}

\maketitle

\section{Introduction}
\label{intro_section}

The Korteweg--de Vries (KdV) equation is a fundamental model for the
propagation of weakly nonlinear, weakly dispersive, essentially
one-directional waves.  It arises, for instance, in shallow-water theory,
ion-acoustic waves in plasmas, and nonlinear wave propagation in dispersive
media.  Stochastic KdV equations incorporate random perturbations into this
dispersive wave model and are used to describe uncertainty, unresolved
fluctuations, or random forcing effects in such physical systems; see, for
example, \cite{cha,he,sca}.

In this paper we consider the following 
stochastic KdV equation 
\begin{equation}
    \label{kdv_eq}
\dot{u}+\partial_x^3 u -\mu\partial_x (u^2)=\varepsilon Q^{1/2}\dot{W}, \quad t\in[0,T], \quad x \in \mathbbm{T}=[-\pi,\pi],
\end{equation}
subject to periodic boundary conditions.
Here \(T,\mu>0\),
$0<\varepsilon\ll1$ measures the noise intensity, \(W\) is a cylindrical
Wiener process on a completed filtered probability space
\((\Omega,\mathcal F,(\mathcal F_t)_{t\ge0},\mathbbm{P})\), and
\(Q^{1/2}\) is a Hilbert--Schmidt operator (see Section \ref{ass_section} for details).   The well-posedness and qualitative properties of stochastic KdV equations
have been investigated both on the real line and on periodic domains; see, for instance, \cite{deb2,deb0} for the real-line
case and \cite{deb1,Oh} for periodic settings.  Explicit solutions are
available only in very special cases, such as certain one-soliton
configurations driven by a standard Brownian motion \cite{deb1,lin}. Consequently, numerical approximation
is indispensable for understanding the dynamics of stochastic KdV equations
beyond such exceptional cases.

For deterministic dispersive equations, including the KdV equation obtained from \eqref{kdv_eq} with \(\varepsilon=0\), numerical approximation has been studied extensively; see, for example, \cite{kat0} and the references therein. A central difficulty in the numerical analysis of the deterministic KdV equation is the presence of the third-order derivative $\partial_x^3 u$, which typically forces standard temporal discretizations to require $H^{s+3}$-regularity of the exact solution in order to prove first-order convergence in $H^s$, $s\ge0$. This difficulty has motivated the development of low regularity integrators for deterministic KdV equations. In particular, \cite{kat0} proposed an exponential-type integrator with convergence order $1/2$ in $H^2$ and order $1$ in $H^1$ for exact solutions in $H^3$. A low regularity convergence analysis in $L^2$ was obtained in \cite{kat1} for filtered time discretizations, with order $\min(1,s/3)$ for solutions in $H^s$, while \cite{li} established order of convergence $\gamma \in (0,1]$ in $H^\gamma$, up to a logarithmic correction, for an unfiltered low regularity integrator under $H^1$-regularity.

The stochastic case remains much less understood.  Much of the numerical
literature on stochastic dispersive PDEs concerns nonlinear stochastic
Schr\"odinger equations \cite{arm,bre,cui,cui3,cui2,dig} and stochastic wave equations  \cite{cao,cohenlang,hong2}.
For stochastic KdV equations, existing work has mainly focused on trajectory simulation
and predicting the evolution or conservation of physical quantities. For instance simulations based on Crank--Nicolson-type schemes have been studied
in \cite{deb3,lin}, while conservation issues for stochastic \(\theta\)-methods
and finite element discretizations were considered in \cite{da1,da0,liu_ma}.
Recently, suitable convergence results for an explicit exponential method for the damped stochastic KdV equation under $H^4$-regularity have been provided in \cite{li_new2}, where the authors obtained the results introducing an artificial viscosity and imposing a CFL restriction. In \cite{li_new1}, the strong convergence of the Crank-Nicolson method introduced in \cite{deb3} has been addressed under $H^{9/2}$-regularity.
However,
strong error estimates for temporal approximations of stochastic KdV
equations at low regularity, and in particular of \eqref{kdv_eq}, appear to be largely absent
from the existing literature.

There are two main obstacles. First, a direct temporal discretization has to
handle the additional derivative appearing in the nonlinear term. Second, at low regularity regime, the full nonlinear stochastic KdV dynamics
are not known to satisfy the exponential integrability estimates that are
commonly used in strong convergence analysis for stochastic PDEs with
superlinear nonlinearities; see, for instance, \cite{cox,cui}. These two issues
make a direct stochastic extension of deterministic low regularity KdV
integrators highly nontrivial.

The aim of this work is to develop novel approach to overcome these difficulties in the small-noise
regime. Our strategy is to introduce the rescaled fluctuation
$$
u^\varepsilon :=\frac{1}{\varepsilon} (u-\psi),
$$
where $\psi$ solves the deterministic KdV equation, namely \eqref{kdv_eq} with $\varepsilon=0$. This gives the decomposition
$$
u=\psi+\varepsilon u^\varepsilon,
$$
which separates the
nonlinear deterministic evolution from the random perturbation. The
deterministic component \(\psi\) is then approximated by means of 
low regularity integrators available in the literature (see, e.g.,  \cite{kat0}),
while the stochastic fluctuation is treated through
a linearization argument. More precisely, it turns out that \(u^\varepsilon\) satisfies a
stochastic KdV-type equation with an \(\mathcal O(1)\) noise term, an
\(\mathcal O(\varepsilon)\) quadratic nonlinearity, and a linear transport term
whose coefficient is determined by \(\psi\). Neglecting the
\(\mathcal O(\varepsilon)\) nonlinear term leads to a linear stochastic
fluctuation equation, whose solution is denoted by \(\chi\).

 For $t_n=n\tau$ and $\tau>0$, the proposed approximation has the form
\[
u_n=\psi_n+\varepsilon\chi_n,
\]
where \(\psi_n\) approximates \(\psi(t_n)\) and \(\chi_n\) approximates
\(\chi(t_n)\). The strong error  $\|u(t_n)-u_n\|_{L^2(\Omega; L^2)}$ is then decomposed as
$$
\|u(t_n)-u_n\|_{L^2(\Omega; L^2)}\le \|\psi(t_n)-\psi_n\|+\varepsilon\|u^\varepsilon(t_n)-\chi(t_n)\|_{L^2(\Omega; L^2)}+\varepsilon \|\chi(t_n)-\chi_n\|_{L^2(\Omega; L^2)}.
$$
This error decomposition is essential: it isolates the deterministic discretization error, the small-noise linearization error, and
the temporal discretization error of the linearized stochastic fluctuation. In particular, owing to the linearity, the last error term can be analyzed without invoking exponential moment estimates for either the exact or numerical solutions. This is one of the main advantages of our approach.

The main results of this paper are threefold. First,
under suitable regularity assumptions on \(\psi\) and \(u^\varepsilon\),  we show that the
linearized fluctuation provides a first-order approximation of the rescaled
stochastic fluctuation, i.e.,
\[
    \|u^\varepsilon(t_n)-\chi(t_n)\|_{L^2(\Omega;L^2)}
    =\mathcal O(\varepsilon).
\]
Second, for the linearized stochastic fluctuation equation, we construct a low regularity time integrator by introducing a suitable change of variables and carrying out an exact integration in Fourier space.

Third, by
combining the deterministic approximation error, the small-noise linearization
error, and the fluctuation discretization error, we establish the strong error
estimate 
\[
\sup_{0\le n\le N}
\|u(t_n)-u_n\|_{L^2(\Omega;L^2)}
=
\mathcal O\big(\tau+\varepsilon^2\big)
\]
under \(H^1\)-regularity of the stochastic fluctuation and assuming order $1$ of convergence in $L^2$ for the deterministic integrator. Under \(H^2\)-regularity of the stochastic solution and further requiring convergence with order $2$ in $L^2$ and $1$ in $H^1$ to the deterministic solver , this estimate improves to
\[
\sup_{0\le n\le N}
\|u(t_n)-u_n\|_{L^2(\Omega;L^2)}
=
\mathcal O\big(\tau^2+\varepsilon^2\big).
\]

To the best of our knowledge, these are the first strong error bounds obtained for temporal discretizations of stochastic KdV equations at low regularity.

It is worthwhile remarking that these bounds are joint small-noise time-discretization bounds. For fixed $\varepsilon>0$, the proposed numerical approximation converges to $\psi+\varepsilon\chi$, and then the overall numerical error retains a modeling error $C\varepsilon^2$, if compared to the original solution $u$; nevertheless, the estimates show that, in the small-noise regime, strong
approximations of size $\mathcal{O}(\tau)$ can already be obtained for stochastic solutions with
\(H^1\)-regularity. This contrasts with more classical schemes, such as
Crank--Nicolson-type and classical exponential-type methods, which typically require higher regularity to
exhibit comparable convergence behavior \cite{li_new1,li_new2}; this point, as well as, the superiority of the performance of the proposed integrator, compared with classical integrators, will be further illustrated in
the numerical experiments in Section~\ref{exp_section}.  

Beyond the stochastic
KdV equation, the analysis demonstrates how low regularity
techniques in Fourier space can be combined with stochastic perturbative arguments to obtain
strong convergence rates for stochastic dispersive PDEs. In this sense, the
present work contributes to the emerging numerical analysis of stochastic dispersive equations at low regularity regime. Related contributions in stochastic
dispersive PDEs include \cite{arm,cao,cui2,dig}. 

The paper is organized as follows. In Section \ref{ass_section}, we introduce the notation, assumptions, and preliminary results used throughout the paper. In Section \ref{lin_section}, we present the perturbative decomposition, the linearized stochastic problem and we analyze the small-noise linearization error. In Section \ref{lw_section}, we construct the low regularity time discretization for the linearized stochastic equation. Section \ref{err_section} is devoted to the strong error analysis of the proposed numerical scheme. The theoretical findings are illustrated by numerical experiments in Section \ref{exp_section}, and the proofs of technical results are collected in Appendix \ref{appen}.

\section{Assumptions and preliminary results}
\label{ass_section}
This section fixes the functional framework and collects the assumptions and preliminary estimates used throughout the paper. We first recall the Fourier and Sobolev notation on the one-dimensional torus, then specify the stochastic KdV equation and the noise setting, and finally present the analytic estimates used repeatedly in the error analysis.

\subsection{Functional setting and Fourier notation}

Throughout the paper, we work on the one-dimensional torus
\(\mathbbm T=[-\pi,\pi]\)
with periodic boundary conditions. For any \(p\in[1,\infty]\), we denote by \(L^p=L^p(\mathbb T;\mathbbm R)\) the space of real-valued periodic functions on \(\mathbbm T\). For \(p\in[1,\infty)\), its norm is given by
\[
    \|f\|_{L^p}
    =
    \left(\int_{\mathbbm T}|f(x)|^p\,dx\right)^{1/p},
\]
and, for \(p=\infty\),
\[
    \|f\|_{L^\infty}
    =
    \operatorname*{ess\,sup}_{x\in\mathbbm T}|f(x)|.
\]
When \(p=2\), we simply write \(\|f\|=\|f\|_{L^2}\). The space \(L^2\) is endowed with the inner product
\[
    \langle f,g\rangle
    =
    \int_{\mathbbm T} f(x)g(x)\,dx,
    \qquad f,g\in L^2.
\]

For \(f\in L^2\), we use the Fourier expansion
\[
    f(x)=\frac{1}{\sqrt{2\pi}}\sum_{\ell\in\mathbbm Z}\widehat f_\ell e^{\mathrm i\ell x},
\]
where \(\{\widehat f_\ell\}_{\ell\in\mathbbm{Z}}\) denotes the sequence of Fourier coefficients of \(f\). For \(s\in\mathbbm R\), we define \(H^s=H^s(\mathbb T;\mathbbm R)\) by the norm
\[
    \|f\|_{H^s}
    =
    \left(\sum_{\ell\in\mathbbm{Z}}(1+|\ell|^2)^s |\widehat f_\ell|^2\right)^{1/2}.
\]
For \(s\ge0\), the homogeneous Sobolev seminorm is denoted by
\[
    \|f\|_{\dot H^s}
    =
    \left(\sum_{\ell\in\mathbbm Z}|\ell|^{2s}|\widehat f_\ell|^2\right)^{1/2}.
\]
In particular, when \(s\) is a nonnegative integer, \(\|f\|_{\dot H^s}=\|\partial_x^s f\|\).

For \(q\in[1,\infty)\) and a function
\[
    w(x)=\frac{1}{\sqrt{2\pi}}\sum_{\ell\in\mathbbm Z}\widehat w_\ell e^{\mathrm i\ell x},
\]
we write
\[
    \|w\|_{\ell^q}
    =
    \left(\sum_{\ell\in\mathbbm{Z}}|\widehat w_\ell|^q\right)^{1/q},
    \qquad
    \ell^q:=\{w:\|w\|_{\ell^q}<\infty\}.
\]
For \(j\in\mathbbm N_0\), we also set
\[
    w^{(j)}(x)
    =
    \frac{1}{\sqrt{2\pi}}
    \sum_{\ell\in\mathbbm Z}|\ell|^j|\widehat w_\ell| e^{\mathrm i\ell x}.
\]
Here and below, \(\mathbbm N=\{1,2,\ldots\}\) and \(\mathbbm N_0=\{0\}\cup\mathbbm N\). Moreover, \(W^{k,\infty}=W^{k,\infty}(\mathbbm T;\mathbbm R)\), \(k\in\mathbbm N_0\), denotes the usual Sobolev space of functions whose weak derivatives up to order \(k\) belong to \(L^\infty\).

Let \(\sigma\ge0\). We denote by \(\mathcal L_2^\sigma\) the space of Hilbert--Schmidt operators from \(L^2\) to \(H^\sigma\), equipped with the norm
\[
    \|\Phi\|_{\mathcal L_2^\sigma}
    =
    \left(\sum_{k\in\mathbbm N}\|\Phi e_k\|_{H^\sigma}^2\right)^{1/2},
\]
where \(\{e_k\}_{k\in\mathbbm N}\) is any orthonormal basis of \(L^2\). The value of this norm is independent of the particular choice of the orthonormal basis. Finally, for a Banach space \(X\) and \(p\in[1,\infty)\), we denote by \(L^p(\Omega;X)\) the space of \(X\)-valued random variables with finite norm
\[
    \|Y\|_{L^p(\Omega;X)}
    =
    \left(\mathbb E\|Y\|_X^p\right)^{1/p}.
\]

Throughout this work, the nonrandom constant \(C\) may change from line to line. We write \(A\lesssim B\) if \(A\le cB\) for a positive nonrandom constant \(c\) independent of \(\tau\), \(\varepsilon\), and the quantities being estimated.

\subsection{The stochastic KdV equation and well-posedness}

We consider the stochastic periodic KdV equation
\begin{equation}
\label{eq:skdv-section2}
    du(t)
    =
    -\partial_x^3u(t)\,dt
    +\mu\partial_x\bigl(u(t)^2\bigr)\,dt
    +\varepsilon Q^{1/2}\,dW(t),
    \qquad u(0)=\xi.
\end{equation}
Here \(\mu>0\), \(0<\varepsilon\ll1\), and the initial datum \(\xi\) is nonrandom. The process \(W\) is a cylindrical Wiener process on \(L^2\), written as
\[
    W(t)=\sum_{k=1}^\infty \beta_k(t)e_k,
\]
where \(\{\beta_k\}_{k\in\mathbbm{N}}\) is a sequence of independent standard Brownian motions and \(\{e_k\}_{k\in\mathbbm{N}}\) is an orthonormal basis of \(L^2\). The operator \(Q^{1/2}\) is assumed to be self-adjoint and Hilbert--Schmidt on \(L^2\). Additional spatial regularity and further conditions on $Q^{1/2}$ will be specified in the sequel of this work.

Let
\[
    S(t)=e^{-t\partial_x^3},
    \qquad t\ge0,
\]
denote the Airy group on the torus. The mild form of \eqref{eq:skdv-section2} is
\begin{equation}
\label{eq:mild-section2}
    u(t)
    =
    S(t)\xi
    +\mu\int_0^t S(t-s)\partial_x\bigl(u(s)^2\bigr)\,ds
    +\varepsilon\int_0^t S(t-s)Q^{1/2}\,dW(s).
\end{equation}
We will use the following standard well-posedness result for the stochastic KdV equation (see, e.g. \cite{deb2,deb0,deb1,Oh}).

\begin{proposition}
    [Well-posedness]
\label{thm:wellposedness-section2}
Assume that \(\xi\in H^1\) and \(Q^{1/2}\in\mathcal L_2^1\). Then, for every \(T>0\), equation \eqref{eq:skdv-section2} admits a unique mild solution
\[
    u\in C([0,T];H^1), \qquad \text{a.s.}
\]
Equivalently, \(u\) satisfies \eqref{eq:mild-section2} for every \(t\in[0,T]\).
\end{proposition}

\subsection{Time-dependent spaces and basic estimates}

We next introduce the deterministic and stochastic function spaces used in the subsequent analysis. For \(s\in\mathbb N_0\) and \(0\le t_1\le t_2\), define
\[
    \mathbbm L_{t_1,t_2}^{\infty,s}
    :=
    \left\{f:[t_1,t_2]\to H^s:
    \|f\|_{L^\infty([t_1,t_2];H^s)}<\infty\right\}.
\]
When \(t_1=0\), we use the abbreviation
\[
    \mathbbm L_{0,t}^{\infty,s}=\mathbbm L_t^{\infty,s},
    \qquad t\ge0.
\]
Similarly, for \(p\in[1,\infty)\), \(s\in\mathbbm{N}_0\), and \(0\le t_1\le t_2\), set
\[
    \mathcal L_{t_1,t_2}^{\infty,p,s}
    :=
    \left\{v:[t_1,t_2]\to L^p(\Omega;H^s):
    \|v\|_{L^\infty([t_1,t_2];L^p(\Omega;H^s))}<\infty\right\}.
\]
Again, we write
\[
    \mathcal L_{0,t}^{\infty,p,s}=\mathcal L_t^{\infty,p,s},
    \qquad t\ge0.
\]

The following one-dimensional Gagliardo--Nirenberg inequalities will be used repeatedly: for every \(f\in H^1\) with $\hat{f}_0=0$, one has
\begin{align}
    \|f\|_{L^\infty}
    &\lesssim
    \|f\|^{1/2}\|\partial_x f\|^{1/2},
    \label{gn1}
    \\
    \|f\|_{L^3}
    &\lesssim
    \|\partial_x f\|^{1/6}\|f\|^{5/6}.
    \label{gn2}
\end{align}

We shall also use the algebra property of Sobolev spaces: for \(s>1/2\),
\begin{equation}
\label{bil_est}
    \|fg\|_{H^s}
    \lesssim
    \|f\|_{H^s}\|g\|_{H^s},
    \qquad f,g\in H^s.
\end{equation}

When the functions have zero spatial mean, the homogeneous seminorm controls the lower-order norms. More precisely, if \(\widehat f_0=0\), then, for every \(\sigma\ge0\),
\begin{equation}
\label{hom_b1}
    \|f\|\le \|f\|_{\dot H^\sigma}.
\end{equation}
If, in addition, \(\sigma>1/2\), then
\begin{equation}
\label{hom_b2}
    \|f\|_{L^\infty}
    \lesssim
    \|f\|_{\dot H^\sigma}.
\end{equation}
Finally, for zero-mean functions and \(s>1/2\), we use the homogeneous bilinear estimate
\begin{equation}
\label{bil_est_hom}
    \|fg\|_{\dot H^s}
    \lesssim
    \|f\|_{\dot H^s}\|g\|_{\dot H^s}.
\end{equation}

We will frequently apply It\^o's formula to functionals of solutions of SPDEs. In some places, the formal regularity required by It\^o's formula may not be available directly \cite{dap0}. The identities used below can nevertheless be justified by a standard regularization procedure, for example by applying the arguments to Galerkin approximations and then passing to the limit \cite{deb2}.

\section{\texorpdfstring{Perturbative decomposition and linearization argument}{Perturbative decomposition and linearization argument}}
\label{lin_section}

In this section, we present the main idea behind the construction of our temporal discretization of the stochastic KdV equation \eqref{kdv_eq}.
We first split the exact stochastic solution into a deterministic KdV component and a rescaled stochastic fluctuation. We then introduce a linearized stochastic equation for this fluctuation and prove that, in the small-noise regime, the fluctuation is approximated by its linearization with a strong error of size $\mathcal{O}(\varepsilon)$.

Let \(u\) be the mild solution to \eqref{kdv_eq} with nonrandom initial datum $u(0)=\xi$. We denote by \(\psi\) the solution of the deterministic KdV equation with the same initial datum, namely
\begin{equation}\label{ans0}
\dot{\psi}
=
-\partial_x^3\psi+\mu\partial_x(\psi^2),
\qquad \psi(0)=\xi.
\end{equation}
Define the process
\begin{equation}
    \label{fluct}
u^\varepsilon(t):=\frac{1}{\varepsilon}(u(t)-\psi(t)).
\end{equation}
Equivalently, 
\begin{equation}\label{ans}
u(t)=\psi(t)+\varepsilon u^\varepsilon(t).
\end{equation}
This identity separates the deterministic evolution from the random perturbation and is the starting point of the small-noise analysis.

Using \eqref{kdv_eq} and \eqref{ans0}, we get
\begin{align*}
    d u^\varepsilon (t)&=\frac{1}{\varepsilon}\big(d u(t)-d \psi(t)\big)\\
    &=\frac{1}{\varepsilon}\big[-\partial_x^3 u(t)dt+\mu \partial_x( u(t)^2 )dt+\varepsilon Q^{1/2}dW(t)+\partial_x^3 \psi(t)dt-\mu\partial_x (\psi(t)^2)dt\big]\\
    &=-\partial_x^3 u^\varepsilon(t)dt+Q^{1/2}dW(t)+\mu \partial_x \big((u(t)+\psi(t))u^\varepsilon(t)\big)dt\\
    &=-\partial_x^3 u^\varepsilon(t)dt+Q^{1/2}dW(t)+\mu \partial_x (u^\varepsilon(t) \psi(t))dt+\mu\partial_x (u^\varepsilon(t)u(t))dt.
\end{align*}
Moreover, by \eqref{fluct}, 
\begin{align*}
u^\varepsilon(t)u(t) &= u^\varepsilon(t)\big(\frac{1}{\varepsilon}(u(t)-\psi(t))\varepsilon+\psi(t)\big)\\
    &=\varepsilon(u^\varepsilon(t))^2+u^\varepsilon(t)\psi(t).
\end{align*}
Therefore the fluctuation \(u^\varepsilon\) satisfies
\begin{equation}\label{KdVs}
du^\varepsilon(t)
=
\Bigl(
-\partial_x^3u^\varepsilon(t)
+2\mu\partial_x(\psi(t) u^\varepsilon(t))
+\mu\varepsilon\partial_x\bigl((u^\varepsilon(t))^2\bigr)
\Bigr)\,dt
+Q^{1/2}\,dW(t),
\end{equation}
with initial condition $u^\varepsilon(0)=0$.
Equation \eqref{KdVs} is a stochastic KdV-type equation for the fluctuation. It contains an \(\mathcal O(1)\) additive noise term, a linear transport term \(2\mu\partial_x(\psi u^\varepsilon)\) determined by the deterministic solution \(\psi\), and a quadratic nonlinear term of size \(\mathcal O(\varepsilon)\).

\begin{remark}
\label{wp_stoc}
By the classical well-posedness theory for the periodic KdV equation
\cite{bour0,gub,tao}, if \(\xi\in H^\sigma\) for some
\(\sigma\in \mathbbm N\), then, for every \(T>0\),
\[
\psi\in C([0,T];H^\sigma).
\]

\end{remark}

The small parameter \(\varepsilon\) suggests that the quadratic term in \eqref{KdVs} should be treated as a perturbation. We therefore introduce the linearized fluctuation equation obtained by neglecting the \(\mathcal O(\varepsilon)\) nonlinear term. This gives the auxiliary linear stochastic problem
\begin{equation}\label{lin_eq}
d\chi(t)
=
\bigl(-\partial_x^3\chi(t)+2\mu\partial_x(\psi(t)\chi(t))\bigr)\,dt
+Q^{1/2}\,dW(t),
\qquad \chi(0)=0 .
\end{equation}
The process \(\chi\) captures the leading stochastic fluctuation around the deterministic KdV solution \(\psi\).

As in \cite{kat0}, the following zero-mode condition is imposed throughout the sequel of this paper.

 \begin{ass}[Zero-mode condition]
\label{ass_zero_mode}
We assume that the initial datum has zero spatial mean, namely
\[
\widehat{\xi}_0=0,
\]
and that the noise does not excite the zero Fourier mode. More precisely, the covariance square-root operator satisfies
\[
Q^{1/2}{\bf 1}=0,
\]where ${\bf 1}$ stands for the zero Fourier mode. Moreover, we assume that $Q^{1/2}e_k$ has zero spatial mean, for any $k\in\mathbbm{N}^+$.
\end{ass}

\begin{remark}\label{rmk_zero}
For the deterministic solution \(\psi\), mass conservation holds; see, for instance, \cite{kat0}. Hence, Assumption~\ref{ass_zero_mode} implies
\[
\widehat{\psi}_0(t)=0, \qquad t\ge 0.
\]
For the stochastic solution \(\chi\), since \(\chi(0)=0\), it follows directly from Assumption~\ref{ass_zero_mode} that
\[
\widehat{\chi}_0(t)=0, \qquad t\ge 0,
\]
almost surely. Similarly, also $\widehat{u}_0(t)=0, t\ge 0$, and then $\widehat{u^\varepsilon}_0(t)=0, t\ge 0$, almost surely. Finally, if the zero-mode condition is not imposed, one may use the modification described in \cite[Remark~1.3]{kat0}. We do not pursue this extension here.
\end{remark}

We give the following lemma, whose proof is put in Appendix \ref{appen_h1b}.

\begin{lemma}\label{lem_H1b}
Suppose that $Q^{1/2}\in\mathcal{L}_2^1$ and $\xi \in H^2$. Then the linearized equation \eqref{lin_eq} admits a unique solution satisfying $\chi \in C([0,T]; H^1)$. Moreover, for every $p\in\mathbbm{N}$,
\begin{equation}
\label{es_t}
\mathbbm{E}\Big[\sup_{t\in[0,T]}\|\chi(t)\|_{H^1}^{2p}\Big]
+
\sup_{\varepsilon\in(0,1)}
\mathbbm{E}\Big[\sup_{t\in[0,T]}\|u^\varepsilon(t)\|_{H^1}^{2p}\Big]
<\infty.
\end{equation}
\end{lemma}

We now quantify the error made by replacing the nonlinear fluctuation \(u^\varepsilon\) with its linearization \(\chi\).

\begin{proposition}
\label{prop_wp_lin}
Assume the hypotheses of Lemma \ref{lem_H1b}. Then, for any $p\in\mathbbm{N}$, there exists a constant $C>0$, depending on $p$, $T$, $\|Q^{1/2}\|_{\mathcal{L}_2^1}$, $\displaystyle\sup_{t\in[0,T]}\|\partial_x \psi(t)\|_{L^\infty}$, $\|\chi\|_{\mathcal{L}_{0,T}^{\infty,4p,1}}$, and $\displaystyle\sup_{\varepsilon\in(0,1)}\|u^{\varepsilon}\|_{\mathcal{L}_{0,T}^{\infty,8p,1}}$, such that
$$
\sup_{t\in[0,T]}
\|u^{\varepsilon}(t)-\chi(t)\|_{L^{2p}(\Omega;L^2)}
\le C\varepsilon,
$$
for every $\varepsilon\in(0,1)$.
\end{proposition}

\begin{proof}
Let
$
w^\varepsilon(t):=u^\varepsilon(t)-\chi(t).
$  The stochastic forcing cancels in the difference. 
By applying the chain rule to $\|w^\varepsilon(t)\|^{2p}$, and using integration by parts, we obtain
\begin{equation*}
d\|w^\varepsilon(t)\|^{2p}
=
4\mu p\,\|w^\varepsilon(t)\|^{2p-2}
\bigl(I_1(t)+\varepsilon I_2(t)\bigr)\,dt,
\end{equation*}
where
$$
I_1(t):=\frac12\big\langle (w^\varepsilon(t))^2,\partial_x \psi(t)\big\rangle,
\qquad
I_2(t):=\big\langle w^\varepsilon(t),
u^\varepsilon(t)\partial_x u^\varepsilon(t)\big\rangle.
$$

The term $I_1$ is controlled by
$$
I_1(t)
\le \frac12 \|\partial_x \psi(t)\|_{L^\infty}\|w^\varepsilon(t)\|^2.
$$
Therefore,
\begin{align}
d\|w^\varepsilon(t)\|^{2p}
\le{}&
2\mu p\,\|\partial_x \psi(t)\|_{L^\infty}
\|w^\varepsilon(t)\|^{2p}\,dt+4\mu p\,\varepsilon\,\|w^\varepsilon(t)\|^{2p-2}\,I_2(t)\,dt .
\label{f10}
\end{align}

It remains to estimate the nonlinear perturbation.
Since \(u^\varepsilon=w^\varepsilon+\chi\), we decompose 
\begin{align*}
\varepsilon I_2(t)
&=
\varepsilon \int_{\mathbbm{T}}
w^\varepsilon(t)
u^\varepsilon(t)\partial_x u^\varepsilon(t)\,dx \\
&=
\varepsilon \int_{\mathbbm{T}}
(w^\varepsilon(t))^2\partial_x u^\varepsilon(t)\,dx
\varepsilon \int_{\mathbbm{T}}
w^\varepsilon(t)
\chi(t)\partial_x u^\varepsilon(t)\,dx .
\end{align*}
For the first term, by the Cauchy--Schwarz inequality, the Gagliardo--Nirenberg inequality \eqref{gn1} and Young's inequality, we have
\begin{align*}
\varepsilon \int_{\mathbbm{T}}
(w^\varepsilon(t))^2\partial_x u^\varepsilon(t)\,dx
&\le
\varepsilon \|w^\varepsilon(t)\|_{L^\infty}\|w^\varepsilon(t)\|
\|\partial_x u^\varepsilon(t)\| \\
&\le
C\varepsilon \|w^\varepsilon(t)\|^{3/2}
\|\partial_x w^\varepsilon(t)\|^{1/2}
\|\partial_x u^\varepsilon(t)\| \\
&\le
\frac{3C}{4}\|w^\varepsilon(t)\|^2
+\frac{C}{4}\varepsilon^4
\|\partial_x w^\varepsilon(t)\|^2
\|\partial_x u^\varepsilon(t)\|^4 .
\end{align*} 
Similarly, for the second term, 
\begin{align*}
\varepsilon \int_{\mathbbm{T}}
w^\varepsilon(t)
\chi(t)\partial_x u^\varepsilon(t)\,dx
&\le
\varepsilon \|w^\varepsilon(t)\|\|\chi(t)\|_{L^\infty}
\|\partial_x u^\varepsilon(t)\| \\
&\le
C\varepsilon \|w^\varepsilon(t)\|
\|\chi(t)\|^{1/2}\|\partial_x\chi(t)\|^{1/2}
\|\partial_x u^\varepsilon(t)\| \\
&\le
\frac{C}{2}\|w^\varepsilon(t)\|^2
+\frac{C}{2}\varepsilon^2
\|\chi(t)\|\|\partial_x\chi(t)\|
\|\partial_x u^\varepsilon(t)\|^2 .
\end{align*}
Combining the above estimates, we obtain
\begin{align*}
\varepsilon I_2(t)
\le
\frac{5C}{4}\|w^\varepsilon(t)\|^2
+\frac{C}{4}\varepsilon^4
\|\partial_x w^\varepsilon(t)\|^2
\|\partial_x u^\varepsilon(t)\|^4
+\frac{C}{2}\varepsilon^2
\|\chi(t)\|\|\partial_x\chi(t)\|
\|\partial_x u^\varepsilon(t)\|^2 .
\end{align*}
Since $\varepsilon\in(0,1)$, it follows that
$$
\varepsilon^4
\|\partial_x w^\varepsilon(t)\|^2
\|\partial_x u^\varepsilon(t)\|^4
\le
\varepsilon^2 
\|\partial_x w^\varepsilon(t)\|^2
\|\partial_x u^\varepsilon(t)\|^4 .
$$
Therefore,
\begin{align}
\label{f1}
\varepsilon I_2(t)
\le
\frac{5C}{4}\|w^\varepsilon(t)\|^2
+\varepsilon^2 K\bigl(\chi(t),u^\varepsilon(t)\bigr),
\end{align}
where
\begin{align*}
K\bigl(\chi(t),u^\varepsilon(t)\bigr)
&:=
\frac{C}{2}\|\partial_x u^\varepsilon(t)\|^2
\bigg[
\|\chi(t)\|\,\|\partial_x\chi(t)\|
+\frac{1}{2}
\|\partial_x u^\varepsilon(t)-\partial_x\chi(t)\|^2
\|\partial_x u^\varepsilon(t)\|^2
\bigg].
\end{align*}
Substituting \eqref{f1} into \eqref{f10}, we find
\begin{align*}
d\|w^\varepsilon(t)\|^{2p}
\le{}&
\Big(
2\mu p\,\|\partial_x\psi(t)\|_{L^\infty}
+5\mu p C
\Big)
\|w^\varepsilon(t)\|^{2p}\,dt+
4\mu p\,\varepsilon^2\,
\|w^\varepsilon(t)\|^{2p-2}
K\bigl(\chi(t),u^\varepsilon(t)\bigr)\,dt.
\end{align*}
Integrating in time and applying Young's inequality to the last term, we obtain
\begin{align*}
\mathbbm{E}\big[\|w^\varepsilon(t)\|^{2p}\big]
\le{}&
\Big(
2\mu p\,\sup_{t\in[0,T]}\|\partial_x\psi(t)\|_{L^\infty}
+5\mu p C
+4\mu(p-1)
\Big)
\int_0^t
\mathbbm{E}\big[\|w^\varepsilon(s)\|^{2p}\big]\,ds \\
&+
4\mu t\,\varepsilon^{2p}
\sup_{t\in[0,T]}
\mathbbm{E}\Big[
K\bigl(\chi(t),u^\varepsilon(t)\bigr)^p
\Big].
\end{align*}
We now bound the last factor.  
Using the Cauchy--Schwarz inequality, we obtain
\begin{align*}
&\mathbbm{E}\Big[
K\bigl(\chi(t),u^\varepsilon(t)\bigr)^p
\Big] \\
&\qquad \le
2^{p-1}\Big(\frac{C}{2}\Big)^p
\mathbbm{E}\Big[
\|\partial_x u^\varepsilon(t)\|^{2p}
\|\chi(t)\|^p
\|\partial_x\chi(t)\|^p
\Big] \\
&\qquad\quad +
2^{p-1}\Big(\frac{C}{4}\Big)^p
\mathbbm{E}\Big[
\|\partial_x u^\varepsilon(t)\|^{4p}
\|\partial_x u^\varepsilon(t)-\partial_x\chi(t)\|^{2p}
\Big] \\
&\qquad \le
2^{p-1}\Big(\frac{C}{2}\Big)^p
\Big(
\mathbbm{E}\big[\|\partial_x u^\varepsilon(t)\|^{4p}\big]
\Big)^{1/2}
\Big(
\mathbbm{E}\big[\|\chi(t)\|^{2p}\|\partial_x\chi(t)\|^{2p}\big]
\Big)^{1/2} \\
&\qquad\quad +
2^{p-1}\Big(\frac{C}{4}\Big)^p
\Big(
\mathbbm{E}\big[\|\partial_x u^\varepsilon(t)\|^{8p}\big]
\Big)^{1/2}
\Big(
\mathbbm{E}\big[\|\partial_x u^\varepsilon(t)-\partial_x\chi(t)\|^{4p}\big]
\Big)^{1/2}.
\end{align*}
Thanks to Lemma \ref{lem_H1b} and the assumptions of the proposition, we then infer that 
$$
\sup_{\varepsilon\in(0,1)}\sup_{t\in[0,T]}
\mathbbm{E}\Big[
K\bigl(\chi(t),u^\varepsilon(t)\bigr)^p
\Big]
<\infty.
$$
Therefore,
$$
\mathbbm{E}\big[\|u^\varepsilon(t)-\chi(t)\|^{2p}\big]
\le
C\int_0^t
\mathbbm{E}\big[\|u^\varepsilon(s)-\chi(s)\|^{2p}\big]\,ds
+C\varepsilon^{2p}.
$$
An application of Gronwall's inequality yields
$$
\sup_{t\in[0,T]}
\mathbbm{E}\big[\|u^\varepsilon(t)-\chi(t)\|^{2p}\big]
\le C\varepsilon^{2p},
$$
for every $\varepsilon\in(0,1)$,
and hence
$$
\sup_{t\in[0,T]}
\|u^\varepsilon(t)-\chi(t)\|_{L^{2p}(\Omega;L^2)}
\le C\varepsilon,
$$
for every $\varepsilon\in(0,1)$. This completes the proof.
\end{proof}

\section{Low regularity time approximation of the linearized equation}
\label{lw_section}
This section constructs a low regularity time integrator for the linearized stochastic equation \eqref{lin_eq}. The derivation relies on the introduction of twisted variables and a suitable exact time integration, in Fourier space, of an integral term arising in the construction.

We start from the mild formulation of \eqref{lin_eq}. Since \(\chi(0)=0\), one has 
$$
\chi(t)
=
2\mu \int_{0}^{t}{\rm e}^{-(t-s)\partial_x^3}\partial_x\big(\psi(s)\chi(s)\big)\,ds
+\int_0^t {\rm e}^{-(t-s)\partial_x^3} Q^{1/2}\,dW(s),
$$
where $\psi$ is the solution of the deterministic KdV equation \eqref{ans0}. We next introduce the twisted variables
$$
\tilde{\chi}(t) = {\rm e}^{t\partial_x^3}\chi(t),
\qquad
\tilde{\psi}(t)={\rm e}^{t\partial_x^3}\psi(t).
$$
The Airy group is an isometry on every Sobolev space \(H^\sigma\); hence
\begin{equation}
\label{iso_twi}
\|\widetilde\chi(t)\|_{H^\sigma}=\|\chi(t)\|_{H^\sigma},
\qquad
\|\widetilde\psi(t)\|_{H^\sigma}=\|\psi(t)\|_{H^\sigma},
\qquad \sigma,t\in\mathbbm{R}.
\end{equation}
Then $\tilde{\chi}$ and $\tilde{\psi}$ satisfy
\begin{align}
\label{mild_sol_twis_0}
\tilde{\chi}(t)
=
2\mu \int_{0}^{t}{\rm e}^{s\partial_x^3}
\partial_x\Big(
{\rm e}^{-s\partial_x^3}\tilde{\psi}(s)\,
{\rm e}^{-s\partial_x^3}\tilde{\chi}(s)
\Big)\,ds
+\int_{0}^{t}{\rm e}^{s\partial_x^3}Q^{1/2}\,dW(s),
\end{align}
and
$$
\tilde{\psi}(t)
=
\xi
+\mu \int_{0}^{t}{\rm e}^{s\partial_x^3}
\partial_x\Big(
{\rm e}^{-s\partial_x^3}\tilde{\psi}(s)
\Big)^2\,ds,\quad  \text{respectively.}
$$

We shall also use the following time-regularity estimate for the exact solution of the linearized equation. The proof is put in Appendix \ref{app_time_reg}.

\begin{lemma}
\label{time_reg_est}
Let $Q^{1/2}\in\mathcal{L}_2^1$, and $\xi\in H^2$.
Then, for any $t,s\ge 0$ with $t+s\in[0,T]$, the twisted solution $\tilde{\chi}$ \eqref{mild_sol_twis_0} satisfies the following time-regularity estimate:
$$
\|\tilde\chi(t+s)-\tilde\chi(t)\|_{L^2(\Omega;H^1)}\lesssim \sqrt{s}.
$$
\end{lemma}

\subsection{Derivation of the low regularity time approximation}

Let \(\tau>0\), \(t_n=n\tau\), \(n=0,1,\ldots,N\), and \(N\tau=T\). Applying \eqref{mild_sol_twis_0} on the interval \([t_n,t_{n+1}]\), after the change of variables \(s\mapsto t_n+s\), gives
\begin{equation}
\label{mild_sol_twis_stoc}
\begin{aligned}
\widetilde\chi(t_n+\tau)
=
\widetilde\chi(t_n)
&+2\mu \int_0^\tau {\rm e}^{(t_n+s)\partial_x^3}
\partial_x\Bigl(
{\rm e}^{-(t_n+s)\partial_x^3}\widetilde\psi(t_n+s)\,
{\rm e}^{-(t_n+s)\partial_x^3}\widetilde\chi(t_n+s)
\Bigr)\,ds \\
&+\int_{t_n}^{t_n+\tau}{\rm e}^{s\partial_x^3}Q^{1/2}\,dW(s).
\end{aligned}
\end{equation}
Similarly, \(\widetilde\psi\) satisfies
\begin{equation}
\label{mild_sol_twis_det}
\widetilde\psi(t_n+\tau)
=
\widetilde\psi(t_n)
+
\mu \int_0^\tau {\rm e}^{(t_n+s)\partial_x^3}
\partial_x\Bigl(
{\rm e}^{-(t_n+s)\partial_x^3}\widetilde\psi(t_n+s)
\Bigr)^2\,ds.
\end{equation}

The additive stochastic convolution in \eqref{mild_sol_twis_stoc} can be simulated exactly under standard structural assumptions on the covariance operator. We state this point separately because it is useful for implementation.  

\begin{remark}
\label{simul}
Exact simulation of the stochastic convolution in \eqref{met} is available, for example, when \(Q\) is diagonal in the Fourier basis. More precisely, assume that there exists a real sequence \(\{q_\ell\}_{\ell\in\mathbbm{Z}}\) such that
\[
Q {\rm e}^{i\ell x}=q_\ell {\rm e}^{i\ell x},
\qquad \ell\in\mathbbm{Z}.
\]
This is the setting used in the numerical experiments. For a general covariance operator, one may approximate the stochastic convolution by
\[
\int_{t_n}^{t_{n+1}} {\rm e}^{-(t_{n+1}-s)\partial_x^3}Q^{1/2}\,dW(s)
\approx
{\rm e}^{-\tau\partial_x^3}Q^{1/2}\bigl(W(t_{n+1})-W(t_n)\bigr).
\]
This Euler approximation has strong order \(0<\gamma\le 1\) in \(H^s\), provided that \(Q^{1/2}\in\mathcal{L}_2^{s+3\gamma}\). In particular, first order in \(L^2\) follows if \(Q^{1/2}\in\mathcal{L}_2^3\).
\end{remark}

It remains to approximate the deterministic integral in \eqref{mild_sol_twis_stoc}. Thanks to the regularity estimate \eqref{time_reg_est}, we are allowed to freeze the twisted variables at time \(t_n\), keeping the oscillatory phases exactly:
\begin{align}
&\int_0^\tau {\rm e}^{(t_n+s)\partial_x^3}
\partial_x\Bigl(
{\rm e}^{-(t_n+s)\partial_x^3}\widetilde\psi(t_n+s)\,
{\rm e}^{-(t_n+s)\partial_x^3}\widetilde\chi(t_n+s)
\Bigr)\,ds \notag\\
&\qquad\approx
\int_0^\tau {\rm e}^{(t_n+s)\partial_x^3}
\partial_x\Bigl(
{\rm e}^{-(t_n+s)\partial_x^3}\widetilde\psi(t_n)\,
{\rm e}^{-(t_n+s)\partial_x^3}\widetilde\chi(t_n)
\Bigr)\,ds.
\label{discr}
\end{align}
The right-hand side can be evaluated exactly. Indeed, write
\[
\widetilde\psi(t_n)=\frac{1}{\sqrt{2\pi}}\sum_{\ell\in\mathbbm{Z}}
\widehat{\widetilde\psi}_\ell(t_n){\rm e}^{i\ell x},
\qquad
\widetilde\chi(t_n)=\frac{1}{\sqrt{2\pi}}\sum_{\ell\in\mathbbm{Z}}
\widehat{\widetilde\chi}_\ell(t_n){\rm e}^{i\ell x}.
\]
Then
\begin{align*}
&\partial_x\Bigl(
{\rm e}^{-(t_n+s)\partial_x^3}\widetilde\psi(t_n)\,
{\rm e}^{-(t_n+s)\partial_x^3}\widetilde\chi(t_n)
\Bigr) \\
&\quad=
\frac{1}{2\pi}
\sum_{\ell_1,\ell_2\in\mathbbm{Z}}
\widehat{\widetilde\psi}_{\ell_1}(t_n)
\widehat{\widetilde\chi}_{\ell_2}(t_n)
{\rm e}^{i(\ell_1^3+\ell_2^3)(t_n+s)}
i(\ell_1+\ell_2){\rm e}^{i(\ell_1+\ell_2)x}.
\end{align*}
Consequently,
\begin{align}
&\int_0^\tau {\rm e}^{(t_n+s)\partial_x^3}
\partial_x\Bigl(
{\rm e}^{-(t_n+s)\partial_x^3}\widetilde\psi(t_n)\,
{\rm e}^{-(t_n+s)\partial_x^3}\widetilde\chi(t_n)
\Bigr)\,ds\notag \\
&\quad=
\frac{i}{2\pi}
\sum_{\ell_1,\ell_2\in\mathbbm{Z}}
\int_0^\tau
{\rm e}^{-i\bigl((\ell_1+\ell_2)^3-\ell_1^3-\ell_2^3\bigr)(t_n+s)}\,ds\,
(\ell_1+\ell_2)
\widehat{\widetilde\psi}_{\ell_1}(t_n)
\widehat{\widetilde\chi}_{\ell_2}(t_n)
{\rm e}^{i(\ell_1+\ell_2)x}\label{ad0}.
\end{align}
Using the identity
\begin{equation}
\label{alg_ide}
(\ell_1+\ell_2)^3-\ell_1^3-\ell_2^3=3\ell_1\ell_2(\ell_1+\ell_2),
\end{equation}
and the fact that for $\ell_1+\ell_2=0$ the right-hand side of \eqref{ad0} is zero, we get
\begin{align*}
 &   \int_0^\tau {\rm e}^{(t_n+s)\partial_x^3}
\partial_x\Bigl(
{\rm e}^{-(t_n+s)\partial_x^3}\widetilde\psi(t_n)\,
{\rm e}^{-(t_n+s)\partial_x^3}\widetilde\chi(t_n)
\Bigr)\,ds\\
&\quad=\frac{i}{2\pi}
\sum_{\substack{\ell_1,\ell_2\in\mathbbm{Z}\\ \ell_1+\ell_2\ne 0}}
\int_0^\tau
{\rm e}^{-i3(\ell_1+\ell_2)\ell_1\ell_2(t_n+s)}\,ds\,
(\ell_1+\ell_2)
\widehat{\widetilde\psi}_{\ell_1}(t_n)
\widehat{\widetilde\chi}_{\ell_2}(t_n)
{\rm e}^{i(\ell_1+\ell_2)x}\\
&\quad=\frac{i}{2\pi}
\sum_{\substack{\ell_1,\ell_2\in\mathbbm{Z}\\ \ell_1+\ell_2\ne 0\\ \ell_1\ne 0, \ell_2\ne 0}}
\int_0^\tau
{\rm e}^{-i3(\ell_1+\ell_2)\ell_1\ell_2(t_n+s)}\,ds\,
(\ell_1+\ell_2)
\widehat{\widetilde\psi}_{\ell_1}(t_n)
\widehat{\widetilde\chi}_{\ell_2}(t_n)
{\rm e}^{i(\ell_1+\ell_2)x}\\
&\qquad\qquad\qquad\qquad  +\frac{\tau}{2\pi} \hat{\tilde\psi}_{0}(t_n)\sum_{\ell\ne 0} i \ell\hat{\tilde\chi}_{\ell}(t_n){\rm e}^{i\ell x}+\frac{\tau}{2\pi} \hat{\tilde\chi}_{0}(t_n)\sum_{\ell\ne 0} i \ell\hat{\tilde\psi}_{\ell}(t_n){\rm e}^{i\ell  x}.
\end{align*}
Note that, for $\ell_1,\ell_2\ne 0, \ell_1+\ell_2\ne 0$, we have
$$
\int_0^\tau
{\rm e}^{-i3(\ell_1+\ell_2)\ell_1\ell_2(t_n+s)}\,ds = \Big({\rm e}^{-i(t_n+\tau)\bigl((\ell_1+\ell_2)^3-\ell_1^3-\ell_2^3\Big)}
-
{\rm e}^{-it_n\bigl((\ell_1+\ell_2)^3-\ell_1^3-\ell_2^3\bigr)}
\Bigr)
\frac{1}{-3i\ell_1\ell_2(\ell_1+\ell_2)}.
$$
Thus, we achieve
\begin{align*}
   & \int_0^\tau {\rm e}^{(t_n+s)\partial_x^3}
\partial_x\Bigl(
{\rm e}^{-(t_n+s)\partial_x^3}\widetilde\psi(t_n)\,
{\rm e}^{-(t_n+s)\partial_x^3}\widetilde\chi(t_n)
\Bigr)\,ds \\
&=\frac{1}{2\pi}
\sum_{\substack{\ell_1,\ell_2\in\mathbbm{Z}\\ \ell_1,\ell_2\ne0}}
\Bigl(
{\rm e}^{-i(t_n+\tau)\bigl((\ell_1+\ell_2)^3-\ell_1^3-\ell_2^3\bigr)}
-
{\rm e}^{-it_n\bigl((\ell_1+\ell_2)^3-\ell_1^3-\ell_2^3\bigr)}
\Bigr)
\frac{1}{-3\ell_1\ell_2}
\widehat{\widetilde\psi}_{\ell_1}(t_n)
\widehat{\widetilde\chi}_{\ell_2}(t_n)
{\rm e}^{i(\ell_1+\ell_2)x}\\
&\qquad\qquad\qquad\qquad  +\frac{\tau}{2\pi} \hat{\tilde\psi}_{0}(t_n)\sum_{\ell\ne 0} i \ell\hat{\tilde\chi}_{\ell}(t_n){\rm e}^{i\ell x}+\frac{\tau}{2\pi} \hat{\tilde\chi}_{0}(t_n)\sum_{\ell\ne 0} i \ell\hat{\tilde\psi}_{\ell}(t_n){\rm e}^{i\ell  x}.
\end{align*}
Now, invoking Assumption \ref{ass_zero_mode}, one has $\hat{\tilde\psi}_{0}(t_n)=\hat{\tilde\chi}_{0}(t_n)=0$. Thus, we get 
\begin{align}
&\int_0^\tau {\rm e}^{(t_n+s)\partial_x^3}
\partial_x\Bigl(
{\rm e}^{-(t_n+s)\partial_x^3}\widetilde\psi(t_n)\,
{\rm e}^{-(t_n+s)\partial_x^3}\widetilde\chi(t_n)
\Bigr)\,ds\notag \\
&
=\frac{1}{2\pi}
\sum_{\substack{\ell_1,\ell_2\in\mathbbm{Z}\\ \ell_1,\ell_2\ne0}}
\Bigl(
{\rm e}^{-i(t_n+\tau)\bigl((\ell_1+\ell_2)^3-\ell_1^3-\ell_2^3\bigr)}
-
{\rm e}^{-it_n\bigl((\ell_1+\ell_2)^3-\ell_1^3-\ell_2^3\bigr)}
\Bigr)
\frac{1}{-3\ell_1\ell_2}
\widehat{\widetilde\psi}_{\ell_1}(t_n)
\widehat{\widetilde\chi}_{\ell_2}(t_n)
{\rm e}^{i(\ell_1+\ell_2)x}\label{fou0}.
\end{align}

If we define the operator $\partial_x^{-1}$ as \cite{kat0}
$$
\partial_x^{-1} {\rm e}^{i\ell x} = \begin{cases}
    \frac{1}{i\ell}{\rm e}^{i\ell x}, & \ell\ne 0,\\
    0, & \ell=0,
\end{cases}
$$
i.e.,
\begin{equation}
\label{inv_op}
\partial_x^{-1} f(x)
=
\frac{1}{\sqrt{2\pi}}
\sum_{\ell\ne0}\frac{1}{i\ell}\widehat f_\ell{\rm e}^{i\ell x},
\end{equation}
then \eqref{fou0} gives
\begin{align}
    &\int_0^\tau {\rm e}^{(t_n+s)\partial_x^3}
\partial_x\Bigl(
{\rm e}^{-(t_n+s)\partial_x^3}\widetilde\psi(t_n)\,
{\rm e}^{-(t_n+s)\partial_x^3}\widetilde\chi(t_n)
\Bigr)\,ds\notag \\
&=\frac{1}{3}\Big[{\rm e}^{(t_n+\tau)\partial_x^3}
\Bigl({\rm e}^{-(t_n+\tau)\partial_x^3}\partial_x^{-1}\widetilde\psi(t_n)\Bigr)
\Bigl({\rm e}^{-(t_n+\tau)\partial_x^3}\partial_x^{-1}\widetilde\chi(t_n)\Bigr)-{\rm e}^{t_n\partial_x^3}
\Bigl({\rm e}^{-t_n\partial_x^3}\partial_x^{-1}\widetilde\psi(t_n)\Bigr)
\Bigl({\rm e}^{-t_n\partial_x^3}\partial_x^{-1}\widetilde\chi(t_n)\Bigr) \Big].\label{fou1}
\end{align}

Then, in view of \eqref{fou1} and \eqref{discr}, we obtain following approximation of \(\widetilde\chi(t_{n+1})\):
\begin{align*}
\widetilde\chi(t_{n+1})
\approx{}&
\widetilde\chi(t_n)
+
\frac{2\mu}{3}{\rm e}^{(t_n+\tau)\partial_x^3}
\Bigl({\rm e}^{-(t_n+\tau)\partial_x^3}\partial_x^{-1}\widetilde\psi(t_n)\Bigr)
\Bigl({\rm e}^{-(t_n+\tau)\partial_x^3}\partial_x^{-1}\widetilde\chi(t_n)\Bigr) \\
&-
\frac{2\mu}{3}{\rm e}^{t_n\partial_x^3}
\Bigl({\rm e}^{-t_n\partial_x^3}\partial_x^{-1}\widetilde\psi(t_n)\Bigr)
\Bigl({\rm e}^{-t_n\partial_x^3}\partial_x^{-1}\widetilde\chi(t_n)\Bigr)
+
\int_{t_n}^{t_n+\tau}{\rm e}^{s\partial_x^3}Q^{1/2}\,dW(s).
\end{align*}
It is worth noting that, under Assumption~\ref{ass_zero_mode}, the right-hand side of the above approximation has vanishing zero Fourier mode. This allows us to construct a low regularity integrator for \eqref{lin_eq}.

 Let \(\widetilde{\psi}_n\) be a numerical approximation of \(\widetilde{\psi}(t_n)\) obtained from a deterministic KdV integrator which preserves the spatial mean.
 More precisely, under Assumption~\ref{ass_zero_mode}, we assume that
\begin{align}  \label{zero_det}
(\widehat{\widetilde{\psi}}_n
)_0=0
\qquad n=0,1,\ldots,N .
\end{align}
Unless otherwise specified, condition~\eqref{zero_det} is assumed throughout the remainder of the paper.
Deterministic KdV integrators satisfying such a zero-mode preservation property are available in the literature; see, for instance, \cite{kat0} and the references therein. This zero-mode preservation is needed in order to apply the identity \eqref{fou1} in the subsequent error analysis.

Starting from \(\widetilde{\chi}_0=0\), and using the fact that the stochastic forcing has no zero Fourier mode by Assumption~\ref{ass_zero_mode}, we define the approximation of \(\widetilde{\chi}(t_n)\) recursively by
\begin{align}
\widetilde\chi_{n+1}
={}&
\widetilde\chi_n
+
\frac{2\mu}{3}{\rm e}^{(t_n+\tau)\partial_x^3}
\Bigl({\rm e}^{-(t_n+\tau)\partial_x^3}\partial_x^{-1}\widetilde\psi_n\Bigr)
\Bigl({\rm e}^{-(t_n+\tau)\partial_x^3}\partial_x^{-1}\widetilde\chi_n\Bigr) \notag\\
&-
\frac{2\mu}{3}{\rm e}^{t_n\partial_x^3}
\Bigl({\rm e}^{-t_n\partial_x^3}\partial_x^{-1}\widetilde\psi_n\Bigr)
\Bigl({\rm e}^{-t_n\partial_x^3}\partial_x^{-1}\widetilde\chi_n\Bigr)
+
\int_{t_n}^{t_n+\tau}{\rm e}^{s\partial_x^3}Q^{1/2}\,dW(s).
\label{met_twis}
\end{align}

It can be verified that the above recursion also preserves the zero Fourier mode, that is,
\[
(\widehat{\widetilde{\chi}}_n
)_0=0,
\qquad n=0,1,\ldots,N,
\]
almost surely. Hence the scheme remains in the zero-mean subspace at every time step.

Finally, transforming back to the original variable and using
\[
\psi_n=e^{-t_n\partial_x^3}\widetilde{\psi}_n,
\]
we obtain the time-stepping scheme
\begin{align}
\chi_{n+1}
={}&
{\rm e}^{-\tau\partial_x^3}\chi_n
+
\frac{2\mu}{3}
\Bigl({\rm e}^{-\tau\partial_x^3}\partial_x^{-1}\psi_n\Bigr)
\Bigl({\rm e}^{-\tau\partial_x^3}\partial_x^{-1}\chi_n\Bigr) \notag\\
&\hspace{2cm}-
\frac{2\mu}{3}{\rm e}^{-\tau\partial_x^3}
\Bigl(\partial_x^{-1}\psi_n\Bigr)
\Bigl(\partial_x^{-1}\chi_n\Bigr)
+
\int_{t_n}^{t_{n+1}}{\rm e}^{-(t_{n+1}-s)\partial_x^3}Q^{1/2}\,dW(s).
\label{met}
\end{align}

The numerical scheme~\eqref{met} is implemented as follows. We first compute a deterministic approximation \(\psi_n\), \(n=0,1,\ldots,N\), of the KdV equation on the same temporal grid. Given \(\chi_n\) and \(\psi_n\), the next stochastic fluctuation \(\chi_{n+1}\) is then computed by~\eqref{met}. Additional assumptions on the deterministic approximation \(\psi_n\) will be stated in the subsequent error analysis.
 
\section{Error analysis}
\label{err_section}

In this section we will present a strong error analysis for the low regularity approximation \eqref{met} and then we use it to achieve strong error estimates for the approximation $u_n=\psi_n+\varepsilon \chi_n$ to the solution to the stochastic KdV equation \eqref{kdv_eq}.

\subsection{One-step formulation, local error and stability  estimates}

Let $\varphi_t^\tau$ denote the exact one-step flow map associated with \eqref{mild_sol_twis_stoc}. More precisely, for any $t\ge 0$, $\tau>0$, and any $\mathcal F_t$-measurable initial datum $v$, $\varphi_t^\tau(v)$ is defined as the exact value at time $t+\tau$ of the solution to \eqref{mild_sol_twis_stoc} starting from $v$ at time $t$. In particular,
$
\tilde \chi(t+\tau)=\varphi_t^\tau(\tilde \chi(t)).
$

We also introduce the numerical one-step map $\Phi_t^\tau$ associated with the low regularity integrator for the twisted variable $\tilde\chi$. For $v,z$ with zero spatial mean, we define
\begin{align}
\Phi_t^\tau(v,z)
={}&
v
+2\mu\int_0^\tau {\rm e}^{\partial_x^3(t+s)}\partial_x
\Big(
{\rm e}^{-\partial_x^3(t+s)}z\,
{\rm e}^{-\partial_x^3(t+s)}v
\Big)\,ds
+\int_t^{t+\tau}{\rm e}^{s\partial_x^3}Q^{1/2}\,dW(s)\notag\\
={}&
v
+\frac{2\mu}{3}
{\rm e}^{(t+\tau)\partial_x^3}
\Big(
{\rm e}^{-(t+\tau)\partial_x^3}\partial_x^{-1}z
\Big)
\Big(
{\rm e}^{-(t+\tau)\partial_x^3}\partial_x^{-1}v
\Big)
\label{num_map}\\
&\quad
-\frac{2\mu}{3}
{\rm e}^{t\partial_x^3}
\Big(
{\rm e}^{-t\partial_x^3}\partial_x^{-1}z
\Big)
\Big(
{\rm e}^{-t\partial_x^3}\partial_x^{-1}v
\Big)
+\int_t^{t+\tau}{\rm e}^{s\partial_x^3}Q^{1/2}\,dW(s).\notag
\end{align}

In particular, by \eqref{met_twis}, the numerical approximation can be written compactly as 
$
\tilde \chi_{n+1}=\Phi_{t_n}^\tau(\tilde \chi_n,\tilde{\psi}_n).
$

We begin with the local truncation error of the one-step approximation.
\begin{lemma}[Local error]
\label{Loc_err_lem}
Assume that $Q^{1/2}\in\mathcal{L}_2^1$, and $\xi\in H^2$. Then, for every $n=0,\dots,N-1$, with $t_n=n\tau$, one has
$$
\big\|\varphi_{t_n}^\tau(\tilde \chi(t_n))-\Phi_{t_n}^\tau(\tilde \chi(t_n),\tilde\psi(t_n))\big\|_{L^2(\Omega;L^2)}
\lesssim \tau^{3/2}.
$$
\end{lemma}

\begin{proof} 
First, note that,by assumptions, $\|\tilde \psi\|_{\mathbbm{L}_T^{\infty,2}}<\infty,\;
\|\tilde{\chi}\|_{\mathcal{L}_T^{\infty,p,1}}<\infty, \;
p\in \mathbbm{N}.$ We now estimate the mean-square local error. By \eqref{mild_sol_twis_stoc} and \eqref{num_map}, we have
\begin{align}
&\mathbbm{E}\big\|
\varphi_{t_n}^\tau(\tilde \chi(t_n))
-\Phi_{t_n}^\tau(\tilde \chi(t_n),\tilde\psi(t_n))
\big\|^2 \notag\\
={}&
4\mu^2\,
\mathbbm{E}\Bigg\|
\int_0^\tau
{\rm e}^{(t_n+s)\partial_x^3}\partial_x
\Big(
{\rm e}^{-(t_n+s)\partial_x^3}\tilde\psi(t_n+s)\,
{\rm e}^{-(t_n+s)\partial_x^3}\tilde \chi(t_n+s)
-{\rm e}^{-(t_n+s)\partial_x^3}\tilde\psi(t_n)\,
{\rm e}^{-(t_n+s)\partial_x^3}\tilde \chi(t_n)
\Big)\,ds
\Bigg\|^2.
\label{tin}
\end{align}
We split the difference inside the integral into two parts 
$$
\mathbbm{E}\big\|
\varphi_{t_n}^\tau(\tilde \chi(t_n))
-\Phi_{t_n}^\tau(\tilde \chi(t_n),\tilde\psi(t_n))
\big\|^2
\lesssim L_1+L_2,
$$
where
\begin{align*}
L_1
&=
\mathbbm{E}\Bigg\|
\int_0^\tau
{\rm e}^{(t_n+s)\partial_x^3}\partial_x
\Big[
{\rm e}^{-(t_n+s)\partial_x^3}
\big(\tilde\psi(t_n+s)-\tilde\psi(t_n)\big)\,
{\rm e}^{-(t_n+s)\partial_x^3}\tilde \chi(t_n+s)
\Big]\,ds
\Bigg\|^2,\\
L_2
&=
\mathbbm{E}\Bigg\|
\int_0^\tau
{\rm e}^{(t_n+s)\partial_x^3}\partial_x
\Big[
{\rm e}^{-(t_n+s)\partial_x^3}\tilde\psi(t_n)\,
{\rm e}^{-(t_n+s)\partial_x^3}
\big(\tilde \chi(t_n+s)-\tilde \chi(t_n)\big)
\Big]\,ds
\Bigg\|^2.
\end{align*}

We first estimate $L_1$. Since ${\rm e}^{t\partial_x^3}$ is an isometry on Sobolev spaces, the Hölder inequality gives
\begin{align*}
L_1
&\le
\mathbbm{E}\Bigg(
\int_0^\tau
\Big\|
{\rm e}^{-(t_n+s)\partial_x^3}
\big(\tilde\psi(t_n+s)-\tilde\psi(t_n)\big)\,
{\rm e}^{-(t_n+s)\partial_x^3}\tilde \chi(t_n+s)
\Big\|_{\dot{H}^1}\,ds
\Bigg)^2\\
&\le
\tau
\int_0^\tau
\mathbbm{E}
\Big\|
{\rm e}^{-(t_n+s)\partial_x^3}
\big(\tilde\psi(t_n+s)-\tilde\psi(t_n)\big)\,
{\rm e}^{-(t_n+s)\partial_x^3}\tilde \chi(t_n+s)
\Big\|_{\dot{H}^1}^2\,ds.
\end{align*}
Next, combining the bilinear estimate \eqref{hom_b1} and \eqref{reg_n0}, we infer that
\begin{align}
L_1
&\lesssim
\tau
\int_0^\tau
\|\tilde\psi(t_n+s)-\tilde\psi(t_n)\|_{\dot{H}^1}^2
\,\mathbbm{E}\|\tilde \chi(t_n+s)\|_{\dot{H}^1}^2\,ds
\notag\\
&\lesssim
\tau
\int_0^\tau
\|\tilde\psi(t_n+s)-\tilde\psi(t_n)\|_{\dot{H}^1}^2\,ds.
\label{I1_0}
\end{align}

With arguments analogous to those in the proof of Lemma 2.5 in \cite{kat0}, we infer that
\begin{equation}
\label{t0}
\|\tilde\psi(t_n+s)-\tilde\psi(t_n)\|_{\dot{H}^1}\lesssim \sqrt{s}.
\end{equation}
Substituting \eqref{t0} into \eqref{I1_0}, we obtain
\begin{equation}
\label{t8}
L_1\lesssim \tau^3.
\end{equation}

We now turn to $L_2$. Arguing as above, and using Lemma \ref{time_reg_est} together with \eqref{reg_n0}, we obtain
\begin{equation}
\label{t7}
\begin{aligned}
L_2
&\lesssim
\tau\int_0^\tau
\mathbbm{E}
\Big\|
{\rm e}^{-(t_n+s)\partial_x^3}\tilde\psi(t_n)\,
{\rm e}^{-(t_n+s)\partial_x^3}
\big(\tilde \chi(t_n+s)-\tilde \chi(t_n)\big)
\Big\|_{\dot{H}^1}^2\,ds\\
&\lesssim
\tau\int_0^\tau
\mathbbm{E}\|\tilde \chi(t_n+s)-\tilde \chi(t_n)\|_{\dot{H}^1}^2\,ds\\
&\lesssim \tau^3.
\end{aligned}
\end{equation}
Combining \eqref{t8} and \eqref{t7} with \eqref{tin}, and then taking square roots, yields the desired estimate.
\end{proof}

The next lemma gives stability of the numerical map $\Phi_t^\tau$ with respect to its first argument.

\begin{lemma}
\label{stab_lem}
For any $t\ge 0$, $\tau>0$, any $\mathcal F_t$-measurable random variables $v,w\in L^2(\Omega;L^2)$, and any nonrandom $z$ satisfying $z^{(1)}\in L^\infty$, one has
$$
\big\|\Phi_t^\tau(v,z)-\Phi_t^\tau(w,z)\big\|_{L^2(\Omega; L^2)}
\le
(1+c\tau)^{1/2}\big\|v-w\big\|_{L^2(\Omega; L^2)},
$$
where the constant $c>0$ depends on $\|z^{(1)}\|_{L^\infty}$.
\end{lemma}

\begin{proof} Set
$
\delta:=v-w.
$ 
From the definition of $\Phi_t^\tau$, we have
\begin{align*}
\Phi_t^\tau(v,z)-\Phi_t^\tau(w,z)
={}&
\delta
+\frac{2\mu}{3}
{\rm e}^{(t+\tau)\partial_x^3}
\Big(
{\rm e}^{-(t+\tau)\partial_x^3}\partial_x^{-1}z
\Big)
\Big(
{\rm e}^{-(t+\tau)\partial_x^3}\partial_x^{-1}\delta
\Big)
-\frac{2\mu}{3}
{\rm e}^{t\partial_x^3}
\Big(
{\rm e}^{-t\partial_x^3}\partial_x^{-1}z
\Big)
\Big(
{\rm e}^{-t\partial_x^3}\partial_x^{-1}\delta
\Big).
\end{align*}
Therefore,
$$
\|\Phi_t^\tau(v,z)-\Phi_t^\tau(w,z)\|^2
=
\|\delta\|^2+\frac{4\mu^2}{9}(J_1+J_2),
$$
where
\begin{align*}
J_1
&=
\frac{3}{\mu}
\Big\langle
{\rm e}^{(t+\tau)\partial_x^3}
\Big[
\big({\rm e}^{-(t+\tau)\partial_x^3}\partial_x^{-1}z\big)
\big({\rm e}^{-(t+\tau)\partial_x^3}\partial_x^{-1}\delta\big)
\Big]
-
{\rm e}^{t\partial_x^3}
\Big[
\big({\rm e}^{-t\partial_x^3}\partial_x^{-1}z\big)
\big({\rm e}^{-t\partial_x^3}\partial_x^{-1}\delta\big)
\Big],
\delta
\Big\rangle,\\
J_2
&=
\Big\|
{\rm e}^{(t+\tau)\partial_x^3}
\Big[
\big({\rm e}^{-(t+\tau)\partial_x^3}\partial_x^{-1}z\big)
\big({\rm e}^{-(t+\tau)\partial_x^3}\partial_x^{-1}\delta\big)
\Big]
-
{\rm e}^{t\partial_x^3}
\Big[
\big({\rm e}^{-t\partial_x^3}\partial_x^{-1}z\big)
\big({\rm e}^{-t\partial_x^3}\partial_x^{-1}\delta\big)
\Big]
\Big\|^2.
\end{align*}

We begin with the term $J_1$. Introduce the twisted variables
$$
\tilde z:={\rm e}^{-t\partial_x^3}z,\qquad
\tilde v:={\rm e}^{-t\partial_x^3}v,\qquad
\tilde w:={\rm e}^{-t\partial_x^3}w.
$$
Then
$$
\delta
=
{\rm e}^{t\partial_x^3}(\tilde v-\tilde w),
\qquad
{\rm e}^{-(t+\tau)\partial_x^3}\delta
=
{\rm e}^{-\tau\partial_x^3}(\tilde v-\tilde w),
\qquad
{\rm e}^{-(t+\tau)\partial_x^3}z
=
{\rm e}^{-\tau\partial_x^3}\tilde z.
$$
Substituting these identities into the definition of $J_1$, and using
$$
\langle {\rm e}^{t\partial_x^3}f,g\rangle
=
\langle f,{\rm e}^{-t\partial_x^3}g\rangle,
\qquad
\langle f,g\rangle
=
-\langle \partial_x f,\partial_x^{-1}g\rangle,
$$
we obtain
\begin{align*}
\frac{\mu}{3}J_1
&=
\Big\langle
{\rm e}^{\tau\partial_x^3}
\Big[
\big({\rm e}^{-\tau\partial_x^3}\partial_x^{-1}\tilde z\big)
\big({\rm e}^{-\tau\partial_x^3}\partial_x^{-1}(\tilde v-\tilde w)\big)
\Big]
-
(\partial_x^{-1}\tilde z)\big(\partial_x^{-1}(\tilde v-\tilde w)\big),
\tilde v-\tilde w
\Big\rangle\\
&=
-\Big\langle
\partial_x
\Big[
\big({\rm e}^{-\tau\partial_x^3}\partial_x^{-1}\tilde z\big)
\big({\rm e}^{-\tau\partial_x^3}\partial_x^{-1}(\tilde v-\tilde w)\big)
\Big],
\partial_x^{-1}{\rm e}^{-\tau\partial_x^3}(\tilde v-\tilde w)
\Big\rangle
+
\Big\langle
\partial_x
\Big[
(\partial_x^{-1}\tilde z)\big(\partial_x^{-1}(\tilde v-\tilde w)\big)
\Big],
\partial_x^{-1}(\tilde v-\tilde w)
\Big\rangle.
\end{align*}
We now decompose $J_1$ as
\begin{align*}
\frac{\mu}{3}J_1
&=-
\Big\langle
\big({\rm e}^{-\tau\partial_x^3}\tilde z\big)
\big({\rm e}^{-\tau\partial_x^3}\partial_x^{-1}(\tilde v-\tilde w)\big),
\partial_x^{-1}{\rm e}^{-\tau\partial_x^3}(\tilde v-\tilde w)
\Big\rangle\\
&\quad
-
\Big\langle
\big({\rm e}^{-\tau\partial_x^3}\partial_x^{-1}\tilde z\big)
\big({\rm e}^{-\tau\partial_x^3}(\tilde v-\tilde w)\big),
\partial_x^{-1}{\rm e}^{-\tau\partial_x^3}(\tilde v-\tilde w)
\Big\rangle\\
&\quad
+
\Big\langle
\tilde z\,\partial_x^{-1}(\tilde v-\tilde w),
\partial_x^{-1}(\tilde v-\tilde w)
\Big\rangle
+
\Big\langle
\partial_x^{-1}\tilde z\,(\tilde v-\tilde w),
\partial_x^{-1}(\tilde v-\tilde w)
\Big\rangle.
\end{align*}
Accordingly, we split
$$
\frac{\mu}{3}J_1=J_1^a+J_1^b,
$$
where
\begin{align*}
J_1^a
&=-
\Big\langle
{\rm e}^{\tau\partial_x^3}
\Big[
\big({\rm e}^{-\tau\partial_x^3}\tilde z\big)
\big({\rm e}^{-\tau\partial_x^3}\partial_x^{-1}(\tilde v-\tilde w)\big)
\Big]
-
\tilde z\,\partial_x^{-1}(\tilde v-\tilde w),
\partial_x^{-1}(\tilde v-\tilde w)
\Big\rangle,\\
J_1^b
&=-
\Big\langle
{\rm e}^{\tau\partial_x^3}
\Big[
\big({\rm e}^{-\tau\partial_x^3}\partial_x^{-1}\tilde z\big)
\big({\rm e}^{-\tau\partial_x^3}(\tilde v-\tilde w)\big)
\Big]
-
\partial_x^{-1}\tilde z\,(\tilde v-\tilde w),
\partial_x^{-1}(\tilde v-\tilde w)
\Big\rangle.
\end{align*}

We start by estimating $J_1^a$. Let
$
\tilde y:=\partial_x^{-1}(\tilde v-\tilde w). 
$
Then, by definition,
$
J_1^a=-
\Big\langle
{\rm e}^{\tau\partial_x^3}
\Big[
\big({\rm e}^{-\tau\partial_x^3}\tilde z\big)
\big({\rm e}^{-\tau\partial_x^3}\tilde y\big)
\Big]
-
\tilde z\,\tilde y,
\tilde y
\Big\rangle.
$
Using the identity
$
\langle f,g\rangle=-\langle \partial_x^{-1}f,\partial_x g\rangle,
$
we may equivalently write
\begin{align*}
J_1^a
&=
\Big\langle
\partial_x^{-1}
\Big(
{\rm e}^{\tau\partial_x^3}
\Big[
\big({\rm e}^{-\tau\partial_x^3}\tilde z\big)
\big({\rm e}^{-\tau\partial_x^3}\tilde y\big)
\Big]
-
\tilde z\,\tilde y
\Big),
\tilde v-\tilde w
\Big\rangle.
\end{align*}
Since
$
\widehat{{\rm e}^{t\partial_x^3}f}_\ell
=
{\rm e}^{-it\ell^3}\hat f_\ell, 
$
the $\ell$-th Fourier coefficient of
$
{\rm e}^{\tau\partial_x^3}
\Big[
\big({\rm e}^{-\tau\partial_x^3}\tilde z\big)
\big({\rm e}^{-\tau\partial_x^3}\tilde y\big)
\Big]
$
is given by
$$
\displaystyle\frac{1}{\sqrt{2\pi}} {\rm e}^{-i\tau\ell^3} \sum_{\substack{\ell_1,\ell_2\ne 0\\ \ell_1+\ell_2=\ell}}
{\rm e}^{i\tau(\ell_1^3+\ell_2^3)}
\hat{\tilde z}_{\ell_1}\hat{\tilde y}_{\ell_2}.
$$
Hence, the $\ell$-th Fourier coefficient of
$
{\rm e}^{\tau\partial_x^3}
\Big[
\big({\rm e}^{-\tau\partial_x^3}\tilde z\big)
\big({\rm e}^{-\tau\partial_x^3}\tilde y\big)
\Big]
-
\tilde z\,\tilde y
$
is
$$
\displaystyle\frac{1}{\sqrt{2\pi}} \sum_{\substack{\ell_1,\ell_2\ne 0\\ \ell_1+\ell_2=\ell}}
\Big(
{\rm e}^{-i\tau(\ell^3-\ell_1^3-\ell_2^3)}-1
\Big)
\hat{\tilde z}_{\ell_1}\hat{\tilde y}_{\ell_2}.
$$
Using \eqref{alg_ide} and \eqref{inv_op}, we then obtain
$$
\partial_x^{-1} \Bigg({\rm e}^{\tau\partial_x^3}
\Big[
\big({\rm e}^{-\tau\partial_x^3}\tilde z\big)
\big({\rm e}^{-\tau\partial_x^3}\tilde y\big)
\Big]
-
\tilde z\,\tilde y\Bigg)=\displaystyle\frac{1}{\sqrt{2\pi}}\sum_{\ell\ne 0} {\rm e}^{i\ell x} (i\ell)^{-1}
\displaystyle\frac{1}{\sqrt{2\pi}} \sum_{\substack{\ell_1,\ell_2\ne 0\\ \ell_1+\ell_2=\ell}}
\Big(
{\rm e}^{-i\tau(\ell^3-\ell_1^3-\ell_2^3)}-1
\Big)
\hat{\tilde z}_{\ell_1}\hat{\tilde y}_{\ell_2}.
$$
Therefore, we get
\begin{align*}
J_1^a
&=\displaystyle\frac{1}{\sqrt{2\pi}} 
\sum_{\ell\ne 0}
(i\ell)^{-1}
(\hat{\tilde v}-\hat{\tilde w})_{-\ell}
\sum_{\substack{\ell_1,\ell_2\ne 0\\ \ell_1+\ell_2=\ell}}
\big(
{\rm e}^{-3i\tau\ell_1\ell_2(\ell_1+\ell_2)}-1
\big)
\hat{\tilde z}_{\ell_1}\hat{\tilde y}_{\ell_2}.
\end{align*}
Then, a standard interpolation argument yields the following bound
\begin{align*}
|J_1^a|
&\lesssim
\tau
\sum_{\substack{\ell_1,\ell_2\ne 0\\ \ell_1+\ell_2=\ell}}
|(\ell_1+\ell_2)^{-1}|
\,|(\hat{\tilde v}-\hat{\tilde w})_{-(\ell_1+\ell_2)}|
\,|\ell_1\ell_2(\ell_1+\ell_2)|
\,|\hat{\tilde z}_{\ell_1}|
\,|\hat{\tilde y}_{\ell_2}|\\
&\lesssim
\tau
\sum_{\ell_1,\ell_2\in\mathbbm{Z}}
\big|(\hat{\tilde v}-\hat{\tilde w})_{-(\ell_1+\ell_2)}\big|
\,|\ell_1\ell_2|
\,|\hat{\tilde z}_{\ell_1}|
\,|\hat{\tilde y}_{\ell_2}|\\
&= \tau \sum_{\ell,k\in\mathbbm{Z}}
|(\hat{\tilde v}-\hat{\tilde w})_{-\ell}|
\,|k(\ell-k)|
\,|\hat{\tilde z}_{k}|
\,|\hat{\tilde y}_{\ell-k}|.
\end{align*}
Applying the Cauchy--Schwarz inequality, we obtain
\begin{align*}
|J_1^a|
&\lesssim
\tau
\Big(
\sum_{\ell\in\mathbbm{Z}}|(\hat{\tilde v}-\hat{\tilde w})_{-\ell}|^2
\Big)^{1/2}
\Big(
\sum_{\ell\in\mathbbm{Z}}
\Big(
\sum_{k\in\mathbbm{Z}}
|k|\,|\ell-k|\,|\hat{\tilde z}_{k}|\,|\hat{\tilde y}_{\ell-k}|
\Big)^2
\Big)^{1/2}.
\end{align*}
The second factor is exactly the $\ell^2$-norm of the discrete convolution $\tilde z^{(1)}\star \tilde y^{(1)}$. Hence,
\begin{align*}
|J_1^a|
&\lesssim
\tau
\|\tilde v-\tilde w\|\,
\|\tilde z^{(1)}\star \tilde y^{(1)}\|_{\ell^2}.
\end{align*}
Using the standard correspondence between discrete convolution and pointwise multiplication in physical space, we get
$
\|\tilde z^{(1)}\star \tilde y^{(1)}\|_{\ell^2}
\lesssim
\|\tilde z^{(1)}\tilde y^{(1)}\|.
$
Therefore,
\begin{align}
|J_1^a|
&\lesssim
\tau
\|\tilde v-\tilde w\|\,
\|\tilde z^{(1)}\tilde y^{(1)}\|\lesssim
\tau
\|\tilde v-\tilde w\|\,
\|\tilde y^{(1)}\|\,
\|\tilde z^{(1)}\|_{L^\infty}.
\end{align}
Finally, since
$$
\tilde y=\partial_x^{-1}(\tilde v-\tilde w),
\qquad
\|\tilde y^{(1)}\|=\|\tilde v-\tilde w\|,
$$
and ${\rm e}^{-t\partial_x^3}$ is an isometry on $L^2$, we conclude that
\begin{equation}
\label{stab0}
|J_1^a|
\lesssim
\tau
\|v-w\|^2
\|z^{(1)}\|_{L^\infty}.
\end{equation}

We next estimate $J_1^b$. Since $\partial_x^{-1}$ commutes with ${\rm e}^{\tau\partial_x^3}$ and due to the property $
\langle {\rm e}^{t\partial_x^3}f,g\rangle=\langle f,{\rm e}^{-t\partial_x^3}g\rangle
$, we rewrite $J_1^b$ as
\begin{align*}
J_1^b
&= -
\Big\langle
({\rm e}^{-\tau\partial_x^3}\partial_x^{-1}\tilde z)
({\rm e}^{-\tau\partial_x^3}(\tilde v-\tilde w)),
{\rm e}^{-\tau\partial_x^3}\partial_x^{-1}(\tilde v-\tilde w)
\Big\rangle
-
\Big\langle
\partial_x^{-1}\tilde z\,(\tilde v-\tilde w),
\partial_x^{-1}(\tilde v-\tilde w)
\Big\rangle\\
&=-
\Big\langle
\partial_x^{-1}\tilde z,\,
{\rm e}^{\tau\partial_x^3}
\Big[
({\rm e}^{-\tau\partial_x^3}(\tilde v-\tilde w))
({\rm e}^{-\tau\partial_x^3}\partial_x^{-1}(\tilde v-\tilde w))
\Big]
\Big\rangle
-
\Big\langle
\partial_x^{-1}\tilde z,\,
(\tilde v-\tilde w)\partial_x^{-1}(\tilde v-\tilde w)
\Big\rangle.
\end{align*}
Using
$
f\,\partial_x^{-1}f=\frac12\partial_x\big(\partial_x^{-1}f\big)^2
$ 
with $f=\tilde v-\tilde w$ and $f={\rm e}^{-\tau\partial_x^3}(\tilde v-\tilde w)$, respectively, the  integration by parts formula yields that 
\begin{align*}
J_1^b
&=-
\frac12
\Big\langle
\partial_x^{-1}\tilde z,\,
{\rm e}^{\tau\partial_x^3}\partial_x
\Big(
{\rm e}^{-\tau\partial_x^3}\partial_x^{-1}(\tilde v-\tilde w)
\Big)^2
-
\partial_x
\Big(
\partial_x^{-1}(\tilde v-\tilde w)
\Big)^2
\Big\rangle\\
&=-
\frac12
\Big\langle
\tilde z,\,
\big(\partial_x^{-1}(\tilde v-\tilde w)\big)^2
-
{\rm e}^{\tau\partial_x^3}
\Big(
{\rm e}^{-\tau\partial_x^3}\partial_x^{-1}(\tilde v-\tilde w)
\Big)^2
\Big\rangle.
\end{align*}
Passing to Fourier coefficients and arguing as before, we obtain
\begin{align*}
J_1^b
&=-
\frac{1}{2\sqrt{2\pi}}
\sum_{\substack{\ell_1,\ell_2\ne 0}}
\hat{\tilde z}_{-(\ell_1+\ell_2)}
\frac{1}{(i\ell_1)(i\ell_2)}
\Big(
1-{\rm e}^{-i\tau\big((\ell_1+\ell_2)^3-\ell_1^3-\ell_2^3\big)}
\Big)
\hat{\tilde \delta}_{\ell_1}\hat{\tilde\delta}_{\ell_2},
\end{align*}
with the notation $\tilde\delta =\tilde v-\tilde w$. This infers
\begin{align*}
|J_1^b|
&\lesssim
\sum_{\substack{\ell_1,\ell_2\ne 0\\ \ell_1+\ell_2\ne 0}}
|\hat{\tilde z}_{-(\ell_1+\ell_2)}|
\frac{1}{|\ell_1\ell_2|}
\big|
1-{\rm e}^{-3i\tau\ell_1\ell_2(\ell_1+\ell_2)}
\big|
\,|\hat{\tilde\delta}_{\ell_1}|
\,|\hat{\tilde \delta}_{\ell_2}|\\
&\lesssim \tau
\sum_{\ell_1,\ell_2\in\mathbbm{Z}}
|(\ell_1+\ell_2)\hat{\tilde z}_{-(\ell_1+\ell_2)}|
\,|\hat{\tilde \delta}_{\ell_1}|
\,|\hat{\tilde \delta }_{\ell_2}|
\end{align*}
Similarly as done for the estimation of $J_1^a$, an application of the Cauchy--Schwarz inequality yields
\begin{align*}
|J_1^b|
&\lesssim
\tau
\Big(
\sum_{\ell\in\mathbbm{Z}}
\Big(
\sum_{k\in\mathbbm{Z}}
|\ell\,\hat{\tilde z}_{-\ell}|
\,|\hat{\tilde\delta}_{\ell-k}|
\Big)^2
\Big)^{1/2}
\Big(
\sum_{\ell\in\mathbbm{Z}}|\hat{\tilde\delta}_{\ell}|^2
\Big)^{1/2}\\
&\lesssim
\tau
\|\tilde z^{(1)}\star \tilde \delta^{(0)}\|_{\ell^2}
\,
\| \tilde\delta\|\\
&\lesssim \tau
\|\tilde z^{(1)}\|_{L^\infty}\,
\|\tilde \delta\|^2.
\end{align*}
Finally, since ${\rm e}^{-t\partial_x^3}$ is an isometry on $L^2$, we have
$
\|\tilde \delta\|=\|v-w\|,
$ and $
\|\tilde z^{(1)}\|_{L^\infty}=\|z^{(1)}\|_{L^\infty}.
$
Thus,
\begin{equation}
\label{stab1}
|J_1^b|
\lesssim
\tau
\|v-w\|^2
\|z^{(1)}\|_{L^\infty}.
\end{equation}

Combining \eqref{stab0} and \eqref{stab1}, we conclude that
\begin{equation}
\label{j1es}
J_1\lesssim \tau \|v-w\|^2\|z^{(1)}\|_{L^\infty}.
\end{equation}

We now estimate $J_2$. Set
$
a:=\partial_x^{-1}\tilde z,
$ and $b:=\partial_x^{-1}(\tilde v-\tilde w),
$ 
and define
$A:=
\big({\rm e}^{-\tau\partial_x^3}a\big)
\big({\rm e}^{-\tau\partial_x^3}b\big),
$ and $
B:=
{\rm e}^{-\tau\partial_x^3}(ab).$ 
Since $
J_2
=
\Big\|
{\rm e}^{\tau\partial_x^3}A-ab
\Big\|^2
=
\|A-B\|^2,$ we obtain
$$
J_2
=
\langle A-B,A\rangle-\langle A-B,B\rangle
=:J_2^a+J_2^b.
$$
where
\begin{align*}
J_2^a
&=
\Big\langle
\big({\rm e}^{-\tau\partial_x^3}\partial_x^{-1}\tilde z\big)
\big({\rm e}^{-\tau\partial_x^3}\partial_x^{-1}(\tilde v-\tilde w)\big)
-
{\rm e}^{-\tau\partial_x^3}
\big[\partial_x^{-1}\tilde z\,\partial_x^{-1}(\tilde v-\tilde w)\big],
\big({\rm e}^{-\tau\partial_x^3}\partial_x^{-1}\tilde z\big)
\big({\rm e}^{-\tau\partial_x^3}\partial_x^{-1}(\tilde v-\tilde w)\big)
\Big\rangle,\\
J_2^b
&=
-\Big\langle
\big({\rm e}^{-\tau\partial_x^3}\partial_x^{-1}\tilde z\big)
\big({\rm e}^{-\tau\partial_x^3}\partial_x^{-1}(\tilde v-\tilde w)\big)
-
{\rm e}^{-\tau\partial_x^3}
\big[\partial_x^{-1}\tilde z\,\partial_x^{-1}(\tilde v-\tilde w)\big],
{\rm e}^{-\tau\partial_x^3}
\big[\partial_x^{-1}\tilde z\,\partial_x^{-1}(\tilde v-\tilde w)\big]
\Big\rangle.
\end{align*}
We estimate $J_2^b$ first. Let
$
F:=
(\partial_x^{-1}\tilde z)(\partial_x^{-1}(\tilde v-\tilde w)).
$
As a result,
\begin{align*}
J_2^b
&=
-\Big\langle
\big({\rm e}^{-\tau\partial_x^3}\partial_x^{-1}\tilde z\big)
\big({\rm e}^{-\tau\partial_x^3}\partial_x^{-1}(\tilde v-\tilde w)\big)
-
{\rm e}^{-\tau\partial_x^3}F,
\,{\rm e}^{-\tau\partial_x^3}F
\Big\rangle.
\end{align*}
Using similar arguments as done for estimating the terms $J_1^a$ and $J_1^b$,
the $\ell$-th Fourier coefficient of
$
\big({\rm e}^{-\tau\partial_x^3}\partial_x^{-1}\tilde z\big)
\big({\rm e}^{-\tau\partial_x^3}\partial_x^{-1}(\tilde v-\tilde w)\big)
$
is given by
\begin{align*}
\displaystyle\frac{1}{\sqrt{2\pi}}\sum_{\substack{\ell_1,\ell_2\ne 0 \\ \ell_1+\ell_2=\ell}}
{\rm e}^{i\tau\ell_1^3}(i\ell_1)^{-1}\hat{\tilde z}_{\ell_1}\,
{\rm e}^{i\tau\ell_2^3}(i\ell_2)^{-1}
(\hat{\tilde v}_{\ell_2}-\hat{\tilde w}_{\ell_2}).
\end{align*}
Therefore, one gets
\begin{align*}
J_2^b
&=
-\sum_{\ell\in\mathbbm{Z}}
\Bigg[
\displaystyle\frac{1}{\sqrt{2\pi}}\sum_{\substack{\ell_1,\ell_2\ne 0\\ \ell_1+\ell_2=\ell}}
{\rm e}^{i\tau(\ell_1^3+\ell_2^3)}
\frac{1}{(i\ell_1)(i\ell_2)}
\hat{\tilde z}_{\ell_1}
(\hat{\tilde v}_{\ell_2}-\hat{\tilde w}_{\ell_2})
-
{\rm e}^{i\tau\ell^3}\hat F_\ell
\Bigg]
{\rm e}^{-i\tau\ell^3}\hat F_{-\ell}\\
&=
-\displaystyle\frac{1}{\sqrt{2\pi}}\sum_{\ell\in\mathbbm{Z}}
\sum_{\substack{\ell_1,\ell_2\ne 0\\ \ell_1+\ell_2=\ell}}
{\rm e}^{i\tau(\ell_1^3+\ell_2^3-\ell^3)}
\frac{1}{(i\ell_1)(i\ell_2)}
\hat{\tilde z}_{\ell_1}
(\hat{\tilde v}_{\ell_2}-\hat{\tilde w}_{\ell_2})
\hat F_{-\ell}
+\sum_{\ell\in\mathbbm{Z}}\hat F_\ell \hat F_{-\ell}.
\end{align*}
Noting that
$$
\hat F_\ell
=\displaystyle\frac{1}{\sqrt{2\pi}}
\sum_{\substack{\ell_1,\ell_2\ne 0\\ \ell_1+\ell_2=\ell}}
\frac{1}{(i\ell_1)(i\ell_2)}
\hat{\tilde z}_{\ell_1}
(\hat{\tilde v}_{\ell_2}-\hat{\tilde w}_{\ell_2}),
$$
the second term $\sum_{\ell\in\mathbbm{Z}}\hat F_\ell \hat F_{-\ell}$ can be rewritten as
\begin{align*}
\sum_{\ell\in\mathbbm{Z}}\hat F_\ell \hat F_{-\ell}
=
\displaystyle\frac{1}{\sqrt{2\pi}}
\sum_{\ell\in\mathbbm{Z}}
\sum_{\substack{\ell_1,\ell_2\ne 0\\ \ell_1+\ell_2=\ell}}
\frac{1}{(i\ell_1)(i\ell_2)}
\hat{\tilde z}_{\ell_1}
(\hat{\tilde v}_{\ell_2}-\hat{\tilde w}_{\ell_2})
\hat F_{-\ell}.
\end{align*}
Hence, using this, the relation in \eqref{alg_ide} and noting that the Fourier coefficient vanishes for $\ell_1+\ell_2=0$, we get
\begin{align*}
J_2^b
&=
-\displaystyle\frac{1}{\sqrt{2\pi}}\sum_{\ell\in\mathbbm{Z}}
\sum_{\substack{\ell_1,\ell_2\ne 0\\ \ell_1+\ell_2=\ell}}
\frac{1}{(i\ell_1)(i\ell_2)}
\Big(
{\rm e}^{i\tau(\ell_1^3+\ell_2^3-\ell^3)}-1
\Big)
\hat{\tilde z}_{\ell_1}
(\hat{\tilde v}_{\ell_2}-\hat{\tilde w}_{\ell_2})
\hat F_{-\ell}\\
&=
-\displaystyle\frac{1}{\sqrt{2\pi}}\sum_{\substack{\ell_1,\ell_2\ne 0\\ \ell_1+\ell_2\ne 0}}
\hat F_{-(\ell_1+\ell_2)}
\frac{1}{(i\ell_1)(i\ell_2)}
\Big(
{\rm e}^{-3i\tau\ell_1\ell_2(\ell_1+\ell_2)}
-1
\Big)
\hat{\tilde z}_{\ell_1}
(\hat{\tilde v}_{\ell_2}-\hat{\tilde w}_{\ell_2})
\end{align*}
Thus, in a similar manner as done for the estimation of $J_1$, we obtain
\begin{align*}
|J_2^b|
&\lesssim
\sum_{\substack{\ell_1,\ell_2\ne 0\\ \ell_1+\ell_2\ne 0}}
|\hat F_{-(\ell_1+\ell_2)}|
\frac{1}{|\ell_1\ell_2|}
\big|
1-{\rm e}^{-3i\tau\ell_1\ell_2(\ell_1+\ell_2)}
\big|
\,|\hat{\tilde z}_{\ell_1}|
\,|\hat{\tilde v}_{\ell_2}-\hat{\tilde w}_{\ell_2}|\\
&\lesssim
\tau
\sum_{\ell,k\in\mathbbm{Z}}
|\ell \hat F_{-\ell}|
\,|\hat{\tilde z}_k|
\,|\hat{\tilde v}_{\ell-k}-\hat{\tilde w}_{\ell-k}|\\
&\lesssim
\tau
\Big(
\sum_{\ell}|\ell \hat F_{-\ell}|^2
\Big)^{1/2}
\Big(
\sum_{\ell}
\Big(
\sum_k
|\hat{\tilde z}_k|
\,|\hat{\tilde v}_{\ell-k}-\hat{\tilde w}_{\ell-k}|
\Big)^2
\Big)^{1/2}\\
&\lesssim
\tau
\|F\|_{H^1}\,
\|\tilde z^{(0)}\star(\tilde v-\tilde w)^{(0)}\|_{\ell^2}.
\end{align*}
Using the standard convolution estimate in Fourier space, equivalently the product estimate in physical space, we infer that
\begin{align*}
\|\tilde z^{(0)}\star(\tilde v-\tilde w)^{(0)}\|_{\ell^2}
&\le C
\|\tilde z^{(0)}\|_{L^\infty}\,
\|\tilde v-\tilde w\|.
\end{align*}
Therefore,
\begin{align*}
|J_2^b|
&\lesssim
\tau
\|F\|_{H^1}\,
\|\tilde z^{(0)}\|_{L^\infty}\,
\|\tilde v-\tilde w\|.
\end{align*}
Finally, by the bilinear estimate \eqref{bil_est},
\begin{align*}
\|F\|_{H^1}
&=
\|(\partial_x^{-1}\tilde z)\big(\partial_x^{-1}(\tilde v-\tilde w)\big)\|_{H^1}\lesssim
\|\tilde z\|\,
\|\tilde v-\tilde w\|,
\end{align*}
and thus,
\begin{equation}
\label{j2b_est}
|J_2^b|
\lesssim
\tau
\|v-w\|^2
\|z^{(0)}\|_{L^\infty}\|z\|.
\end{equation}

We next estimate $J_2^a$. Set
$$
G:=
\big({\rm e}^{-\tau\partial_x^3}\partial_x^{-1}\tilde z\big)
\big({\rm e}^{-\tau\partial_x^3}\partial_x^{-1}(\tilde v-\tilde w)\big)
=A.
$$
Arguing exactly as the estimation of the term $J_2^b$, with $F$ replaced by $G$, we obtain
\begin{align*}
|J_2^a|
&\lesssim
\tau
\|G\|_{H^1}\,
\|\tilde z^{(0)}\star(\tilde v-\tilde w)^{(0)}\|_{\ell^2}.
\end{align*}
Since ${\rm e}^{-\tau\partial_x^3}$ is an isometry and $\partial_x^{-1}$ commutes with ${\rm e}^{-\tau\partial_x^3}$, the bilinear estimate \eqref{bil_est} yields
\begin{align*}
\|G\|_{H^1}
&=
\Big\|
\big({\rm e}^{-\tau\partial_x^3}\partial_x^{-1}\tilde z\big)
\big({\rm e}^{-\tau\partial_x^3}\partial_x^{-1}(\tilde v-\tilde w)\big)
\Big\|_{H^1}\lesssim
\|\tilde z\|\,
\|\tilde v-\tilde w\|.
\end{align*}
Therefore,
\begin{equation}
\label{j2a_est}
|J_2^a|
\lesssim
\tau
\|v-w\|^2
\|z^{(0)}\|_{L^\infty}\|z\|.
\end{equation}
Combining \eqref{j2b_est} and \eqref{j2a_est}, we conclude that
\begin{equation}
\label{j2es}
J_2
\lesssim
\tau
\|v-w\|^2
\|z^{(0)}\|_{L^\infty}\|z\|.
\end{equation}

Finally, combining \eqref{j1es} and \eqref{j2es}, also noting that the norms  $\|z^{(0)}\|_{L^\infty}$ and $\|z\|$ are controlled by the norm $\|z^{(1)}\|_{L^\infty}$, we obtain
$$
\|\Phi_t^\tau(v,z)-\Phi_t^\tau(w,z)\|^2
\le
\Big(
1+c\tau
\Big)\|v-w\|^2,
$$
where the constant $c>0$ depends on $\|z^{(1)}\|_{L^\infty}$. Taking expectations on both sides and then the square roots yield the desired result.
\end{proof}

We also need the following stability estimate for the map $\Phi_t^\tau$ with respect to its second argument.

\begin{lemma}
\label{det_conv_lem}
For any nonrandom $w,z\in L^2$, any $t,\tau>0$, and any $\mathcal F_t$-measurable random variable $v$ satisfying $v^{(0)}\in L^2(\Omega;L^\infty)$, one has
$$
\big\|\Phi_t^\tau(v,w)-\Phi_t^\tau(v,z)\big\|_{L^2(\Omega; L^2)}
\lesssim
\sqrt{\tau} \big\|w-z\big\| \big\|v^{(0)}\big\|_{L^2(\Omega; L^\infty)}.
$$
\end{lemma}

\begin{proof}By the same change of variables as in the proof of Lemma \ref{stab_lem}, we obtain
\begin{align*}
&\big\|\Phi_t^\tau(v,w)-\Phi_t^\tau(v,z)\big\|^2=
\frac{4\mu^2}{9}\,\langle \Xi,\Xi\rangle
=
\frac{4\mu^2}{9}\big(\xi_1+\xi_2\big),
\end{align*}
where
\begin{align*}
\Xi
&=
{\rm e}^{(t+\tau)\partial_x^3}
\Big(
{\rm e}^{-\tau\partial_x^3}\partial_x^{-1}(\tilde w-\tilde z)
\Big)
\Big(
{\rm e}^{-\tau\partial_x^3}\partial_x^{-1}\tilde v
\Big)
-
{\rm e}^{t\partial_x^3}
\Big(
\partial_x^{-1}(\tilde w-\tilde z)
\Big)
\Big(
\partial_x^{-1}\tilde v
\Big),
\end{align*}
and
\begin{align*}
\xi_1
&=
\Big\langle
\big({\rm e}^{-\tau\partial_x^3}\partial_x^{-1}(\tilde w-\tilde z)\big)
\big({\rm e}^{-\tau\partial_x^3}\partial_x^{-1}\tilde v\big)
-
{\rm e}^{-\tau\partial_x^3}
\big[
\partial_x^{-1}(\tilde w-\tilde z)\,\partial_x^{-1}\tilde v
\big],
\big({\rm e}^{-\tau\partial_x^3}\partial_x^{-1}(\tilde w-\tilde z)\big)
\big({\rm e}^{-\tau\partial_x^3}\partial_x^{-1}\tilde v\big)
\Big\rangle,\\
\xi_2
&=
-
\Big\langle
{\rm e}^{\tau\partial_x^3}
\Big[
\big({\rm e}^{-\tau\partial_x^3}\partial_x^{-1}(\tilde w-\tilde z)\big)
\big({\rm e}^{-\tau\partial_x^3}\partial_x^{-1}\tilde v\big)
\Big]
-
(\partial_x^{-1}(\tilde w-\tilde z))(\partial_x^{-1}\tilde v),
(\partial_x^{-1}(\tilde w-\tilde z))(\partial_x^{-1}\tilde v)
\Big\rangle.
\end{align*}
Let
$
F:=(\partial_x^{-1}(\tilde w-\tilde z))(\partial_x^{-1}\tilde v).
$
Arguing as in the proof of Lemma \ref{stab_lem} for estimating the term $J_1$ and $J_2$, and using interpolation together with H\"older's inequality, we obtain
\begin{align*}
|\xi_2|
&=\displaystyle\frac{1}{\sqrt{2\pi}}
\Big|
\sum_{\substack{\ell_1,\ell_2\ne0\\ \ell_1+\ell_2\ne 0}}
\hat F_{-(\ell_1+\ell_2)}
\frac{1}{\ell_1\ell_2}
\big(1-{\rm e}^{-3i\tau\ell_1\ell_2(\ell_1+\ell_2)}\big)
(\hat{\tilde w}_{\ell_1}-\hat{\tilde z}_{\ell_1})\hat{\tilde v}_{\ell_2}
\Big|\\
&\lesssim
\tau
\sum_{\ell_1,\ell_2\in\mathbbm{Z}}
|\hat F_{-(\ell_1+\ell_2)}|
\,|\ell_1+\ell_2|\,
|\hat{\tilde w}_{\ell_1}-\hat{\tilde z}_{\ell_1}|\,
|\hat{\tilde v}_{\ell_2}|\\
&\lesssim
\tau
\|F\|_{H^1}\,
\|(\tilde w-\tilde z)^{(0)}\star \tilde v^{(0)}\|_{\ell^2}\\
&\lesssim
\tau
\|F\|_{H^1}\,
\|w-z\|\,
\|v^{(0)}\|_{L^\infty}.
\end{align*}
Using the bilinear estimate \eqref{bil_est}, we deduce that
$$
|\xi_2|
\lesssim
\tau
\|(\partial_x^{-1}(w-z))(\partial_x^{-1}v)\|_{H^1}\,
\|w-z\|\,
\|v^{(0)}\|_{L^\infty}
\lesssim
\tau \|w-z\|^2\|v^{(0)}\|_{L^\infty}^2.
$$
A similar estimate holds for $\xi_1$. Therefore,
$$
\|\Phi_t^\tau(v,w)-\Phi_t^\tau(v,z)\|^2
\lesssim
\tau \|w-z\|^2\|v^{(0)}\|_{L^\infty}^2.
$$
Taking expectations on both sides and passing to square roots complete the proof.
\end{proof}

\subsection{Global error bound under $H^1$-regularity}

We now have all the ingredients needed to state and prove a convergence result for the low regularity scheme \eqref{met}. We  impose the following assumptions on the deterministic approximation of $\psi$.
\begin{ass}
    \label{ass_int_det}
    There exists $\sigma_0\ge 0$ such that if $\xi \in H^{\sigma_0}$ then there exists a sufficiently large $N_0\in \mathbbm{N}^+$ such that the solution $\psi_n$ of deterministic numerical integrator  satisfies the following integrability condition
    $$
   \sup_{N>N_0} \max_{0\le n\le N} \big\|\psi_n^{(1)}\big\|_{L^\infty}\lesssim 1,
    $$
    with the notation $\psi_n^{(1)}(x)=\frac{1}{\sqrt{2\pi}}\sum_{\ell\in\mathbbm{Z}} \big|\big(\widehat{\psi_n}\big)_{\ell}\big|\big|\ell \big|{\rm e}^{i \ell x}$, being $\psi_n(x)=\frac{1}{\sqrt{2\pi}}\sum_{\ell\in\mathbbm{Z}} \big(\widehat{\psi_n}\big)_{\ell}{\rm e}^{i \ell x}$.
\end{ass}

\begin{ass}
\label{ass_det_conv}
There exists $\sigma_1\ge 0$ such that if $\xi\in H^{\sigma_1}$ then the deterministic numerical solution $\psi_n$ satisfies the following convergence estimate
$$
\max_n \|\psi(t_n)-\psi_n\|\lesssim \tau^\beta,
$$
for some $\beta>1/2$, provided $\tau$ sufficiently small independently from $n$.
\end{ass}

\begin{remark}
Assumption \ref{ass_int_det} states an uniform integrability condition on $\psi_n$, under $H^{\sigma_0}$-regularity of the exact solutions. A sufficient condition for verifying Assumption \ref{ass_int_det} is the following
\begin{equation}
\label{ass_int_det_equiv}
\sup_{N >N_0}\max_{0\le n \le N}\sum_{\ell\in\mathbbm{Z}}\big| \big(\widehat{\psi_n}\big)_\ell \big| \big|\ell\big|<\infty,
\end{equation}
that is satisfied if, e.g., $\psi_n \in H^{3/2+}$ uniformly on $n$. In addition, note that such an assumption is only required in the stability estimate of the one-step map.

Finally, Assumption \ref{ass_det_conv} states convergence in $L^2$ with order $\beta$ for the deterministic integrator $\psi_n$, under $H^{\sigma_1}$-regularity of the exact solutions.
\end{remark}

Combining the local error estimate, the two stability bounds, Assumption \ref{ass_int_det}, and Assumption~\ref{ass_det_conv}, we obtain the first convergence result.

\begin{theorem}
\label{conv_thm}
Assume $Q^{1/2}\in\mathcal{L}_2^1$, Assumption \ref{ass_int_det}, and Assumption \ref{ass_det_conv}. Let $\xi \in H^{\max(2,\sigma_0,\sigma_1)}$.  Then, there exists $\tau_0>0$ such that the numerical solution $\chi_n$ in \eqref{met} satisfies
$$
\max_{0\le n\le N}
\|\chi(t_n)- \chi_n\|_{L^2(\Omega;L^2)}
\lesssim
\tau^{\min(\frac12,\beta-\frac12)},
$$
for all $\tau\in(0,\tau_0)$.
\end{theorem}

\begin{proof} We prove the result for the twisted variable $\tilde{\chi}$. The triangular inequality gives
\begin{align*}
    \big\|\tilde{\chi}(t_n)-\tilde \chi_n\big\|_{L^2(\Omega; L^2)}&\le \big\|\varphi_{t_{n-1}}^\tau(\tilde \chi(t_{n-1}))
-\Phi_{t_{n-1}}^\tau(\tilde \chi(t_{n-1}),\tilde\psi(t_{n-1}))\big\|_{L^2(\Omega; L^2)}\\
&\hspace{2.5cm}+\big\|\Phi_{t_{n-1}}^\tau(\tilde \chi(t_{n-1}),\tilde\psi(t_{n-1}))
-\Phi_{t_{n-1}}^\tau(\tilde \chi(t_{n-1}),\tilde\psi_{n-1})\big\|_{L^2(\Omega; L^2)}\\
&\hspace{2.5cm}+\big\|\Phi_{t_{n-1}}^\tau(\tilde \chi(t_{n-1}),\tilde\psi_{n-1})
-\Phi_{t_{n-1}}^\tau(\tilde \chi_{n-1},\tilde\psi_{n-1})\big\|_{L^2(\Omega; L^2)}.
\end{align*}
By Lemma \ref{Loc_err_lem}, Lemma \ref{stab_lem} and Lemma \ref{det_conv_lem}, we then have
\begin{align*}
\big\|\tilde{\chi}(t_n)-\tilde \chi_n\big\|_{L^2(\Omega; L^2)}&\le c \tau^{3/2}+c \tau^{1/2}\big\|\tilde\psi(t_{n-1})-\tilde\psi_{n-1}\big\|+{\rm e}^{c\tau/2}\big\|\tilde\chi(t_{n-1})-\tilde\chi_{n-1}\big\|_{L^2(\Omega; L^2)}\\
&\le c \tau^{\min(3/2, \beta+1/2)}+{\rm e}^{c\tau/2}\big\|\tilde\chi(t_{n-1})-\tilde\chi_{n-1}\big\|_{L^2(\Omega; L^2)},
\end{align*}
where the constant $c$ is independent from $n$ and $\tau$. A recursion on $n$ then completes the proof. 
\end{proof}

\begin{remark}
As a consequence of Theorem \ref{conv_thm}, convergence of order one-half is achieved for the linearized equation provided that
$$
\xi\in H^{\max(2,\sigma_0,\sigma_1)},\qquad
Q^{1/2}\in\mathcal{L}_2^1,
$$
where $\sigma_0,\sigma_1$ come from Assumption \ref{ass_int_det} and Assumption \ref{ass_det_conv}, respectively, with the requirement that $\sigma_1$ is such that the numerical integrator for the deterministic KdV equation converges with order one in $L^2$.
\end{remark}

With Theorem \ref{conv_thm} at hand, we can state the following Corollary, establishing the first strong convergence result of this paper, for the numerical approximation $u_n=\psi_n+\varepsilon \chi_n$ for the solution $u(t_n)$ to the stochastic KdV equation \eqref{kdv_eq}.

\begin{corollary}
\label{cor1}
Under the assumptions of Theorem \ref{conv_thm}, one has
$$
\max_n \big\|u(t_n)-u_n\big\|_{L^2(\Omega;L^2)}
\lesssim
\tau^{\min(1,\beta,2\beta-1)}+\varepsilon^2.
$$

\end{corollary}
\begin{proof}
The triangle inequality gives
\begin{align*}
\|u(t_n)-u_n\|_{L^2(\Omega;L^2)}
&\le
\|\psi(t_n)-\psi_n\|
+\varepsilon \|u^\varepsilon(t_n)-\chi_n\|_{L^2(\Omega;L^2)}\\
&\le
\|\psi(t_n)-\psi_n\|
+\varepsilon \|u^\varepsilon(t_n)-\chi(t_n)\|_{L^2(\Omega;L^2)}+\varepsilon \|\chi(t_n)-\chi_n\|_{L^2(\Omega;L^2)}.
\end{align*}
The conclusion follows directly from Assumption \ref{ass_det_conv}, Proposition \ref{prop_wp_lin}, Theorem \ref{conv_thm} and by an application of the Young's inequality.
\end{proof}

\subsection{Improved order of convergence under $H^2$-regularity}
The preceding argument gives at most order $1/2$ for the fluctuation under its $H^1$ regularity assumptions. This is a bound from the present analysis, not an optimality assertion. In this section, we will show how the convergence rate can be improved if additional conditions are imposed. The key ingredient is to exploit the exact representation of the squared $L^2(\Omega, L^2)$ error, conditional expectations and martingale properties of stochastic integrals.

We start by providing the following regularity result, that will be used for obtaining the improved estimate. Its proof follows from arguments similar to those used for Lemma~\ref{lem_H1b}, and is therefore omitted.

\begin{lemma}
\label{reg_plus}
Assume that $Q^{1/2}\in\mathcal{L}_2^2$, and $\xi\in H^3$. Then, $\chi\in\mathcal{L}_T^{\infty,p,2}$ for every $p\in\mathbbm N^+$.
\end{lemma}

The next lemma is a key ingredient to achieve improved convergence rates for the numerical solution $\chi_n$. It exploits the additional regularity of the exact stochastic solution stated in Lemma \ref{reg_plus}.
\begin{lemma}
\label{loc_err2}
    Let us assume the assumptions of Lemma \ref{reg_plus}. Then, for any $n=0,1,\dots,N-1$, we have the following error estimate
    $$
    \Big\|\mathbbm{E}\Big[\varphi_{t_n}^\tau\big(\tilde\chi(t_n)\big)-\Phi_{t_n}^\tau\big(\tilde\chi(t_n),\tilde\psi(t_n)\big)\Big|\mathcal{F}_{t_n}\Big]\Big\|_{L^2(\Omega; L^2)} \lesssim \tau^2.
    $$
\end{lemma}

\begin{proof}
    Comparing \eqref{mild_sol_twis_stoc} and \eqref{num_map}, and using the martingale property of the stochastic It\^o integral, we obtain
\begin{align*}
&
\mathbbm{E}
\Big[
\varphi_{t_n}^\tau(\tilde \chi(t_n))-\Phi_{t_n}^\tau(\tilde \chi(t_n),\tilde\psi(t_n))
\Bigm|
\mathcal F_{t_n}
\Big]
\\
&\qquad=
2\mu
\mathbbm{E}
\Bigg[
\int_0^\tau
{\rm e}^{(t_n+s)\partial_x^3}\partial_x
\Big(
{\rm e}^{-(t_n+s)\partial_x^3}\tilde\psi(t_n+s)\,
{\rm e}^{-(t_n+s)\partial_x^3}\tilde \chi(t_n+s)\\
&\hspace{4.8cm}
-
{\rm e}^{-(t_n+s)\partial_x^3}\tilde\psi(t_n)\,
{\rm e}^{-(t_n+s)\partial_x^3}\tilde \chi(t_n)
\Big)\,ds
\Bigm|
\mathcal F_{t_n}
\Bigg],
\end{align*}
so that $\Big\|\mathbbm{E}\Big[\varphi_{t_n}^\tau\big(\tilde\chi(t_n)\big)-\Phi_{t_n}^\tau\big(\tilde\chi(t_n),\tilde\psi(t_n)\big)\Big|\mathcal{F}_{t_n}\Big]\Big\|_{L^2(\Omega; L^2)}\lesssim \big\|I^a\big\|_{L^2(\Omega; L^2)}+\big\|I^b\big\|_{L^2(\Omega; L^2)}$,
where 
\begin{align*}
I^a
&=
\mathbbm{E}
\Bigg[
\int_0^\tau
{\rm e}^{(t_n+s)\partial_x^3}\partial_x
\Big(
{\rm e}^{-(t_n+s)\partial_x^3}
\big(\tilde\psi(t_n+s)-\tilde\psi(t_n)\big)\,
{\rm e}^{-(t_n+s)\partial_x^3}\tilde \chi(t_n+s)
\Big)\,ds
\Bigm|
\mathcal F_{t_n}
\Bigg],\\
I^b
&=
\mathbbm{E}
\Bigg[
\int_0^\tau
{\rm e}^{(t_n+s)\partial_x^3}\partial_x
\Big(
{\rm e}^{-(t_n+s)\partial_x^3}\tilde\psi(t_n)\,
{\rm e}^{-(t_n+s)\partial_x^3}
\big(\tilde \chi(t_n+s)-\tilde \chi(t_n)\big)
\Big)\,ds
\Bigm|
\mathcal F_{t_n}
\Bigg].
\end{align*}
For the term $I^a$, using the bilinear estimate \eqref{bil_est_hom}, and the result from \cite{kat0} 
$$
\|\tilde \psi(t_n+s)-\tilde \psi(t_n)\|_{\dot{H}^1}\lesssim s,
$$
we have
\begin{align}
\big\|I^a\big\|
&\le
\mathbbm{E}\Bigg[
\int_0^\tau
\Big\|
{\rm e}^{-(t_n+s)\partial_x^3}
\big(\tilde\psi(t_n+s)-\tilde\psi(t_n)\big)\,
{\rm e}^{-(t_n+s)\partial_x^3}\tilde \chi(t_n+s)
\Big\|_{\dot{H}^1}\,ds
\Bigm|
\mathcal F_{t_n}
\Bigg]
\notag\\
&\lesssim
\mathbbm{E}\Bigg[
\int_0^\tau
\|\tilde\psi(t_n+s)-\tilde\psi(t_n)\|_{\dot{H}^1}\,
\|\tilde \chi(t_n+s)\|_{\dot{H}^1}\,ds
\Bigm|
\mathcal F_{t_n}
\Bigg]
\label{Gam0}\\
&\lesssim
\mathbbm{E}\Bigg[
\int_0^\tau
s\,\|\tilde \chi(t_n+s)\|_{\dot{H}^1}\,ds
\Bigm|
\mathcal F_{t_n}
\Bigg].
\notag
\end{align}
Then, by repeated applications of Hölder's inequality, we get
\begin{align*}
\big\|I^a\big\|_{L^2(\Omega; L^2)}
&\lesssim
\tau
\Bigg(
\mathbbm{E}
\Big|
\mathbbm{E}\Big[
\int_0^\tau \|\tilde \chi(t_n+s)\|_{\dot{H}^1}\,ds
\Bigm|
\mathcal F_{t_n}
\Big]
\Big|^2
\Bigg)^{1/2}\\
&\lesssim
\tau
\Bigg(
\mathbbm{E}\Big[
\mathbbm{E}\Big[
\Big(\int_0^\tau \|\tilde \chi(t_n+s)\|_{\dot{H}^1}\,ds\Big)^2
\Bigm|
\mathcal F_{t_n}
\Big]
\Big]
\Bigg)^{1/2}\\
&\lesssim
\tau^{3/2}
\Big(
\mathbbm{E}\int_0^\tau \|\tilde \chi(t_n+s)\|_{\dot{H}^1}^2\,ds
\Big)^{1/2}.
\end{align*}
By assumptions, $\tilde \chi \in \mathcal{L}_T^{\infty,p,2}$ and then
\begin{equation}
\label{Ia}
\big\|I^a\big\|_{L^2(\Omega; L^2)}\lesssim \tau^2.
\end{equation}
For the term $I^{b}$, we use the mild form \eqref{mild_sol_twis_stoc} for $\tilde \chi$ and again exploit the martingale property of the stochastic It\^o integral. This gives
\begin{align*}
\big\|I^b\big\|
&=
\Bigg\|
\mathbbm{E}
\Bigg[
\int_0^\tau
{\rm e}^{(t_n+s)\partial_x^3}\partial_x
\Big(
{\rm e}^{-(t_n+s)\partial_x^3}\tilde\psi(t_n)\,
{\rm e}^{-(t_n+s)\partial_x^3}\\
&\qquad\qquad\qquad\qquad\times
\int_0^s
2\mu\,{\rm e}^{(t_n+r)\partial_x^3}\partial_x
\big(
{\rm e}^{-(t_n+r)\partial_x^3}\tilde\psi(t_n+r)\,
{\rm e}^{-(t_n+r)\partial_x^3}\tilde \chi(t_n+r)
\big)\,dr
\Big)\,ds
\Bigm|
\mathcal F_{t_n}
\Bigg]
\Bigg\|\\
 &\lesssim \mathbbm{E}\Big[\int_{0}^{\tau} \big\|{\rm e}^{- (t_n+s)\partial_x^3} \tilde\psi(t_n){\rm e}^{-(t_n+s)\partial_x^3}\\
    &\qquad \qquad \times \int_{0}^{s}{\rm e}^{(t_n+r)\partial_x^3} \partial_x \big({\rm e}^{-(t_n+r)\partial_x^3 }\tilde\psi(t_n+r){\rm e}^{-(t_n+r)\partial_x^3 }\tilde{\chi}(t_n+r)\big)\,dr \big\|_{\dot{H}^1}\,ds \big|\mathcal{F}_{t_n}\Big].
\end{align*}
Using the bilinear estimate \eqref{bil_est_hom} and noting that by our assumptions $\tilde\psi\in\mathbb{L}_T^{\infty,3}$, we deduce
\begin{align*}
\big\|I^b\big\|
&\lesssim
\mathbbm{E}\Bigg[
\int_0^\tau\int_0^s
\Big\|
{\rm e}^{-(t_n+r)\partial_x^3}\tilde\psi(t_n+r)\,
{\rm e}^{-(t_n+r)\partial_x^3}\tilde \chi(t_n+r)
\Big\|_{\dot{H}^2}\,dr\,ds
\Bigm|
\mathcal F_{t_n}
\Bigg]
\notag\\
&\lesssim
\mathbbm{E}\Bigg[
\int_0^\tau\int_0^s
\|\tilde\psi(t_n+r)\|_{H^2}\,
\|\tilde \chi(t_n+r)\|_{\dot{H}^2}\,dr\,ds
\Bigm|
\mathcal F_{t_n}
\Bigg]
\notag\\
&\lesssim
\mathbbm{E}\Bigg[
\int_0^\tau\int_0^s
\|\tilde \chi(t_n+r)\|_{\dot{H}^2}\,dr\,ds
\Bigm|
\mathcal F_{t_n}
\Bigg].
\end{align*}
Furthermore, arguing exactly as above, using the fact that, by Lemma \ref{reg_plus}, $\tilde\chi \in \mathcal{L}_T^{\infty,p,2}$, for every $p\in\mathbbm{N}^+$, we obtain
\begin{align*}
&\big\|I^b\big\|^2_{L^2(\Omega; L^2)}
\lesssim
\mathbbm{E}\Bigg[
\Big(
\int_0^\tau\int_0^s \|\tilde \chi(t_n+r)\|_{\dot{H}^2}\,dr\,ds
\Big)^2
\Bigg]
\notag\\
&\lesssim
\tau^2
\int_0^\tau\int_0^s
\mathbbm{E}\|\tilde \chi(t_n+r)\|_{\dot{H}^2}^2\,dr\,ds
\lesssim \tau^4.
\end{align*}
This gives
\begin{equation}
\label{Ib}
\big\|I^b\big\|_{L^2(\Omega; L^2)}\lesssim \tau^2.
\end{equation}
Combining the bounds \eqref{Ia} and \eqref{Ib}, we complete the proof.
\end{proof}

To obtain an improved convergence rate for $\chi_n$, we impose the following stronger assumption on the deterministic solver $\psi_n$.

\begin{ass}
\label{ass_det_conv2}
There exists $\sigma_2\ge 1$ such that if $\xi\in H^{\sigma_2}$ then the deterministic numerical solution $\psi_n$ satisfies the following convergence estimate
$$
\max_n \|\psi(t_n)-\psi_n\|_{H^1}\lesssim \tau^\alpha,
$$
for some $\alpha>0$, provided $\tau$ sufficiently small independently from $n$.
\end{ass}

The following lemma is a crucial result for obtaining the improved convergence result. It exploits the additional assumption given by Assumption \ref{ass_det_conv2}.
\begin{lemma}
\label{lem_I6}
    Let us assume Assumption \ref{ass_det_conv2}. Moreover, let us take $\xi \in H^{\max(3,\sigma_2)}$ and $Q^{1/2}\in\mathcal{L}_2^2$. Then, for any $n=0,\dots,N-1$, one has the following estimate
    \begin{align*}
    \mathbbm{E}\Big\langle
\Phi_{t_n}^\tau(\tilde\chi(t_n),\tilde\psi(t_n))
-
\Phi_{t_n}^\tau(\tilde\chi(t_n),\tilde\psi_n),
\tilde\chi(t_n)-\tilde\chi_n
\Big\rangle\lesssim \tau^{1+\alpha}\big\|\tilde\chi(t_n)-\tilde\chi_n\big\|_{L^2(\Omega; L^2)}.
     \end{align*}
\end{lemma}
\begin{proof}
    We exploit the resonance structure of the numerical integrator  through Fourier expansion.  Note that, twisting back the variables, we have
$$
\chi(t_n)={\rm e}^{-t_n\partial_x^3}\tilde \chi(t_n),\qquad
\chi_n={\rm e}^{-t_n\partial_x^3}\tilde \chi_n,
\qquad
\psi(t_n)={\rm e}^{-t_n\partial_x^3}\tilde\psi(t_n),\qquad
\psi_n={\rm e}^{-t_n\partial_x^3}\tilde\psi_n.
$$
By straightforward algebraic manipulations, we obtain
\begin{align*}
&\Big\langle
\Phi_{t_n}^\tau(\tilde \chi(t_n),\tilde\psi(t_n))
-
\Phi_{t_n}^\tau(\tilde \chi(t_n),\tilde\psi_n),
\tilde \chi(t_n)-\tilde \chi_n
\Big\rangle=-
\frac{2\mu}{3}(A+B),
\end{align*}
where 
\begin{align*}
A
&=
\Big\langle
{\rm e}^{\tau\partial_x^3}
\Big[
({\rm e}^{-\tau\partial_x^3}(\psi(t_n)-\psi_n))
({\rm e}^{-\tau\partial_x^3}\partial_x^{-1}\chi(t_n))
\Big]
-
(\psi(t_n)-\psi_n)\,\partial_x^{-1}\chi(t_n),
\partial_x^{-1}(\chi(t_n)-\chi_n)
\Big\rangle,\\
B
&=
\Big\langle
{\rm e}^{\tau\partial_x^3}
\Big[
({\rm e}^{-\tau\partial_x^3}\partial_x^{-1}(\psi(t_n)-\psi_n))
({\rm e}^{-\tau\partial_x^3}\chi(t_n))
\Big]
-
\partial_x^{-1}(\psi(t_n)-\psi_n)\,\chi(t_n),
\partial_x^{-1}(\chi(t_n)-\chi_n)
\Big\rangle.
\end{align*}
We first estimate $B$. Let
\begin{equation}
\label{tr0}
\delta:=\chi(t_n)-\chi_n,\qquad
w:=\partial_x^{-1}(\psi(t_n)-\psi_n),\qquad
z:=\chi(t_n).
\end{equation}
Using the same Fourier expansion argument as above, we obtain
\begin{align*}
    \big|B\big|&=\frac{1}{\sqrt{2\pi}}\Big|\sum_{\ell\ne 0}\hat{\delta}_{-\ell}\ell^{-1}\sum_{\substack{\ell_1\ne 0,\ell_2\ne 0\\ \ell_1+\ell_2=\ell}} \big(1-{\rm e}^{-3i\tau\ell_1 \ell_2(\ell_1+\ell_2)}\big)\hat{w}_{\ell_1}\hat{z}_{\ell_2}\Big| \\
    &=\frac{1}{\sqrt{2\pi}}\Big|\sum_{\substack{\ell_1,\ell_2\ne 0\\ \ell_1+\ell_2\ne 0}}\hat{\delta}_{-(\ell_1+\ell_2)}(\ell_1+\ell_2)^{-1} \big(1-{\rm e}^{-3i\tau\ell_1 \ell_2(\ell_1+\ell_2)}\big)\hat{w}_{\ell_1}\hat{z}_{\ell_2}\Big| 
\end{align*}
A standard interpolation argument and Hölder inequality then give
\begin{align*}
     \big|B\big| & \lesssim \tau \sum_{\substack{\ell_1,\ell_2\ne 0\\ \ell_1+\ell_2\ne 0}}\big|\hat{\delta}_{-(\ell_1+\ell_2)}\big| \big| \ell_1 \ell_2\big| \big|\hat{w}_{\ell_1} \big| \big|\hat{z}_{\ell_2} \big|\\
     &\lesssim \tau \Big(\sum_{\ell\in\mathbbm{Z}} \big|\hat{\delta}_{\ell} \big|^2\Big)^{1/2}
     \Big(\sum_{\ell\in\mathbbm{Z}}\Big| \sum_{\ell_1+\ell_2=\ell}\big|\ell_1\big| \big|\ell_2\big|\big|\hat{w}_{\ell_1}\big|
     \big|\hat{z}_{\ell_2}\big|\Big|^2\Big)^{1/2}\\
     &\lesssim\tau \big\|\delta\big\| \big\| w^{(1)} z^{(1)}\big\|\\
     &\lesssim \tau \big\|\delta\big\| \big\|w\big\|_{\dot{H}^1} \big\|z^{(1)}\big\|_{L^\infty}.
\end{align*}
In view of \eqref{tr0}, we get
\begin{align*}
\big| B \big|&\lesssim\tau \big\|\tilde{\chi}(t_n)-\tilde\chi_n\big\| \big\|\tilde\psi(t_n)-\tilde\psi_n\big\| \big\|\tilde\chi^{(1)}(t_n)\big\|_{L^\infty}.
\end{align*}

Similarly, taking
$$
w:=\psi(t_n)-\psi_n,\qquad
z:=\partial_x^{-1}\chi(t_n),
$$
we obtain for the term $A$
\begin{align*}
|A|
&\lesssim
\tau
\|\tilde \chi(t_n)-\tilde \chi_n\|\,
\|\tilde\psi(t_n)-\tilde\psi_n\|_{\dot{H}^1}\,
\|\tilde \chi(t_n)\|_{L^\infty}.
\end{align*}
Hence, taking into account Assumption \ref{ass_det_conv2}, Hölder's inequality and the triangular inequality, we end up with
\begin{align*}
    \mathbbm{E}\Big\langle
\Phi_{t_n}^\tau(\tilde\chi(t_n),\tilde\psi(t_n))
-
\Phi_{t_n}^\tau(\tilde\chi(t_n),\tilde\psi_n),
\tilde\chi(t_n)-\tilde\chi_n
\Big\rangle&\lesssim \mathbbm{E}\Big[\big|A\big|+\big|B\big|\Big]\\
&\lesssim \tau^{1+\alpha} \mathbbm{E}\Big[\big\|\tilde \chi(t_n)-\tilde \chi_n\big\|\big(\|\tilde \chi^{(1)}(t_n)\|_{L^\infty}+\|\tilde \chi(t_n)\|_{L^\infty}\big)\Big]\\
&\lesssim \tau^{1+\alpha} \big\|\tilde\chi(t_n)-\tilde\chi_n\big\|_{L^2(\Omega; L^2)}\\
&\hspace{2cm}\times \big(\|\tilde \chi^{(1)}(t_n)\|_{L^2(\Omega; L^\infty)}+\|\tilde \chi(t_n)\|_{L^2(\Omega; L^\infty)}\big).
\end{align*}
By assumptions, $\tilde \chi \in \mathcal{L}_T^{\infty,p,2}$, for any $p\in\mathbbm{N}^+$. This implesi $\|\tilde \chi^{(1)}(t_n)\|_{L^2(\Omega; L^\infty)}+\|\tilde \chi(t_n)\|_{L^2(\Omega; L^\infty)}\lesssim 1$. The proof is thus concluded.
\end{proof}

We now state and prove the improved convergence result.

\begin{theorem}
\label{conv_thm2}
Let us assume Assumption \ref{ass_int_det}, Assumption \ref{ass_det_conv} and Assumption \ref{ass_det_conv2}. Moreover, let us take $\xi \in H^{\max(3,\sigma_0,\sigma_1,\sigma_2)}$ and $Q^{1/2}\in\mathcal{L}_2^2$. Then, there exists a sufficiently small $\tau_0>0$ such that we have the following convergence resultn for the numerical solution $\chi_n$
$$
\max_n \|\chi(t_n)-\chi_n\|_{L^2(\Omega;L^2)}
\lesssim \tau^{\min(1,\alpha)},
$$
for all $\tau\in(0,\tau_0)$.
\end{theorem}

\begin{proof}
    We first prove the result for the twisted variable $\tilde \chi$. Set
$$
e_n:=\tilde \chi(t_n)-\tilde \chi_n.
$$
Using the exact one-step map and the numerical one-step map, we write
\begin{align*}
\mathbbm{E}\|e_{n+1}\|^2
=&
\mathbbm{E}\big\|
\varphi_{t_n}^\tau(\tilde \chi(t_n))
-\Phi_{t_n}^\tau(\tilde \chi_n,\tilde\psi_n)
\big\|^2\\
=&
\mathbbm{E}\Big\|
\big(\varphi_{t_n}^\tau(\tilde \chi(t_n))
-\Phi_{t_n}^\tau(\tilde \chi(t_n),\tilde \psi(t_n)\big)
+\big(\Phi_{t_n}^\tau(\tilde \chi(t_n),\tilde\psi(t_n))
-\Phi_{t_n}^\tau(\tilde \chi(t_n),\tilde\psi_n)\big)\\
&
+\big(\Phi_{t_n}^\tau(\tilde \chi(t_n),\tilde\psi_n)
-\Phi_{t_n}^\tau(\tilde \chi_n,\tilde\psi_n)\big)
\Big\|^2.
\end{align*}
For convenience, define
\begin{align*}
A_n
&:=
\varphi_{t_n}^\tau(\tilde \chi(t_n))
-\Phi_{t_n}^\tau(\tilde \chi(t_n),\tilde \psi(t_n)),\\
B_n
&:=
\Phi_{t_n}^\tau(\tilde \chi(t_n),\tilde\psi(t_n))
-\Phi_{t_n}^\tau(\tilde \chi(t_n),\tilde\psi_n),\\
C_n
&:=\Phi_{t_n}^\tau(\tilde \chi(t_n),\tilde\psi_n)
-\Phi_{t_n}^\tau(\tilde \chi_n,\tilde\psi_n).
\end{align*}
Then
$
e_{n+1}=A_n+B_n+C_n,
$
and therefore
\begin{align*}
\mathbbm{E}\|e_{n+1}\|^2
&=I_1+I_2+I_3+I_4+I_5+I_6.
\end{align*}
where 
\begin{align*}
&I_1:=\mathbbm{E}\|A_n\|^2, \quad 
I_2:=\mathbbm{E}\|B_n\|^2, \quad 
I_3:=\mathbbm{E}\|C_n\|^2,\\
&I_4:=2\mathbbm{E}\langle A_n,B_n\rangle, \quad 
I_5:=2\mathbbm{E}\langle A_n,C_n\rangle,\quad 
I_6:=2\mathbbm{E}\langle B_n,C_n\rangle.
\end{align*}

By Lemma  \ref{Loc_err_lem}, we have
$
I_1\lesssim \tau^3.
$
Next, by Lemma \ref{det_conv_lem} and Assumption \ref{ass_det_conv},
\begin{align*}
I_2
&\lesssim
\tau \|\tilde\psi(t_n)-\tilde\psi_n\|^2\,
\mathbbm{E}\|\tilde \chi^{(0)}(t_n)\|_{L^\infty}^2\lesssim
\tau^{2\beta+1}\,
\mathbbm{E}\|\tilde \chi^{(0)}(t_n)\|_{L^\infty}^2.
\end{align*}
Noting that, by assumptions, the bounds in \eqref{reg_n0} hold true and that $\|\tilde\chi^{(0)}\|_{L^2(\Omega; L^\infty)}\lesssim \|\tilde\chi\|_{L^2(\Omega; H^1)}$, we get $I_2\lesssim \tau^{2\beta+1}$. 

Moreover, Lemma \ref{stab_lem} yields
$$
I_3\le (1+c\tau)\mathbbm{E}\|e_n\|^2.
$$
For $I_4$, by the Cauchy--Schwarz inequality and Young's inequality,
\begin{align*}
|I_4|
&=
2\big|\mathbbm{E}\langle A_n,B_n\rangle\big|
\le
2\big(\mathbbm{E}\|A_n\|^2\big)^{1/2}\big(\mathbbm{E}\|B_n\|^2\big)^{1/2}=
2 I_1^{1/2}I_2^{1/2}\\
&\le
I_1+I_2
\lesssim
\tau^{1+2\min(1,\beta)}.
\end{align*}
 For the term $I_5$, we note that, by standard properties of conditional expectation, we have
\begin{align*}
I_5
&=
2\mathbbm{E}\Big[
\mathbbm{E}\Big[
\Big\langle
\varphi_{t_n}^\tau(\tilde \chi(t_n))-\Phi_{t_n}^\tau(\tilde \chi(t_n),\tilde\psi(t_n)),
\Phi_{t_n}^\tau(\tilde \chi(t_n),\tilde\psi_n)-\Phi_{t_n}^\tau(\tilde \chi_n,\tilde\psi_n)
\Big\rangle
\Bigm|
\mathcal F_{t_n}
\Big]
\Big].
\end{align*}
Denoting
$
\Psi_t^\tau(v,z):=\Phi_t^\tau(v,z)-v,
$
we decompose $I_5=I_5^1+I_5^2$, where
\begin{align*}
I_5^1
&=
2\mathbbm{E}\Big[
\mathbbm{E}\Big[
\Big\langle
\varphi_{t_n}^\tau(\tilde \chi(t_n))-\Phi_{t_n}^\tau(\tilde \chi(t_n),\tilde\psi(t_n)),
\tilde \chi(t_n)-\tilde \chi_n
\Big\rangle
\Bigm|
\mathcal F_{t_n}
\Big]
\Big],\\
I_5^2
&=
2\mathbbm{E}\Big[
\mathbbm{E}\Big[
\Big\langle
\varphi_{t_n}^\tau(\tilde \chi(t_n))-\Phi_{t_n}^\tau(\tilde \chi(t_n),\tilde\psi(t_n)),
\Psi_{t_n}^\tau(\tilde \chi(t_n),\tilde\psi_n)-\Psi_{t_n}^\tau(\tilde \chi_n,\tilde\psi_n)
\Big\rangle
\Bigm|
\mathcal F_{t_n}
\Big]
\Big].
\end{align*}
The term $I_5^2$ is estimated directly by the Cauchy--Schwarz inequality, Lemma \ref{Loc_err_lem}, and the same argument as in the estimation of terms $J_1$ and $J_2$ in the proof of Lemma \ref{stab_lem} and finally the Young's inequality. Indeed, we deduce
\begin{align}
I_5^2
&\lesssim
\Big(
\mathbbm{E}\|\varphi_{t_n}^\tau(\tilde \chi(t_n))-\Phi_{t_n}^\tau(\tilde \chi(t_n),\tilde\psi(t_n))\|^2
\Big)^{1/2}\times
\Big(
\mathbbm{E}\|\Psi_{t_n}^\tau(\tilde \chi(t_n),\tilde\psi_n)-\Psi_{t_n}^\tau(\tilde \chi_n,\tilde\psi_n)\|^2
\Big)^{1/2}
\notag\\
&\lesssim
\tau^{3/2}\cdot \tau^{1/2}
\big(\mathbbm{E}\|\tilde \chi(t_n)-\tilde \chi_n\|^2\big)^{1/2}
\lesssim
\tau^3+\tau\mathbbm{E}\big\|e_{n}\big\|^2.
\label{T2}
\end{align}
For the term $I_5^1$, using that $\tilde\chi(t_n)$ and $\tilde \chi_n$ are $\mathcal{F}_{t_n}$-measurable, using the Cauchy-Schwartz inequality, Hölder's inequality, Lemma \ref{loc_err2}, and the Young's inequality, one gets
\begin{align}
    I_5^1&=2\mathbbm{E}\Big[\left\langle \mathbbm{E}\Big[\varphi_{t_n}^\tau(\tilde \chi(t_n))-\Phi_{t_n}^\tau(\tilde \chi(t_n),\tilde\psi(t_n))\Big|\mathcal{F}_{t_n}\Big], \tilde\chi(t_n)-\tilde\chi_n\right\rangle\Big]\notag \\
    &\lesssim \mathbbm{E}\Big[\Big\| \mathbbm{E}\Big[\varphi_{t_n}^\tau(\tilde \chi(t_n))-\Phi_{t_n}^\tau(\tilde \chi(t_n),\tilde\psi(t_n))\Big|\mathcal{F}_{t_n}\Big]\Big\| \big\| \tilde\chi(t_n)-\tilde\chi_n\big\|\Big]\notag \\
    &\lesssim \tau^2 \big\| \tilde\chi(t_n)-\tilde\chi_n\big\|_{L^2(\Omega; L^2)}\lesssim \tau^3 +\tau \mathbbm{E}\big\|e_n\big\|^2. \label{T2_2}
\end{align}
By \eqref{T2} and \eqref{T2_2}, one then obtains
$$
I_5\lesssim \tau^3 +\tau \mathbbm{E}\big\|e_n\big\|^2.
$$
It remains to bound the term $I_6$. Denoting
$$
I_6^a:=2\mathbbm{E}\Big\langle
\Phi_{t_n}^\tau(\tilde\chi(t_n),\tilde\psi(t_n))
-
\Phi_{t_n}^\tau(\tilde\chi(t_n),\tilde\psi_n),
\Psi_{t_n}^\tau(\tilde\chi(t_n),\tilde\psi_n)-\Psi_{t_n}^\tau(\tilde\chi_n,\tilde\psi_n)
\Big\rangle,
$$
using Lemma \ref{lem_I6} and the Young's inequality, we get
\begin{align*}
    I_6&\lesssim\tau^{1+\alpha}\big\|\tilde\chi(t_n)-\tilde\chi_n\big\|_{L^2(\Omega; L^2)}+I_6^a\\
    &\lesssim \tau^{1+2\alpha}+\tau\mathbbm{E}\big\|e_n\big\|^2+I_6^a,
\end{align*}
so that, we have to bound the term $I_6^a$. The Cauchy--Schwarz inequality followed by H\"older's inequality gives
$$
I_6^a\lesssim \big\|\Phi_{t_n}^\tau(\tilde\chi(t_n),\tilde\psi(t_n))-\Phi_{t_n}^\tau(\tilde\chi(t_n),\tilde\psi_n)\big\|_{L^2(\Omega; L^2)} \big\|\Psi_{t_n}^\tau(\tilde\chi(t_n),\tilde\psi_n)-\Psi_{t_n}^\tau(\tilde\chi_n,\tilde\psi_n)\big\|_{L^2(\Omega; L^2)}.
$$
By Lemma \ref{det_conv_lem} and our assumptions, we have
$$
\big\|\Phi_{t_n}^\tau(\tilde\chi(t_n),\tilde\psi(t_n))-\Phi_{t_n}^\tau(\tilde\chi(t_n),\tilde\psi_n)\big\|_{L^2(\Omega; L^2)}\lesssim\sqrt{\tau}\big\| \tilde\psi(t_n)-\tilde\psi_n\big\| \big\|\tilde\chi^{(0)}(t_n)\big\|_{L^2(\Omega; L^\infty)}\\
\lesssim \tau^{1/2+\beta}.
$$
Moreover, by the proof of Lemma \ref{stab_lem} and by assumptions, we deduce
$$
\big\|\Psi_{t_n}^\tau(\tilde\chi(t_n),\tilde\psi_n)-\Psi_{t_n}^\tau(\tilde\chi_n,\tilde\psi_n)\big\|_{L^2(\Omega; L^2)}\lesssim \sqrt{\tau} \big\|\tilde\chi(t_n)-\tilde\chi_n\big\|_{L^2(\Omega; L^2)}.
$$
Then, combining the above two estimates and by the Young's inequality, we get
$$
I_6^a\lesssim \tau^{1+\beta}\big\|\tilde\chi(t_n)-\tilde\chi_n\big\|_{L^2(\Omega; L^2)}\lesssim \tau \mathbbm{E}\big\|e_n\big\|^2+\tau^{1+2\beta}.
$$
This finally gives
$$
I_6\lesssim  \tau \mathbbm{E}\big\|e_n\big\|^2+\tau^{1+2\min(\alpha,\beta)}.
$$
Thus, collecting all the estimates for $I_1,I_2,I_3,I_4,I_5$, and $I_6$, and noting that $H^1\subset L^2$, we arrive at
$$
\mathbbm{E}\big\|e_{n+1}\big\|^2\le (1+c\tau)\mathbbm{E}\big\|e_n\big\|^2+c\tau^{1+2\min(1,\alpha)}.
$$
A standard discrete Gronwall argument  yields
\begin{equation}
\label{es1}
\max_{0\le n\le N}\mathbbm{E}\|e_n\|^2
\lesssim
\tau^{2\min(\alpha,1)}.
\end{equation}
Taking square roots, we conclude that
\begin{equation}
\label{es2}
\max_{0\le n\le N}
\big\|\tilde \chi(t_n)-\tilde \chi_n\big\|_{L^2(\Omega;L^2)}
\lesssim
\tau^{\min(1,\alpha)}.
\end{equation}
Since the Airy flow is an isometry on $L^2$, transforming back to the original variable gives the asserted estimate for $\chi_n$.
\end{proof}

We conclude with the following corollary, whose proof is omitted since it follows in the same way as the proof of Corollary \ref{cor1}.

\begin{corollary}
\label{cor2}
Under the assumptions of Theorem \ref{conv_thm2}, the approximation $u_n=\psi_n+\varepsilon \chi_n$ to the stochastic KdV equation \eqref{kdv_eq} satisfies
$$
\max_n \|u(t_n)-u_n\|_{L^2(\Omega;L^2)}
\lesssim
\tau^{\min(2,2\alpha,\beta)}+\varepsilon^2.
$$
\end{corollary}

Corollaries~\ref{cor1} and \ref{cor2} give the two strong error estimates stated in the introduction. In particular, under $H^1$-regularity, the total error is of size $\mathcal{O}(\tau+\varepsilon^2)$, provided that the numerical integrator is first-order convergent in $L^2$. Under $H^2$-regularity, the total error bound improves to $\mathcal{O}(\tau^2+\varepsilon^2)$, provided that the deterministic numerical integrator is first-order convergent in $H^1$ and second-order convergent in $L^2$.
This confirms that the proposed stochastic numerical integrator can show small errors, in the small-noise regime, even at low regularity, specially if compared to other classical stochastic integrators. This will be further pointed out in the next section.

We conclude this section by providing an example of a deterministic numerical integrator, for $\mu=1/2$, satisfying Assumption \ref{ass_int_det}, Assumption \ref{ass_det_conv} and Assumption \ref{ass_det_conv2}. It comes from \cite{kat0}, where the authors provided the following numerical integrator
\begin{equation}
\label{kat_ex0}
    \psi_n={\rm e}^{-\tau\partial_x^3}\psi_{n-1}+\frac{1}{6}\big({\rm e}^{-\tau\partial_x^3}\partial_x^{-1}\psi_{n-1}\big)^2-\frac{1}{6}{\rm e}^{-\tau\partial_x^3} (\partial_x^{-1}\psi_{n-1})^2.
\end{equation}
As stated in \cite{kat0}, this integrator satisfies \eqref{zero_det}. Also, it has been proved that it converges in $H^1$ with order one for exact solutions in $H^3$. Moreover, under the same regularity, it converges in $H^2$ with order $1/2$. This implies that the numerical integrator \eqref{kat_ex0} satisfies the aforementioned assumptions with $\sigma_0=\sigma_1=\sigma_2=3$ and with $\beta=\alpha=1$. Thus Corollary~\ref{cor2} gives $C(\tau+\varepsilon^2)$ for this particular solver. The sharper $C(\tau^2+\varepsilon^2)$ bound needs another deterministic solver with $\beta\ge2$ and $\alpha\ge1$, together with the same Fourier stability assumption. The Crank-Nicolson scheme, e.g., is a candidate under sufficiently smooth data, as highlighted in the following section.

\section{Numerical experiments}

\label{exp_section}

In this section, we present several numerical experiments to illustrate all the theoretical results. 

\subsection{Numerical setting} 

We here introduce the setting used throughout the experiments. In view of Assumption \ref{ass_zero_mode} and Remark \ref{simul}, the covariance operator is chosen diagonal in the Fourier basis and does not excite the zero mode:
$$
q_0=0, \qquad q_{\ell}=|\ell|^{-q}, \qquad \ell \in\mathbbm{Z}_{\ne 0},
$$
where $q>0$ controls the spatial regularity of the noise. Note that, for $\sigma \ge 0$, using $Q{\rm e}^{i\ell x}=q_{\ell}{\rm e}^{i\ell x}$, we have
$$
\|Q^{1/2}\|^2_{\mathcal{L}_2^\sigma}=
\frac{1}{\pi}\sum_{\ell\in\mathbbm{N}}\|Q^{1/2}\cos(\ell x)\|^2_{H^\sigma}
+\frac{1}{\pi}\sum_{\ell\in\mathbbm{N}}\|Q^{1/2}\sin(\ell x)\|^2_{H^\sigma}.
$$
Using Euler formulas for $\cos(\ell x)$ and $\sin(\ell x)$. this becomes
$$
\begin{aligned}
\|Q^{1/2}\|^2_{\mathcal{L}_2^\sigma}
&=\frac{1}{4\pi}\sum_{\ell\in\mathbbm{N}}\|Q^{1/2}({\rm e}^{i\ell x}+{\rm e}^{-i\ell x})\|^2_{H^\sigma}
+\frac{1}{4\pi}\sum_{\ell\in\mathbbm{N}}\|Q^{1/2}({\rm e}^{i\ell x}-{\rm e}^{-i\ell x})\|^2_{H^\sigma}.
\end{aligned}
$$
Therefore,
$$
\|Q^{1/2}\|^2_{\mathcal{L}_2^\sigma}
\lesssim
\frac{1}{4\pi}\sum_{\ell\in\mathbbm{Z}_{\ne 0}}\|Q^{1/2}{\rm e}^{i\ell x}\|^2_{H^\sigma}
\lesssim
\frac{1}{4\pi}\sum_{\ell\in\mathbbm{Z}_{\ne 0}}q_{\ell}\,|\ell|^{2\sigma}.
$$
Since $q_\ell=|\ell|^{-q}$, we further obtain
$$
\|Q^{1/2}\|^2_{\mathcal{L}_2^\sigma}
\lesssim
\frac{1}{4\pi}\sum_{\ell\in\mathbbm{N}}\ell^{2\sigma-q},
$$
which is finite provided that $q>2\sigma+1$. 

We give the initial datum $\xi$ through its Fourier coefficients.
Since $\xi$ is real-valued, it is enough to specify the coefficients for positive modes. We take
\begin{equation}
    \label{det_coeff0}
\hat\xi_0=0, \qquad \hat\xi_\ell=\ell^{-\rho}, \qquad \ell \in \mathbbm{N}, 
\end{equation}
for some $\rho>0$. It is straightforward to check that $\xi\in H^s$, provided $\rho>s+1/2$. 
Also, unless otherwise specified, in \eqref{kdv_eq}, we always take $\mu=1/2$. 
In the spatial direction, we adopt a pseudo-spectral method with $\kappa \in \mathbbm{N}^+$ Fourier modes. More precisely, we consider the spatial grid points
$$
x_j=-\pi+j \frac{2\pi}{\kappa}, \qquad j=1,\dots,\kappa,
$$
and retain Fourier modes in the range $\ell\in[-\kappa/2+1,\dots,\kappa/2]$. In all our experiments, we will take $\kappa=2^{12}$, we will compute the errors at the final time $T=1$ and we will approximate the expectations using $M=100$ sample paths.

\subsection{Errors for the linearized equation} 

 We here verify the convergence behavior predicted by Theorems~\ref{conv_thm} and \ref{conv_thm2}. We take $\xi\in H^3$, obtained taking $\rho=7/2+1e-10$. In this case, the regularity of the linearized stochastic solution depends only on the regularity of the noise term. Our integrator for \eqref{lin_eq} reads
\begin{equation}
\label{stoc_ex1}
\chi_{n+1}={\rm e}^{-\tau \partial_x^3}\chi_n+\frac{1}{3}\big({\rm e}^{-\tau\partial_x^3}\partial_x^{-1}\psi_n\big)\big({\rm e}^{-\tau\partial_x^3}\partial_x^{-1}\chi_{n}\big)+\Delta \mathcal{W}^n-\frac{1}{3}{\rm e}^{-\tau\partial_x^3}\big(\partial_x^{-1}\psi_n\big)\big(\partial_x^{-1}\chi_{n}\big),
\end{equation}
where
$$
\Delta\mathcal{W}^n:=\int_{t_n}^{t_{n+1}}{\rm e}^{-(t_{n+1}-s)\partial_x^3} Q^{1/2}\,dW(s).
$$
In the diagonal covariance setting above, this stochastic convolution can be simulated exactly mode by mode. The deterministic component $\psi_n$ is computed using \eqref{kat_ex0}. 
First, we take $Q^{1/2}\in\mathcal L_2^1$, obtained taking $q=3+1e-10$, so that $\chi\in\mathcal L^{\infty,p,1}$ for every $p\in\mathbbm N$.

From \cite{kat0}, we get $\beta=1$. Consequently, Theorem~5.6 predicts order \(1/2\) convergence for the stochastic approximation (\eqref{stoc_ex1}) in this case.

Second, we take $Q^{1/2}\in\mathcal L_2^2$, obtained taking $q=5+1e-10$. Then $\chi\in\mathcal L^{\infty,p,2}$ for every $p\in\mathbbm N$. According to \cite{kat0}, Assumption~\ref{ass_det_conv2} is satisfied with $\alpha=1$, and then Theorem~\ref{conv_thm2} predicts first-order convergence of \eqref{stoc_ex1}. The panel of Figure~\ref{fig1} shows this first order rate of convergence, while a better rate than $1/2$ is observed for the case $Q^{1/2}\in\mathcal{L}_2^1$, i.e., for the $H^1$-regularity.

In this figure, i.e., Figure~\ref{fig1}, the reference solutions are computed with the same methods using $\tau_{\rm ref}=2^{-16}$. The tested time steps are $\tau_j=2^{-j}$, $j=2,\dots,8$.

\subsection{Approximation of the stochastic KdV equation \eqref{kdv_eq}}

We next test the full approximation $u_n=\psi_n+\varepsilon\chi_n$ for the stochastic KdV equation \eqref{kdv_eq}. We start by testing the following stochastic low regularity scheme, denoted by SLR:
\begin{align}
u_{n+1}={}&{\rm e}^{-\tau\partial_x^3} u_n
+\frac{1}{6}\big({\rm e}^{-\tau\partial_x^3}\partial_x^{-1}\psi_n\big)^2
-\frac{1}{6}{\rm e}^{-\tau\partial_x^3}\big(\partial_x^{-1}\psi_n\big)^2 \notag\\
&+\frac{\varepsilon}{3}\big({\rm e}^{-\tau\partial_x^3}\partial_x^{-1}\psi_n\big)
\big({\rm e}^{-\tau\partial_x^3}\partial_x^{-1}\chi_{n}\big)
-\frac{\varepsilon}{3}{\rm e}^{-\tau\partial_x^3}
\big(\partial_x^{-1}{\psi_n}\big)
\big(\partial_x^{-1}\chi_{n}\big)
+\varepsilon\Delta \mathcal{W}^n,
\label{full_met}
\end{align}
obtained combining the first-order deterministic method \eqref{kat_ex0} and the stochastic method \eqref{stoc_ex1}.

\begin{figure}
\centering
\subfigure{\includegraphics[width=0.48\textwidth]{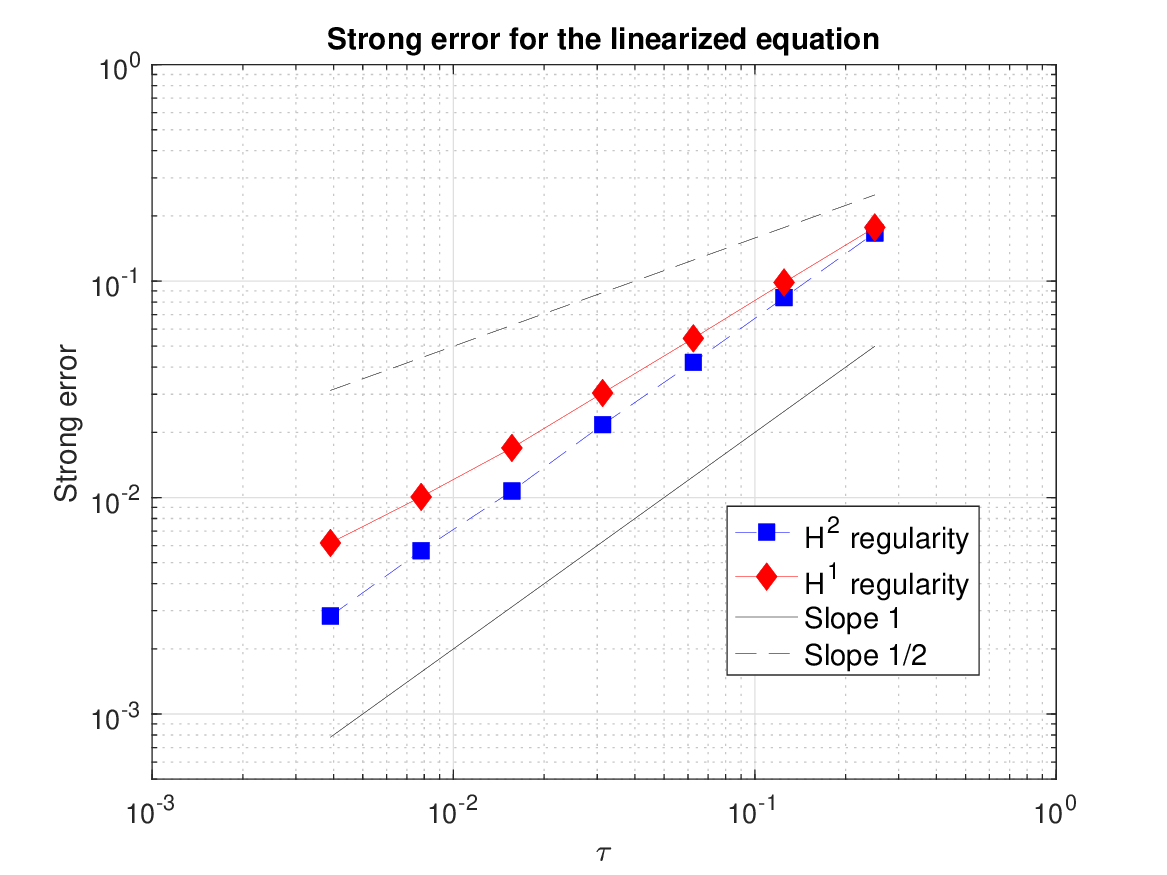} }
          \caption{$L^2(\Omega;L^2)$ errors for \eqref{stoc_ex1} in the $H^1$ and $H^2$-regularity regimes. The panel is plotted in log-log scale.}
           \label{fig1}
\end{figure}

For comparison, we also implement the Crank--Nicolson (CN) method studied in \cite{deb3, li_new1}:
\begin{equation}
\label{crank}
u_{n+1}=u_n-\frac{1}{2}\tau \partial_x^3 \big(u_{n+1}+u_n\big)
+\frac{1}{8}\tau \partial_x \big((u_{n+1}+u_n)^2\big)
+\varepsilon Q^{1/2}\big(W(t_n+\tau)-W(t_n)\big).
\end{equation}
The reference solution is computed by the exponential low regularity scheme
\begin{equation}
\label{ref_ex1}
u_{n+1}={\rm e}^{-\tau\partial_x^3}u_{n}
+\frac{1}{6}\big({\rm e}^{-\tau\partial_x^3}\partial_x^{-1}u_{n}\big)^2
-\frac{1}{6}{\rm e}^{-\tau\partial_x^3}\big(\partial_x^{-1}u_{n}\big)^2
+\varepsilon\Delta\mathcal{W}^n,
\end{equation}
obtained as a small-noise stochastic perturbation of the method \eqref{kat_ex0}.

Figure~\ref{fig3} reports the $L^2(\Omega;L^2)$ errors for the SLR and CN methods, with $\xi \in H^3$ and $Q^{1/2}\in \mathcal{L}_2^1$, for $\varepsilon=0.05$ (on the left) and $\varepsilon=0.01$ (on the right). The reference solution \eqref{ref_ex1} uses $\tau_{\rm ref}=2^{-16}$, while the tested time steps are $\tau_j=2^{-j}$, $j=2,\dots,8$. In both panels, we observe a rate one for our SLR integrator as predicted by Corollary \ref{cor1}.

\begin{figure}[htbp]
\centering
\subfigure{\includegraphics[width=0.48\textwidth]{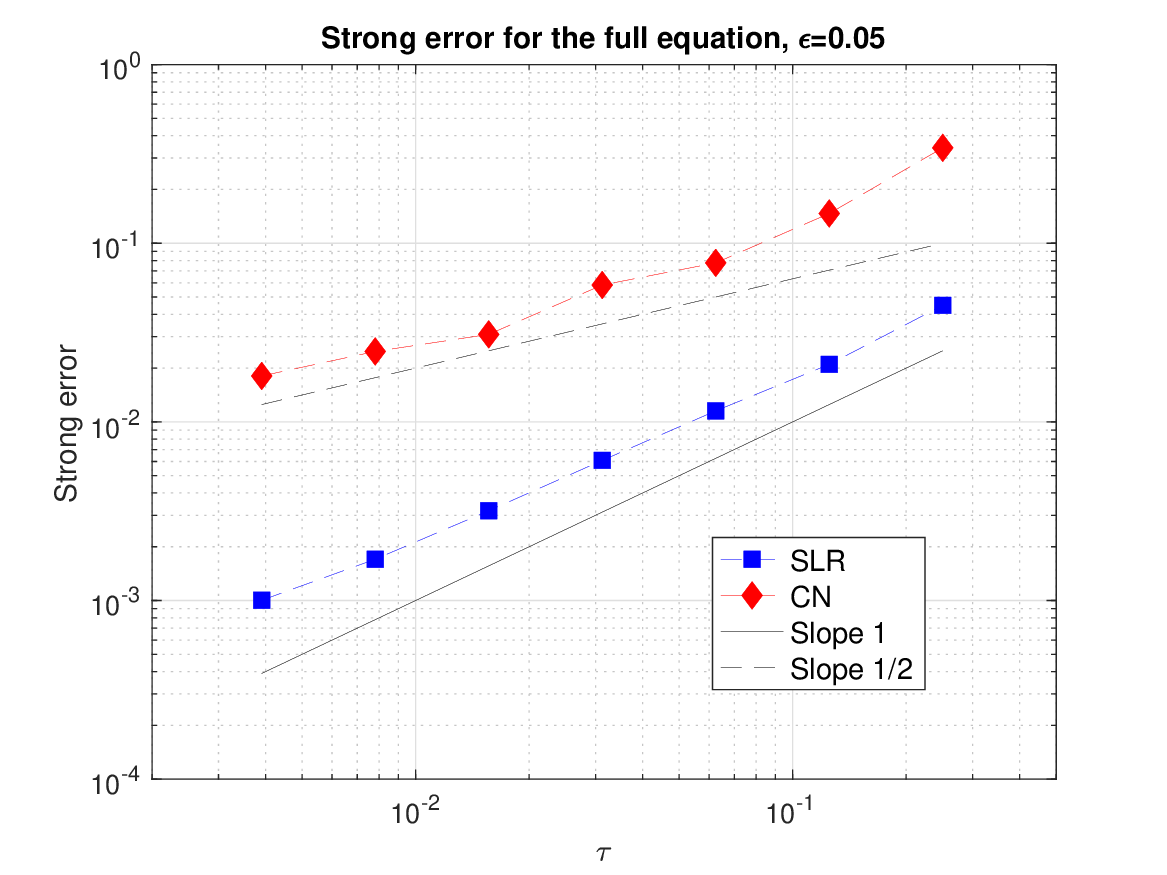}} \quad
\subfigure{\includegraphics[width=0.48\textwidth]{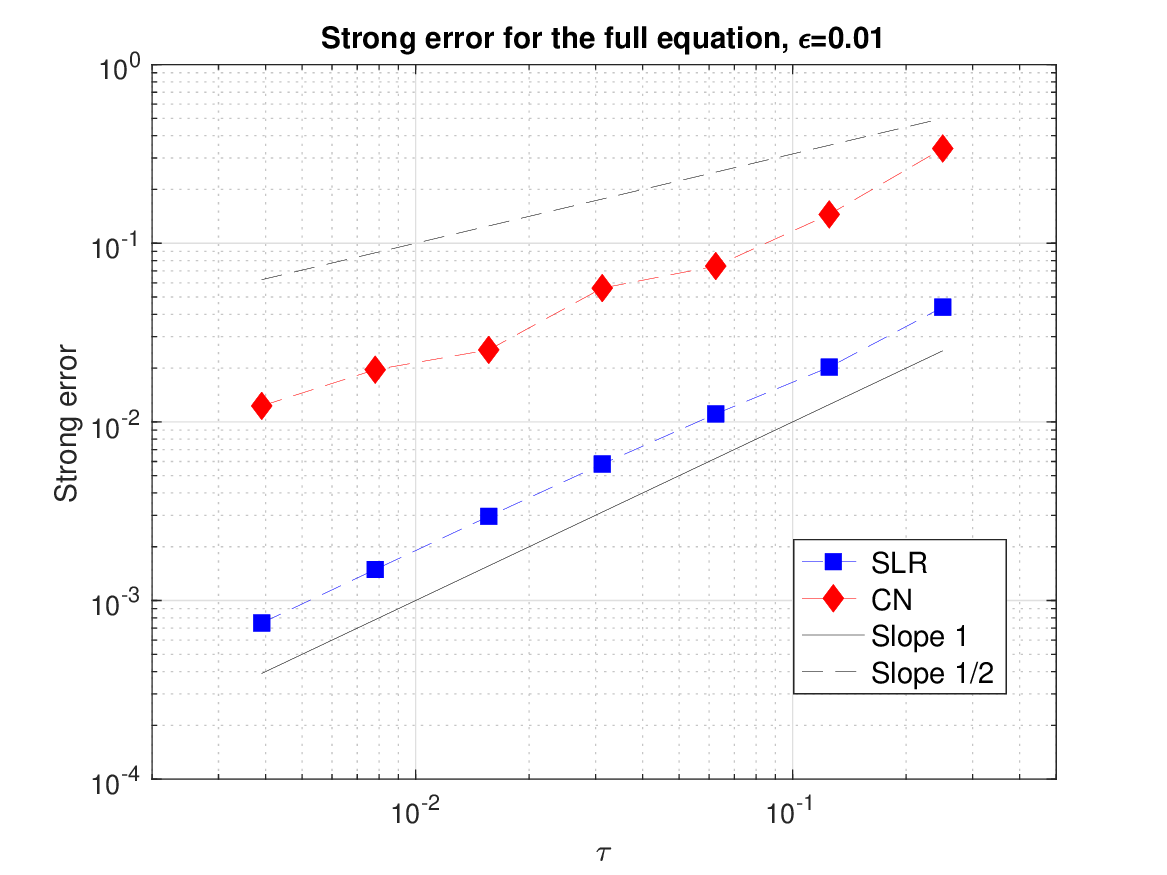}} 
          \caption{$L^2(\Omega;L^2)$ errors at $T=1$ for the SLR method \eqref{full_met} and the CN method \eqref{crank}, under $H^1$-regularity, for $\varepsilon=0.05$ (on the left) and $\varepsilon=0.01$ (on the right). Both panels are plotted in log-log scale.}
           \label{fig3}
\end{figure}

We conclude this subsection by testing the higher order convergence result stated in Corollary \ref{cor2}. To this purpose, we take smooth initial data, i.e., we take $\xi \in H^9$, obtained taking $\rho=19/2+1e-10$, and $Q^{1/2}\in\mathcal{L}_2^2$, obtained taking $q=5+1e-10$. These conditions yield $H^2$-regularity for the stochastic solutions. For $n=0,1\dots,N-1$, we consider the following stochastic low regularity integrator, named SLR2, 
\begin{equation}
\begin{cases}
  u_{n+1}=\psi_{n+1}+\varepsilon\chi_{n+1},\\
    \chi_{n+1} = {\rm e}^{-\tau\partial_x^3} \chi_n 
 + \frac{1}{3}\big({\rm e}^{-\tau\partial_x^3}\partial_x^{-1}\psi_n\big)
\big({\rm e}^{-\tau\partial_x^3}\partial_x^{-1}\chi_{n}\big)
-\frac{1}{3}{\rm e}^{-\tau\partial_x^3}
\big(\partial_x^{-1}{\psi_n}\big)
\big(\partial_x^{-1}\chi_{n}\big)
+\Delta \mathcal{W}^n,\\
  \psi_{n+1}=\psi_{n}-\frac{1}{2}\tau \partial_x^3 \big(\psi_{n+1}+\psi_{n}\big)
+\frac{1}{8}\tau \partial_x \big((\psi_{n+1}+\psi_{n})^2\big),
\end{cases}
\end{equation}
i.e., the deterministic solution is approximated by the standard Crank-Nicolson scheme. Since the deterministic solution is smooth, it is classical to have that such deterministic scheme admits second-order of convergence. Then, the assumptions of Corollary \ref{cor2} are satisfied with $\beta=\alpha=2$. In Figure \ref{fig_n}, we report the $L^2(\Omega; L^2)$ errors for the SLR2 method, SLR method and CN method for $\varepsilon=0.01$. As stated by Corollary \ref{cor2}, we observe the second order of convergence for the SLR2 method up to a small plateau of size  about $1e-5$. For the SLR method, first order of convergence is confirmed, as stated in Corollary \ref{cor1}. For the CN method, a reduction of the rate in $1/2$ is also observed. This also confirms that, despite the fact that $\varepsilon=0.01$, we are actually observing the stochastic effects in the KdV equation. Note that all these plots confirms the superiority of the proposed low regularity schemes respect to the classical CN method when applied to these low regularity regimes. In particular, in Figure \ref{fig_n}, one can observe that the errors of the SLR2 method remains smaller than those shown by the CN method also when the plateau is achieved and also for very small step sizes (less than $1e-4$).

\begin{figure}[htbp]
\centering
\includegraphics[width=0.48\textwidth]{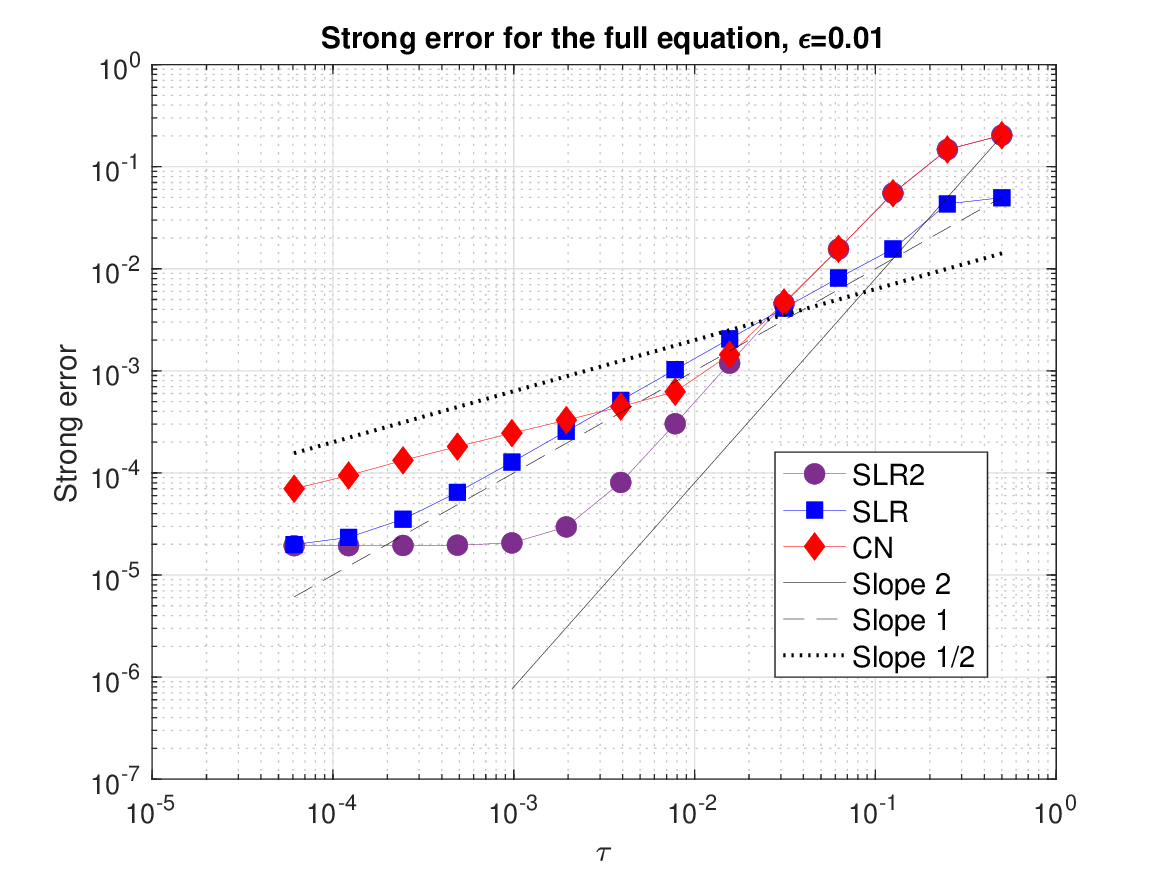}
          \caption{$L^2(\Omega;L^2)$ errors at $T=1$ for the SLR2 method, the SLR method and the CN method, under $H^2$-regularity, for $\varepsilon=0.01$. The panel is plotted in log-log scale.}
           \label{fig_n}
\end{figure}

\subsection{Small-noise linearization test}

The final experiment verifies the perturbative estimate between the rescaled stochastic fluctuation $u^\varepsilon$ and its linearization $\chi$, under $H^1$-regularity. Recall that
\[
u^\varepsilon(T)=\frac{1}{\varepsilon}\big(u(T)-\psi(T)\big).
\]
{According to Proposition~\ref{prop_wp_lin}, the error $\|u^\varepsilon(T)-\chi(T)\|_{L^2(\Omega;L^2)}$ should be of order $\varepsilon$.}

In Figure~\ref{fig4}, we compute this error at $T=1$. The full solution $u(T)$ is approximated by \eqref{ref_ex1}, the deterministic component $\psi(T)$ by \eqref{kat_ex0}, and the linearized fluctuation $\chi(T)$ by \eqref{stoc_ex1}. We use $\tau=2^{-12}$, $\rho=7/2+1e-10$, and $q=3+1e-10$, so that $\xi\in H^3$ and $Q^{1/2}\in\mathcal L_2^1$. The tested noise amplitudes are $\varepsilon_j=2^{-j}$, $j=1,\dots,8$. The observed first-order decay in $\varepsilon$ confirms the linearization estimate.

\begin{figure}
\centering
\subfigure{\includegraphics[width=0.5\textwidth]{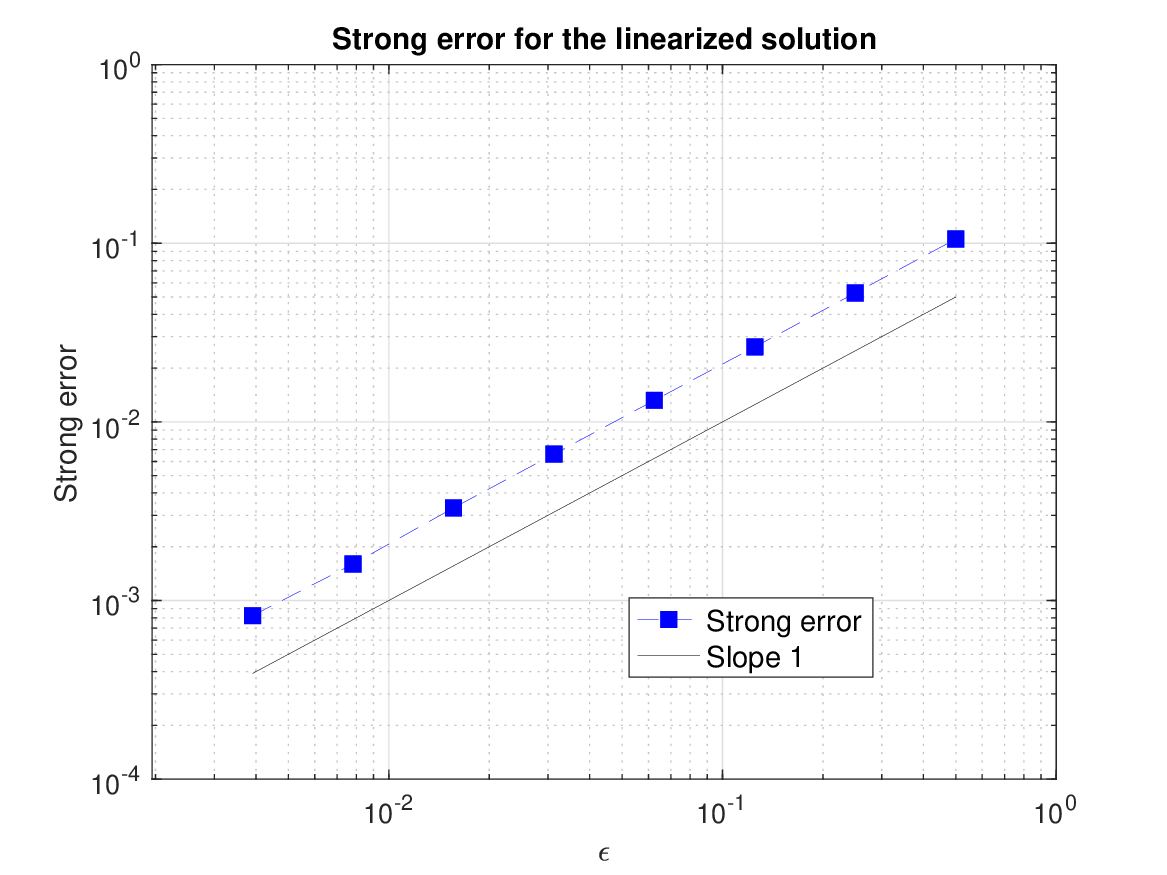}}
          \caption{$L^2(\Omega;L^2)$ errors for $u^\varepsilon(1)-\chi(1)$ as $\varepsilon\to0$, with $\xi\in H^3$ and $Q^{1/2}\in\mathcal L_2^1$. The plot is in log-log scale.}
           \label{fig4}
\end{figure}

\appendix

\section{Proof of technical results}
\label{appen}
\subsection{Proof of Lemma \ref{lem_H1b}}
\label{appen_h1b}
First, we note that, since $\xi\in H^2$, Remark \ref{wp_stoc} gives $\psi\in C([0,T];H^2)$.
The existing well-posedness theory for stochastic KdV-type equations (see, for instance, \cite{deb2,deb0,deb1,Oh}) yields existence and uniqueness of solutions to \eqref{lin_eq}, with $\chi\in C([0,T]; H^1)$. Moreover, \eqref{fluct} gives $u^\varepsilon\in C([0,T]; H^1)$.
We have to prove the bound in \eqref{es_t}. We only give the estimate for $u^\varepsilon$ since the argument for $\chi$ is analogous.
We begin by establishing the following moment estimate
\begin{equation}\label{L2es}
\sup_{\varepsilon\in(0,1)}
\mathbbm{E}\Big[\sup_{t\in[0,T]}\|u^\varepsilon(t)\|^{2p}\Big]
<\infty.
\end{equation}
Let $R\ge 1$ be a sufficiently large number, and define the stopping time
$$
\tau_R:=T\wedge \inf\bigl\{t>0:\|u^\varepsilon(t)\|_{H^1}\ge R\bigr\}.
$$
Then, by \cite[Lemma 3.2]{mao}, one has $\tau_R\to T$ almost surely as $R\to\infty$. 
Define
$$
u^{\varepsilon,R}(t):=u^\varepsilon(t\wedge \tau_R),
\qquad t\in[0,T].
$$
Then $u^{\varepsilon,R}$ satisfies
\begin{align*}
u^{\varepsilon,R}(t)
=
&\int_0^t \Big(
-\partial_x^3u^{\varepsilon,R}(s)
+2\mu\,\partial_x\bigl(\psi(s)u^{\varepsilon,R}(s)\bigr)
+\mu\varepsilon\,\partial_x\bigl((u^{\varepsilon,R}(s))^2\bigr)
\Big)\mathbf{1}_{[0,\tau_R]}(s)\,ds\\
&+\int_0^t \mathbf{1}_{[0,\tau_R]}(s)Q^{1/2}\,dW(s).
\end{align*}

First, we consider the case $p\ge 2$. Applying It\^o's formula to $\|u^{\varepsilon,R}(t)\|^{2p}$ yields
\begin{align*}
\|u^{\varepsilon,R}(t)\|^{2p}
&=-2p\int_0^t \|u^{\varepsilon,R}(s)\|^{2p-2}
\big\langle u^{\varepsilon,R}(s),\partial_x^3u^{\varepsilon,R}(s)\big\rangle
\mathbf{1}_{[0,\tau_R]}(s)\,ds\\
&\quad
+2\mu p\int_0^t \|u^{\varepsilon,R}(s)\|^{2p-2}
\big\langle u^{\varepsilon,R}(s),
2\psi(s)\partial_xu^{\varepsilon,R}(s)
+2u^{\varepsilon,R}(s)\partial_x \psi(s)\\
&\hspace{5.9cm}
+\varepsilon\,\partial_x\bigl((u^{\varepsilon,R}(s))^2\bigr)
\big\rangle
\mathbf{1}_{[0,\tau_R]}(s)\,ds\\
&\quad
+p\sum_k \langle Q^{1/2}e_k,Q^{1/2}e_k\rangle
\int_0^t \|u^{\varepsilon,R}(s)\|^{2p-2}\mathbf{1}_{[0,\tau_R]}(s)\,ds\\
&\quad
+2p(p-1)\int_0^t \|u^{\varepsilon,R}(s)\|^{2p-4}
\sum_k \big\langle u^{\varepsilon,R}(s),Q^{1/2}e_k\big\rangle^2
\mathbf{1}_{[0,\tau_R]}(s)\,ds\\
&\quad
+2p\int_0^t \|u^{\varepsilon,R}(s)\|^{2p-2}
\mathbf{1}_{[0,\tau_R]}(s)
\big\langle u^{\varepsilon,R}(s),Q^{1/2}\,dW(s)\big\rangle.
\end{align*}
Using integration by parts and periodic boundary conditions of $\partial_x u$, we have
$$
\big\langle u^{\varepsilon,R}(s),\partial_x^3u^{\varepsilon,R}(s)\big\rangle
=-\big\langle \partial_x u^{\varepsilon,R}(s), \partial_x^2 u^{\varepsilon,R}(s) \big\rangle =0
$$
and
\begin{align*}
&\big\langle u^{\varepsilon,R}(s),
\psi(s)\partial_xu^{\varepsilon,R}(s)\big\rangle
+\big\langle u^{\varepsilon,R}(s),
u^{\varepsilon,R}(s)\partial_x\psi(s)\big\rangle
=\frac12\big\langle \partial_x\bigl((u^{\varepsilon,R}(s))^2\bigr),\psi(s)\big\rangle
+\big\langle (u^{\varepsilon,R}(s))^2,\partial_x\psi(s)\big\rangle\\
&
=\frac12\big\langle (u^{\varepsilon,R}(s))^2,\partial_x\psi(s)\big\rangle 
\le \frac12\|\partial_x\psi(s)\|_{L^\infty}\|u^{\varepsilon,R}(s)\|^2.
\end{align*}
Therefore, 
\begin{align*}
\|u^{\varepsilon,R}(t)\|^{2p}
&\le2p\mu \|\psi\|_{L^\infty([0,t];W^{1,\infty})}
\int_0^t \|u^{\varepsilon,R}(s)\|^{2p}\,ds\\
&+\|Q^{1/2}\|_{\mathcal{L}_2^0}^2\bigl(p+2p(p-1)\bigr)
\int_0^t \|u^{\varepsilon,R}(s)\|^{2p-2}\,ds\\
&+2p\Big|
\int_0^t \|u^{\varepsilon,R}(s)\|^{2p-2}
\mathbf{1}_{[0,\tau_R]}(s)
\big\langle u^{\varepsilon,R}(s),Q^{1/2}\,dW(s)\big\rangle
\Big|.
\end{align*}
By means of Young's inequality, we infer that
\begin{align*}
\|u^{\varepsilon,R}(t)\|^{2p}
&\le \Big(2\mu p\|\psi\|_{L^\infty([0,t];W^{1,\infty})}+\frac{p-1}{p}\Big)
\int_0^t \|u^{\varepsilon,R}(s)\|^{2p}\,ds\\
&+\frac1p\|Q^{1/2}\|_{\mathcal{L}_2^0}^{2p}(2p^2-p)^p\,t\\
&+2p\Big|
\int_0^t \|u^{\varepsilon,R}(s)\|^{2p-2}
\mathbf{1}_{[0,\tau_R]}(s)
\big\langle u^{\varepsilon,R}(s),Q^{1/2}\,dW(s)\big\rangle
\Big|.
\end{align*}
Thus, we deduce that
\begin{align*}
\mathbbm{E}\Big[\sup_{s\in[0,t]}\|u^{\varepsilon,R}(s)\|^{2p}\Big]
\le&\frac1p\|Q^{1/2}\|_{\mathcal{L}_2^0}^{2p}(2p^2-p)^pT\\
&+\Big(2\mu p\|\psi\|_{L^\infty([0,T];W^{1,\infty})}+\frac{p-1}{p}\Big)
\int_0^t \mathbbm{E}\big[\|u^{\varepsilon,R}(s)\|^{2p}\big]\,ds\\
&+2p\,\mathbbm{E}\Big[
\sup_{s\in[0,t]}
\Big|
\int_0^s \mathbf{1}_{[0,\tau_R]}(r)\|u^{\varepsilon,R}(r)\|^{2p-2}
\big\langle u^{\varepsilon,R}(r),Q^{1/2}\,dW(r)\big\rangle
\Big|
\Big].
\end{align*}
Applying the Burkholder--Davis--Gundy inequality to
$$
\mathbbm{E}\Big[\sup_{s\in[0,t]}\Big|\int_0^s \Phi(r)Q^{1/2}\,dW(r)\Big|\Big],
$$
with $\Phi(r):Q^{1/2}L^2\to\mathbbm{R}$ defined by
$$
\Phi(r)w
=
\mathbf{1}_{[0,\tau_R]}(r)\|u^{\varepsilon,R}(r)\|^{2p-2}
\big\langle u^{\varepsilon,R}(r),w\big\rangle,
\qquad w\in Q^{1/2}L^2,
$$ we obtain
\begin{align*}
&\mathbbm{E}\Big[
\sup_{s\in[0,t]}
\Big|
\int_0^s \mathbf{1}_{[0,\tau_R]}(r)\|u^{\varepsilon,R}(r)\|^{2p-2}
\big\langle u^{\varepsilon,R}(r),Q^{1/2}\,dW(r)\big\rangle
\Big|
\Big] \\
&\qquad
\le C\,\mathbbm{E}\Big[
\Big(
\int_0^t \|u^{\varepsilon,R}(s)\|^{4p-4}
\sum_k \big|\langle u^{\varepsilon,R}(s),Q^{1/2}e_k\rangle\big|^2\,ds
\Big)^{1/2}
\Big].
\end{align*}
Using the Cauchy--Schwarz inequality, we further deduce that
\begin{align*}
&\mathbbm{E}\Big[
\sup_{s\in[0,t]}
\Big|
\int_0^s \mathbf{1}_{[0,\tau_R]}(r)\|u^{\varepsilon,R}(r)\|^{2p-2}
\big\langle u^{\varepsilon,R}(r),Q^{1/2}\,dW(r)\big\rangle
\Big|
\Big] \\
&\qquad
\le C\|Q^{1/2}\|_{\mathcal{L}_2^0}
\mathbbm{E}\Big[
\Big(
\int_0^t \|u^{\varepsilon,R}(s)\|^{4p-2}\,ds
\Big)^{1/2}
\Big]\\
&\qquad\le C\|Q^{1/2}\|_{\mathcal{L}_2^0}
\mathbbm{E}\Big[
\Big(\sup_{s\in[0,t]}\|u^{\varepsilon,R}(s)\|^{2p}\Big)^{1/2}
\Big(
\int_0^t \|u^{\varepsilon,R}(s)\|^{2p-2}\,ds
\Big)^{1/2}
\Big].
\end{align*}
Applying Young's inequality, for any $\eta>0$, yields
\begin{align*}
&\mathbbm{E}\Big[
\sup_{s\in[0,t]}
\Big|
\int_0^s \mathbf{1}_{[0,\tau_R]}(r)\|u^{\varepsilon,R}(r)\|^{2p-2}
\big\langle u^{\varepsilon,R}(r),Q^{1/2}\,dW(r)\big\rangle
\Big|
\Big] \\
&\qquad
\le C\|Q^{1/2}\|_{\mathcal{L}_2^0}\eta\,
\mathbbm{E}\Big[\sup_{s\in[0,t]}\|u^{\varepsilon,R}(s)\|^{2p}\Big]
+C(\eta)\|Q^{1/2}\|_{\mathcal{L}_2^0}
\mathbbm{E}\Big[\int_0^t \|u^{\varepsilon,R}(s)\|^{2p-2}\,ds\Big].
\end{align*}
Invoking Young's inequality once again, we infer that
$$
\|u^{\varepsilon,R}(s)\|^{2p-2}
\le \frac{2p-2}{2p}\|u^{\varepsilon,R}(s)\|^{2p}
+\frac{1}{p},
$$
and hence
\begin{align*}
&\mathbbm{E}\Big[
\sup_{s\in[0,t]}
\Big|
\int_0^s \mathbf{1}_{[0,\tau_R]}(r)\|u^{\varepsilon,R}(r)\|^{2p-2}
\big\langle u^{\varepsilon,R}(r),Q^{1/2}\,dW(r)\big\rangle
\Big|
\Big] \\
&\qquad
\le C\|Q^{1/2}\|_{\mathcal{L}_2^0}\eta\,
\mathbbm{E}\Big[\sup_{s\in[0,t]}\|u^{\varepsilon,R}(s)\|^{2p}\Big]
+\frac{1}{p}C(\eta)\|Q^{1/2}\|_{\mathcal{L}_2^0}\,t \\
&\hspace{2cm}
+\frac{2p-2}{2p}C(\eta)\|Q^{1/2}\|_{\mathcal{L}_2^0}
\mathbbm{E}\Big[\int_0^t \|u^{\varepsilon,R}(s)\|^{2p}\,ds\Big].
\end{align*}
Using
$
\mathbbm{E}\big[\|u^{\varepsilon,R}(s)\|^{2p}\big]
\le
\mathbbm{E}\Big[\sup_{r\in[0,s]}\|u^{\varepsilon,R}(r)\|^{2p}\Big]$ and $t\le T,
$
we conclude that
\begin{align*}
&\mathbbm{E}\Big[
\sup_{s\in[0,t]}
\Big|
\int_0^s \mathbf{1}_{[0,\tau_R]}(r)\|u^{\varepsilon,R}(r)\|^{2p-2}
\big\langle u^{\varepsilon,R}(r),Q^{1/2}\,dW(r)\big\rangle
\Big|
\Big] \\
&\qquad
\le C\|Q^{1/2}\|_{\mathcal{L}_2^0}\eta\,
\mathbbm{E}\Big[\sup_{s\in[0,t]}\|u^{\varepsilon,R}(s)\|^{2p}\Big]
+\frac{T}{p}C(\eta)\|Q^{1/2}\|_{\mathcal{L}_2^0} \\
&\hspace{2cm}
+\frac{2p-2}{2p}C(\eta)\|Q^{1/2}\|_{\mathcal{L}_2^0}
\int_0^t \mathbbm{E}\Big[\sup_{r\in[0,s]}\|u^{\varepsilon,R}(r)\|^{2p}\Big]\,ds.
\end{align*}
We thus conclude that there exist positive constants
$$
C_1=C_1\bigl(\mu,\|\psi\|_{L^\infty([0,T];W^{1,\infty})},p,T,\|Q^{1/2}\|_{\mathcal{L}_2^0},\eta\bigr),
$$
and
$$
C_2=C_2\bigl(p,T,\eta,\|Q^{1/2}\|_{\mathcal{L}_2^0}\bigr),
$$
such that
\begin{align*}
\mathbbm{E}\Big[\sup_{s\in[0,t]}\|u^{\varepsilon,R}(s)\|^{2p}\Big]
\le C_2
&+C_1\int_0^t
\mathbbm{E}\Big[\sup_{r\in[0,s]}\|u^{\varepsilon,R}(r)\|^{2p}\Big]\,ds\\
&+C\|Q^{1/2}\|_{\mathcal{L}_2^0}\eta\,
\mathbbm{E}\Big[\sup_{s\in[0,t]}\|u^{\varepsilon,R}(s)\|^{2p}\Big].
\end{align*}
Choosing $\eta>0$ sufficiently small so that
$
C\|Q^{1/2}\|_{\mathcal{L}_2^0}\eta<1,
$
we obtain
$$
\mathbbm{E}\Big[\sup_{s\in[0,t]}\|u^{\varepsilon,R}(s)\|^{2p}\Big]
\le \widetilde{C}_2
+\widetilde{C}_1\int_0^t
\mathbbm{E}\Big[\sup_{r\in[0,s]}\|u^{\varepsilon,R}(r)\|^{2p}\Big]\,ds,
$$
where
$$
\widetilde{C}_2=\bigl(1-\eta C\|Q^{1/2}\|_{\mathcal{L}_2^0}\bigr)^{-1}C_2,
\qquad
\widetilde{C}_1=\bigl(1-\eta C\|Q^{1/2}\|_{\mathcal{L}_2^0}\bigr)^{-1}C_1.
$$
An application of the Gronwall inequality then gives
\begin{equation}\label{stocR0}
\mathbbm{E}\Big[\sup_{t\in[0,T]}\|u^{\varepsilon,R}(t)\|^{2p}\Big]
\le \widetilde{C}_2 e^{\widetilde{C}_1T},
\end{equation}
and hence
\begin{equation}\label{stoc0}
\mathbbm{E}\Big[\sup_{t\in[0,\tau_R]}\|u^\varepsilon(t)\|^{2p}\Big]
\le \widetilde{C}_2 e^{\widetilde{C}_1T}.
\end{equation}
Finally, \eqref{L2es} follows by letting $R\to\infty$ on both sides of \eqref{stoc0}; see also \cite{mao}.
The result for $p=1$ follows in a similar way.
This completes the proof of \eqref{L2es}.

We now prove the derivative estimate
\begin{equation}\label{H1es}
\sup_{\varepsilon\in(0,1)}
\mathbbm{E}\Big[\sup_{t\in[0,T]}\|\partial_xu^\varepsilon(t)\|^{2p}\Big]
<\infty,
\end{equation}
for every $p\in\mathbbm{N}$. To this end, we introduce the functional  $V:H^1\to\mathbbm{R}$ defined by
$$
V(v):=\frac12\|\partial_x v\|^2+\frac{\mu\varepsilon}{3}\int_{\mathbbm{T}}v^3\,dx.
$$
Applying It\^o's formula to $\bigl(V(u^{\varepsilon,R}(t))\bigr)^{2p}$ and using the fact that 
$$
DV(v)h=\langle \partial_x v,\partial_x h\rangle+\mu\varepsilon\langle v^2,h\rangle,
\qquad
D^2V(v)(h,k)=\langle \partial_x h,\partial_x k\rangle+2\mu\varepsilon\langle vh,k\rangle,
$$
together with the property 
$$
\big\langle \partial_xu^{\varepsilon,R}(t),\partial_x^4u^{\varepsilon,R}(t)\big\rangle=0,
$$
we obtain
\begin{align*}
&\bigl(V(u^{\varepsilon,R}(t))\bigr)^{2p}
=4\mu p\int_0^t \mathbf{1}_{[0,\tau_R]}(s)\bigl(V(u^{\varepsilon,R}(s))\bigr)^{2p-1}
\big\langle \partial_xu^{\varepsilon,R}(s),
\partial_x\bigl(\psi(s)\partial_xu^{\varepsilon,R}(s)+u^{\varepsilon,R}(s)\partial_x\psi(s)\bigr)\big\rangle\,ds\\
&\quad
+4\mu p\varepsilon\int_0^t \mathbf{1}_{[0,\tau_R]}(s)\bigl(V(u^{\varepsilon,R}(s))\bigr)^{2p-1}
\big\langle \partial_xu^{\varepsilon,R}(s),
\partial_x\bigl(u^{\varepsilon,R}(s)\partial_xu^{\varepsilon,R}(s)\bigr)\big\rangle\,ds\\
&\quad
+2p\int_0^t \mathbf{1}_{[0,\tau_R]}(s)\bigl(V(u^{\varepsilon,R}(s))\bigr)^{2p-1}
\big\langle \partial_xu^{\varepsilon,R}(s),\partial_xQ^{1/2}\,dW(s)\big\rangle\\
&\quad
+2p\mu\varepsilon\int_0^t \mathbf{1}_{[0,\tau_R]}(s)\bigl(V(u^{\varepsilon,R}(s))\bigr)^{2p-1}
\big\langle (u^{\varepsilon,R}(s))^2,Q^{1/2}\,dW(s)\big\rangle\\
&\quad
+2p\mu\varepsilon\int_0^t \mathbf{1}_{[0,\tau_R]}(s)\bigl(V(u^{\varepsilon,R}(s))\bigr)^{2p-1}
\big\langle (u^{\varepsilon,R}(s))^2,-\partial_x^3u^{\varepsilon,R}(s)\\
&\hspace{1.3cm}
+2\mu\bigl(\psi(s)\partial_xu^{\varepsilon,R}(s)
+u^{\varepsilon,R}(s)\partial_x\psi(s)\bigr)
+2\mu\varepsilon u^{\varepsilon,R}(s)\partial_xu^{\varepsilon,R}(s)
\big\rangle\,ds\\
&\quad
+p(2p-1)\int_0^t \mathbf{1}_{[0,\tau_R]}(s)\bigl(V(u^{\varepsilon,R}(s))\bigr)^{2p-2}
\sum_k\Big(
\big\langle \partial_xu^{\varepsilon,R}(s),\partial_xQ^{1/2}e_k\big\rangle+\mu\varepsilon\big\langle (u^{\varepsilon,R}(s))^2,Q^{1/2}e_k\big\rangle
\Big)^2\,ds\\
&\quad
+p\int_0^t \mathbf{1}_{[0,\tau_R]}(s)\bigl(V(u^{\varepsilon,R}(s))\bigr)^{2p-1}
\sum_k\Big(
\|\partial_xQ^{1/2}e_k\|^2
+2\mu\varepsilon\big\langle u^{\varepsilon,R}(s)Q^{1/2}e_k,Q^{1/2}e_k\big\rangle
\Big)\,ds.
\end{align*}
By integration by parts and periodicity of $\psi$ and $\partial_x u^{\varepsilon,R}$, we further have
\begin{align*}
&\big\langle \partial_xu^{\varepsilon,R}(s),
\partial_x\bigl(\psi(s)\partial_xu^{\varepsilon,R}(s)
+u^{\varepsilon,R}(s)\partial_x\psi(s)\bigr)\big\rangle\\
=&\frac32\big\langle \partial_x\psi(s),(\partial_xu^{\varepsilon,R}(s))^2\big\rangle
+\big\langle \partial_x^2\psi(s),
u^{\varepsilon,R}(s)\partial_xu^{\varepsilon,R}(s)\big\rangle,
\end{align*}
and
\begin{align*}
&\quad\big\langle \partial_xu^{\varepsilon,R}(s),
\partial_x\bigl(u^{\varepsilon,R}(s)\partial_xu^{\varepsilon,R}(s)\bigr)\big\rangle
=\frac12\big\langle \partial_xu^{\varepsilon,R}(s),
\partial_x^2\bigl((u^{\varepsilon,R}(s))^2\bigr)\big\rangle\\
&=-\frac12\big\langle \partial_x^2u^{\varepsilon,R}(s),
\partial_x\bigl((u^{\varepsilon,R}(s))^2\bigr)\big\rangle=\frac12\big\langle \partial_x^3u^{\varepsilon,R}(s),
(u^{\varepsilon,R}(s))^2\big\rangle,
\end{align*}
$$
\big\langle (u^{\varepsilon,R}(s))^2,
u^{\varepsilon,R}(s)\partial_xu^{\varepsilon,R}(s)\big\rangle=0,
$$
and
\begin{align*}
&\quad\big\langle (u^{\varepsilon,R}(s))^2,
\psi(s)\partial_xu^{\varepsilon,R}(s)\big\rangle
=-\big\langle \partial_x\bigl((u^{\varepsilon,R}(s))^2\psi(s)\bigr),
u^{\varepsilon,R}(s)\big\rangle\\
&=-\big\langle \partial_x\bigl((u^{\varepsilon,R}(s))^2\bigr)\psi(s),
u^{\varepsilon,R}(s)\big\rangle
-\big\langle (u^{\varepsilon,R}(s))^2,
u^{\varepsilon,R}(s)\partial_x\psi(s)\big\rangle.
\end{align*}
Hence,
\begin{align}
&\quad\bigl(V(u^{\varepsilon,R}(t))\bigr)^{2p}
=4\mu p\int_0^t \mathbf{1}_{[0,\tau_R]}(s)\bigl(V(u^{\varepsilon,R}(s))\bigr)^{2p-1}
\big\langle \partial_x^2\psi(s),
u^{\varepsilon,R}(s)\partial_xu^{\varepsilon,R}(s)\big\rangle\,ds \notag \\
&\quad
+6\mu p\int_0^t \mathbf{1}_{[0,\tau_R]}(s)\bigl(V(u^{\varepsilon,R}(s))\bigr)^{2p-1}
\big\langle \partial_x\psi(s),(\partial_xu^{\varepsilon,R}(s))^2\big\rangle\,ds\notag\\
&\quad
-4p\mu^2\varepsilon\int_0^t \mathbf{1}_{[0,\tau_R]}(s)\bigl(V(u^{\varepsilon,R}(s))\bigr)^{2p-1}
\big\langle \partial_x\bigl((u^{\varepsilon,R}(s))^2\bigr)\psi(s),
u^{\varepsilon,R}(s)\big\rangle\,ds\notag \\
&\quad
+2p\int_0^t \mathbf{1}_{[0,\tau_R]}(s)\bigl(V(u^{\varepsilon,R}(s))\bigr)^{2p-1}
\big\langle \partial_xu^{\varepsilon,R}(s),\partial_xQ^{1/2}\,dW(s)\big\rangle\label{nw0} \\
&\quad
+2p\mu\varepsilon\int_0^t \mathbf{1}_{[0,\tau_R]}(s)\bigl(V(u^{\varepsilon,R}(s))\bigr)^{2p-1}
\big\langle (u^{\varepsilon,R}(s))^2,Q^{1/2}\,dW(s)\big\rangle\notag \\
&\quad
+p(2p-1)\int_0^t \mathbf{1}_{[0,\tau_R]}(s)\bigl(V(u^{\varepsilon,R}(s))\bigr)^{2p-2}
\sum_k\Big(
\big\langle \partial_xu^{\varepsilon,R}(s),\partial_xQ^{1/2}e_k\big\rangle
+\mu\varepsilon\big\langle (u^{\varepsilon,R}(s))^2,Q^{1/2}e_k\big\rangle
\Big)^2\,ds\notag \\
&\quad
+p\int_0^t \mathbf{1}_{[0,\tau_R]}(s)\bigl(V(u^{\varepsilon,R}(s))\bigr)^{2p-1}
\sum_k\Big(
\|\partial_xQ^{1/2}e_k\|^2+2\mu\varepsilon\big\langle u^{\varepsilon,R}(s)Q^{1/2}e_k,Q^{1/2}e_k\big\rangle
\Big)\,ds\notag .
\end{align}
By the Cauchy--Schwarz inequality, the Gagliardo--Nirenberg inequality \eqref{gn1}
and Young's inequality with conjugate exponents $\frac43$ and $4$, we have
\begin{align}
\big\langle \partial_x^2\psi(s),
u^{\varepsilon,R}(s)\partial_xu^{\varepsilon,R}(s)\big\rangle
&\le \|u^{\varepsilon,R}(s)\|_{L^\infty}
\|\partial_x^2\psi(s)\|\,
\|\partial_xu^{\varepsilon,R}(s)\|\notag \\
&\le C\|\partial_x^2\psi(s)\|\,
\|u^{\varepsilon,R}(s)\|^{1/2}
\|\partial_xu^{\varepsilon,R}(s)\|^{3/2}\label{nw1}\\
&\le C\|\partial_x^2\psi(s)\|
\Big(
\|u^{\varepsilon,R}(s)\|^2
+\|\partial_xu^{\varepsilon,R}(s)\|^2
\Big)\notag.
\end{align}
Moreover, by integration by parts, the chain rule, and again the Gagliardo--Nirenberg inequality \eqref{gn1},
\begin{align}
&\big\langle \partial_x\bigl((u^{\varepsilon,R}(s))^2\bigr)\psi(s),
u^{\varepsilon,R}(s)\big\rangle=-\frac 23\langle u^{\varepsilon,R}(s)^3, \partial_x \psi(s)\rangle \notag\\
\le &\frac23\|\partial_x\psi(s)\|_{L^\infty}
\|u^{\varepsilon,R}(s)\|_{L^3}^3
\le \frac23 C\|\partial_x\psi(s)\|_{L^\infty}
\|\partial_xu^{\varepsilon,R}(s)\|^{1/2}
\|u^{\varepsilon,R}(s)\|^{5/2} \label{nw2}.
\end{align}
Furthermore,
\begin{align}
&\qquad\quad\sum_k \Big(
\big\langle \partial_xu^{\varepsilon,R}(s),\partial_xQ^{1/2}e_k\big\rangle
+\mu\varepsilon\big\langle (u^{\varepsilon,R}(s))^2,Q^{1/2}e_k\big\rangle
\Big)^2\notag \\
&\qquad
\le 2\sum_k \big|\big\langle \partial_xu^{\varepsilon,R}(s),\partial_xQ^{1/2}e_k\big\rangle\big|^2
+2\mu^2\varepsilon^2\sum_k \big|\big\langle (u^{\varepsilon,R}(s))^2,Q^{1/2}e_k\big\rangle\big|^2\notag \\
&\qquad
\le 2\|\partial_xu^{\varepsilon,R}(s)\|^2\|Q^{1/2}\|_{\mathcal{L}_2^1}^2
+2\mu^2\varepsilon^2\|u^{\varepsilon,R}(s)\|_{L^\infty}^2
\|u^{\varepsilon,R}(s)\|^2
\|Q^{1/2}\|_{\mathcal{L}_2^0}^2\label{nw3}\\
&\qquad
\le 2\|\partial_xu^{\varepsilon,R}(s)\|^2\|Q^{1/2}\|_{\mathcal{L}_2^1}^2
+2\mu^2\varepsilon^2
\|\partial_xu^{\varepsilon,R}(s)\|
\|u^{\varepsilon,R}(s)\|^3
\|Q^{1/2}\|_{\mathcal{L}_2^0}^2\notag,
\end{align}
and
\begin{align}
&\quad\qquad\sum_k \Big(
\|\partial_xQ^{1/2}e_k\|^2
+2\mu\varepsilon\big\langle u^{\varepsilon,R}(s)Q^{1/2}e_k,Q^{1/2}e_k\big\rangle
\Big)\notag\\
&\qquad
\le \|Q^{1/2}\|_{\mathcal{L}_2^1}^2
+2\mu\varepsilon\|u^{\varepsilon,R}(s)\|
\sum_k \|Q^{1/2}e_k\|_{L^\infty}\|Q^{1/2}e_k\|\notag\\
&\qquad
\le \|Q^{1/2}\|_{\mathcal{L}_2^1}^2
+C\mu\varepsilon\|u^{\varepsilon,R}(s)\|
\sum_k \|Q^{1/2}e_k\|^{3/2}\|Q^{1/2}e_k\|^{1/2}_{\dot{H}^1}\notag\\
&\qquad
\le \|Q^{1/2}\|_{\mathcal{L}_2^1}^2
+C\mu\varepsilon\|u^{\varepsilon,R}(s)\|
\Big(\sum_k \|Q^{1/2}e_k\|^{2}\Big)^{\frac{3}{4}}\big(\sum_k\|Q^{1/2}e_k\|^{2}_{\dot{H}^1}\Big)^{\frac{1}{4}}\label{nw4}\\
&\qquad
\le \|Q^{1/2}\|_{\mathcal{L}_2^1}^2
+C\mu\varepsilon\big\|u^{\varepsilon,R}(s)\big\|
\big\| Q^{1/2}\big\|_{\mathcal{L}_2^0}^{\frac{3}{2}}\big\| Q^{1/2}\big\|_{\mathcal{L}_2^1}^{\frac{1}{2}}\notag,
\end{align}
where we used again \eqref{gn1} and the Cauchy--Schwarz inequality.
Combining the estimates \eqref{nw1}--\eqref{nw4} into \eqref{nw0}, we arrive at
\begin{align*}
&|V(u^{\varepsilon,R}(t))|^{2p}
\le 4\mu pC\int_0^t |V(u^{\varepsilon,R}(s))|^{2p-1}
\|\partial_x^2\psi(s)\|
\Big(
\|u^{\varepsilon,R}(s)\|^2+\|\partial_xu^{\varepsilon,R}(s)\|^2
\Big)\,ds\\
&+6\mu p\int_0^t |V(u^{\varepsilon,R}(s))|^{2p-1}
\|\partial_x\psi(s)\|_{L^\infty}
\|\partial_xu^{\varepsilon,R}(s)\|^2\,ds\\
&+4\mu^2 pC\varepsilon\int_0^t |V(u^{\varepsilon,R}(s))|^{2p-1}
\|\partial_x\psi(s)\|_{L^\infty}
\|\partial_xu^{\varepsilon,R}(s)\|^{1/2}
\|u^{\varepsilon,R}(s)\|^{5/2}\,ds\\
&+2p\Big|\int_0^t \mathbf{1}_{[0,\tau_R]}(s)V(u^{\varepsilon,R}(s))^{2p-1}
\big\langle \partial_xu^{\varepsilon,R}(s),\partial_xQ^{1/2}\,dW(s)\big\rangle\Big|\\
&+2p\mu\varepsilon\Big|\int_0^t \mathbf{1}_{[0,\tau_R]}(s)V(u^{\varepsilon,R}(s))^{2p-1}
\big\langle (u^{\varepsilon,R}(s))^2,Q^{1/2}\,dW(s)\big\rangle\Big|\\
&+2p(2p-1)C\int_0^t |V(u^{\varepsilon,R}(s))|^{2p-2}
\Big(
\|\partial_xu^{\varepsilon,R}(s)\|^2\|Q^{1/2}\|_{\mathcal{L}_2^1}^2+\mu^2\varepsilon^2\|\partial_xu^{\varepsilon,R}(s)\|
\|u^{\varepsilon,R}(s)\|^3
\|Q^{1/2}\|_{\mathcal{L}_2^0}^2
\Big)\,ds\\
&+p\int_0^t |V(u^{\varepsilon,R}(s))|^{2p-1}
\Big(
\|Q^{1/2}\|_{\mathcal{L}_2^1}^2
+C\mu\varepsilon\big\|u^{\varepsilon,R}(s)\big\|
\big\| Q^{1/2}\big\|_{\mathcal{L}_2^0}^{\frac{3}{2}}\big\| Q^{1/2}\big\|_{\mathcal{L}_2^1}^{\frac{1}{2}}\Big)
\,ds.
\end{align*}

Moreover, by the definition of $V$, the Gagliardo-Nirenberg inequality \eqref{gn2}, Young's inequality, and the fact that $\varepsilon\in(0,1)$, we have
\begin{align*}
\|\partial_xu^{\varepsilon,R}(s)\|^2
&=2V(u^{\varepsilon,R}(s))
-\frac{2\mu\varepsilon}{3}\int_{\mathbbm{T}}(u^{\varepsilon,R}(s))^3\,dx\\
&\le 2|V(u^{\varepsilon,R}(s))|
+2\frac{\mu}{3}\|u^{\varepsilon,R}(s)\|_{L^3}^3\\
&\le 2|V(u^{\varepsilon,R}(s))|
+2\frac{\mu}{3}C
\|\partial_xu^{\varepsilon,R}(s)\|^{1/2}
\|u^{\varepsilon,R}(s)\|^{5/2}\\
&\le 2|V(u^{\varepsilon,R}(s))|
+\frac{\mu}{3} C
\Big(
\frac{\eta_1}{2}\|\partial_xu^{\varepsilon,R}(s)\|^2
+\eta_1^{-1/3}\frac32\|u^{\varepsilon,R}(s)\|^{10/3}
\Big),
\end{align*}
for any $\eta_1>0$.
Then, choosing $\eta_1$ such that $\frac{\eta_1\mu C}{6}<1$, 
we deduce that
\begin{equation}\label{grad_est}
\|\partial_xu^{\varepsilon,R}(s)\|^2
\le C\Big(
|V(u^{\varepsilon,R}(s))|
+\|u^{\varepsilon,R}(s)\|^{10/3}
\Big).
\end{equation}
Using \eqref{grad_est} and Young's inequality once again, we infer
\begin{align*}
&\quad\qquad|V(u^{\varepsilon,R}(s))|^{2p-1}
\|\partial_xu^{\varepsilon,R}(s)\|^{1/2}
\|u^{\varepsilon,R}(s)\|^{5/2}\\
&\qquad
\le |V(u^{\varepsilon,R}(s))|^{2p-1}
\Big(
\frac14\|\partial_xu^{\varepsilon,R}(s)\|^2
+\frac34\|u^{\varepsilon,R}(s)\|^{10/3}
\Big)\\
&\qquad
\le \frac{C}{4}|V(u^{\varepsilon,R}(s))|^{2p}
+\frac{C+3}{4}
|V(u^{\varepsilon,R}(s))|^{2p-1}
\|u^{\varepsilon,R}(s)\|^{10/3}\\
&\qquad
\le \frac{1}{8p}
\Big(
2pC+(2p-1)(C+3)
\Big)
|V(u^{\varepsilon,R}(s))|^{2p}
+\frac{C+3}{8p}
\|u^{\varepsilon,R}(s)\|^{20p/3},
\end{align*}
and
\begin{align*}
&\quad\qquad|V(u^{\varepsilon,R}(s))|^{2p-2}
\|\partial_xu^{\varepsilon,R}(s)\|
\|u^{\varepsilon,R}(s)\|^3\\
&\qquad
\le \frac12|V(u^{\varepsilon,R}(s))|^{2p-2}
\Big(
\|\partial_xu^{\varepsilon,R}(s)\|^2
+\|u^{\varepsilon,R}(s)\|^6
\Big)\\
&\qquad
\le \frac12|V(u^{\varepsilon,R}(s))|^{2p-2}
\Big[
C|V(u^{\varepsilon,R}(s))|
+C\|u^{\varepsilon,R}(s)\|^{10/3}
\Big]
+\frac{p-1}{2p}|V(u^{\varepsilon,R}(s))|^{2p}
+\frac{1}{2p}\|u^{\varepsilon,R}(s)\|^{6p}\\
&\qquad
\le \frac{C(4p-3)+2p-2}{4p}|V(u^{\varepsilon,R}(s))|^{2p}
+\frac{1}{2p}\|u^{\varepsilon,R}(s)\|^{6p}+\frac{1}{2p}C\|u^{\varepsilon,R}(s)\|^{10p/3}
+\frac{1}{4p}C.
\end{align*}
Then, using all the above bounds together with Young's inequality and the estimates
$$
\int_0^t |V(u^{\varepsilon,R}(s))|^{2p-1}\,ds
\le \frac{2p-1}{2p}\int_0^t |V(u^{\varepsilon,R}(s))|^{2p}\,ds+\frac{t}{2p}
$$
and
$$
\int_0^t |V(u^{\varepsilon,R}(s))|^{2p-2}
\|u^{\varepsilon,R}(s)\|^{10/3}\,ds
\le \frac{2p-2}{2p}\int_0^t |V(u^{\varepsilon,R}(s))|^{2p}\,ds
+\frac1p\int_0^t \|u^{\varepsilon,R}(s)\|^{10p/3}\,ds,
$$
we derive
\begin{align*}
|V(u^{\varepsilon,R}(t))|^{2p}
\le &C\int_0^t a(s)|V(u^{\varepsilon,R}(s))|^{2p}\,ds+C\int_0^t b(s)\,ds\\
&+2p\Big|\int_0^t \mathbf{1}_{[0,\tau_R]}(s)V(u^{\varepsilon,R}(s))^{2p-1}
\big\langle \partial_xu^{\varepsilon,R}(s),\partial_xQ^{1/2}\,dW(s)\big\rangle\Big|\\
&+2p\mu\varepsilon\Big|\int_0^t \mathbf{1}_{[0,\tau_R]}(s)V(u^{\varepsilon,R}(s))^{2p-1}
\big\langle (u^{\varepsilon,R}(s))^2,Q^{1/2}\,dW(s)\big\rangle\Big|,
\end{align*}
where
$$
a(s):=1+\|\partial_x\psi(s)\|_{L^\infty}+\|\partial_x^2\psi(s)\|,
$$
and
\begin{align*}
b(s):={}&
\|\partial_x^2\psi(s)\|\|u^{\varepsilon,R}(s)\|^{4p}
+\|\partial_x^2\psi(s)\|\|u^{\varepsilon,R}(s)\|^{20p/3}\\
&+\|\partial_x\psi(s)\|_{L^\infty}\|u^{\varepsilon,R}(s)\|^{20p/3}
+\|u^{\varepsilon,R}(s)\|^{6p}
+\|u^{\varepsilon,R}(s)\|^{10p/3}
+\|u^{\varepsilon,R}(s)\|^{2p}+1.
\end{align*}
Thus, also using  \eqref{L2es}, we deduce
\begin{align*}
\mathbbm{E}\Big[\sup_{s\in[0,t]}|V(u^{\varepsilon,R}(s))|^{2p}\Big]
\le & C_1\int_0^t
\mathbbm{E}\Big[\sup_{r\in[0,s]}|V(u^{\varepsilon,R}(r))|^{2p}\Big]\,ds
+C_2\\
&+2p\,\mathbbm{E}\Big[
\sup_{s\in[0,t]}
\Big|
\int_0^s \mathbf{1}_{[0,\tau_R]}(r)V(u^{\varepsilon,R}(r))^{2p-1}
\big\langle \partial_xu^{\varepsilon,R}(r),\partial_xQ^{1/2}\,dW(r)\big\rangle
\Big|
\Big]\\
&+2p\mu\varepsilon\,\mathbbm{E}\Big[
\sup_{s\in[0,t]}
\Big|
\int_0^s \mathbf{1}_{[0,\tau_R]}(r)V(u^{\varepsilon,R}(r))^{2p-1}
\big\langle (u^{\varepsilon,R}(r))^2,Q^{1/2}\,dW(r)\big\rangle
\Big|
\Big].
\end{align*}
For simplicity, let
$$
M_1(s):=\int_0^s \mathbf{1}_{[0,\tau_R]}(r)
V(u^{\varepsilon,R}(r))^{2p-1}
\big\langle \partial_x u^{\varepsilon,R}(r), \partial_x Q^{1/2}\,dW(r)\big\rangle.
$$
By the Burkholder--Davis--Gundy inequality and the Cauchy--Schwarz inequality, we have
\begin{align*}
\mathbbm{E}\Big[\sup_{s\in[0,t]}|M_1(s)|\Big]
&\le
C \mathbbm{E}\Big[
\Big(
\int_0^t
|V(u^{\varepsilon,R}(s))|^{4p-2}
\sum_k
\big|
\langle \partial_x u^{\varepsilon,R}(s), \partial_x Q^{1/2}e_k\rangle
\big|^2\,ds
\Big)^{1/2}
\Big] \\
&\le
C \|Q^{1/2}\|_{\mathcal{L}_2^1}
\mathbbm{E}\Big[
\Big(
\int_0^t
|V(u^{\varepsilon,R}(s))|^{4p-2}
\|\partial_x u^{\varepsilon,R}(s)\|^2\,ds
\Big)^{1/2}
\Big].
\end{align*}
Using \eqref{grad_est}, we further obtain
\begin{align*}
\mathbbm{E}\Big[\sup_{s\in[0,t]}|M_1(s)|\Big]
&\le
C\|Q^{1/2}\|_{\mathcal{L}_2^1}
\bigg\{
\mathbbm{E}\Big[
\Big(
\int_0^t
|V(u^{\varepsilon,R}(s))|^{4p-1}\,ds
\Big)^{1/2}
\Big] \\
&\hspace{4cm}
+
\mathbbm{E}\Big[
\Big(
\int_0^t
|V(u^{\varepsilon,R}(s))|^{4p-2}
\|u^{\varepsilon,R}(s)\|^{10/3}\,ds
\Big)^{1/2}
\Big]
\bigg\}.
\end{align*}
Then, we get
\begin{align*}
\mathbbm{E}\Big[\sup_{s\in[0,t]}|M_1(s)|\Big]
&\le
C\|Q^{1/2}\|_{\mathcal{L}_2^1}
\mathbbm{E}\Big[
\Big(
\sup_{r\in[0,t]} |V(u^{\varepsilon,R}(r))|^{2p}
\Big)^{1/2} \\
&\hspace{2cm}\times
\Big\{
\Big(
\int_0^t |V(u^{\varepsilon,R}(s))|^{2p-1}\,ds
\Big)^{1/2}
+
\Big(
\int_0^t
|V(u^{\varepsilon,R}(s))|^{2p-2}
\|u^{\varepsilon,R}(s)\|^{10/3}\,ds
\Big)^{1/2}
\Big\}
\Big].
\end{align*}
Applying Young's inequality, for any $\eta_2>0$, yields
\begin{align*}
\mathbbm{E}\Big[\sup_{s\in[0,t]}|M_1(s)|\Big]
&\le
C\|Q^{1/2}\|_{\mathcal{L}_2^1}\eta_2
\mathbbm{E}\Big[
\sup_{s\in[0,t]} |V(u^{\varepsilon,R}(s))|^{2p}
\Big] \\
&\quad
+C(\eta_2)\|Q^{1/2}\|_{\mathcal{L}_2^1}
\mathbbm{E}\Big[
\int_0^t |V(u^{\varepsilon,R}(s))|^{2p-1}\,ds
\Big] \\
&\quad
+C(\eta_2)\|Q^{1/2}\|_{\mathcal{L}_2^1}
\mathbbm{E}\Big[
\int_0^t
|V(u^{\varepsilon,R}(s))|^{2p-2}
\|u^{\varepsilon,R}(s)\|^{10/3}\,ds
\Big].
\end{align*}
Invoking Young's inequality once more, we infer that
\begin{align*}
|V(u^{\varepsilon,R}(s))|^{2p-1}
&\le
\frac{2p-1}{2p}|V(u^{\varepsilon,R}(s))|^{2p}
+\frac{1}{2p},\\
|V(u^{\varepsilon,R}(s))|^{2p-2}
\|u^{\varepsilon,R}(s)\|^{10/3}
&\le
\frac{2p-2}{2p}|V(u^{\varepsilon,R}(s))|^{2p}
+\frac{1}{p}\|u^{\varepsilon,R}(s)\|^{10p/3},
\end{align*}
and therefore
\begin{align*}
\mathbbm{E}\Big[\sup_{s\in[0,t]}|M_1(s)|\Big]
&\le
C\|Q^{1/2}\|_{\mathcal{L}_2^1}\eta_2
\mathbbm{E}\Big[
\sup_{s\in[0,t]} |V(u^{\varepsilon,R}(s))|^{2p}
\Big] \\
&\quad
+\frac{2p-1}{2p}C(\eta_2)\|Q^{1/2}\|_{\mathcal{L}_2^1}
\int_0^t
\mathbbm{E}\Big[
|V(u^{\varepsilon,R}(s))|^{2p}
\Big]\,ds \\
&\quad
+\frac{t}{2p}C(\eta_2)\|Q^{1/2}\|_{\mathcal{L}_2^1} \\
&\quad
+\frac{2p-2}{2p}C(\eta_2)\|Q^{1/2}\|_{\mathcal{L}_2^1}
\int_0^t
\mathbbm{E}\Big[
|V(u^{\varepsilon,R}(s))|^{2p}
\Big]\,ds \\
&\quad
+\frac{1}{p}C(\eta_2)\|Q^{1/2}\|_{\mathcal{L}_2^1}
\int_0^t
\mathbbm{E}\Big[
\|u^{\varepsilon,R}(s)\|^{10p/3}
\Big]\,ds.
\end{align*}

Similarly, let
$$
M_2(s):=\int_0^s \mathbf{1}_{[0,\tau_R]}(r)
V(u^{\varepsilon,R}(r))^{2p-1}
\big\langle (u^{\varepsilon,R}(r))^2, Q^{1/2}\,dW(r)\big\rangle.
$$
By the Burkholder--Davis--Gundy inequality and the Cauchy--Schwarz inequality,
\begin{align*}
\mathbbm{E}\Big[\sup_{s\in[0,t]}|M_2(s)|\Big]
&\le
C \mathbbm{E}\Big[
\Big(
\int_0^t
|V(u^{\varepsilon,R}(s))|^{4p-2}
\sum_k
\big|
\langle (u^{\varepsilon,R}(s))^2, Q^{1/2}e_k\rangle
\big|^2\,ds
\Big)^{1/2}
\Big] \\
&\le
C \mathbbm{E}\Big[
\Big(
\int_0^t
|V(u^{\varepsilon,R}(s))|^{4p-2}
\|u^{\varepsilon,R}(s)\|^4
\sum_k \|Q^{1/2}e_k\|_{L^\infty}^2\,ds
\Big)^{1/2}
\Big] \\
&\le
C \|Q^{1/2}\|_{\mathcal{L}_2^1}
\mathbbm{E}\Big[
\Big(
\int_0^t
|V(u^{\varepsilon,R}(s))|^{4p-2}
\|u^{\varepsilon,R}(s)\|^4\,ds
\Big)^{1/2}
\Big].
\end{align*}
Applying Young's inequality, we obtain
\begin{align*}
\mathbbm{E}\Big[\sup_{s\in[0,t]}|M_2(s)|\Big]
&\le
C \|Q^{1/2}\|_{\mathcal{L}_2^1}
\mathbbm{E}\Big[
\Big(
\sup_{r\in[0,t]} |V(u^{\varepsilon,R}(r))|^{2p}
\Big)^{1/2}
\Big(
\int_0^t
|V(u^{\varepsilon,R}(s))|^{2p-2}
\|u^{\varepsilon,R}(s)\|^4\,ds
\Big)^{1/2}
\Big] \\
&\le
C\eta_2 \|Q^{1/2}\|_{\mathcal{L}_2^1}
\mathbbm{E}\Big[
\sup_{s\in[0,t]} |V(u^{\varepsilon,R}(s))|^{2p}
\Big] \\
&\quad
+C(\eta_2)\|Q^{1/2}\|_{\mathcal{L}_2^1}
\mathbbm{E}\Big[
\int_0^t
|V(u^{\varepsilon,R}(s))|^{2p-2}
\|u^{\varepsilon,R}(s)\|^4\,ds
\Big].
\end{align*}
Using Young's inequality once again, we have
$$
|V(u^{\varepsilon,R}(s))|^{2p-2}
\|u^{\varepsilon,R}(s)\|^4
\le
\frac{2p-2}{2p}|V(u^{\varepsilon,R}(s))|^{2p}
+\frac{1}{p}\|u^{\varepsilon,R}(s)\|^{4p},
$$
which implies
\begin{align*}
\mathbbm{E}\Big[\sup_{s\in[0,t]}|M_2(s)|\Big]
&\le
C\eta_2 \|Q^{1/2}\|_{\mathcal{L}_2^1}
\mathbbm{E}\Big[
\sup_{s\in[0,t]} |V(u^{\varepsilon,R}(s))|^{2p}
\Big] \\
&\quad
+\frac{2p-2}{2p}C(\eta_2)\|Q^{1/2}\|_{\mathcal{L}_2^1}
\int_0^t
\mathbbm{E}\Big[
|V(u^{\varepsilon,R}(s))|^{2p}
\Big]\,ds \\
&\quad
+\frac{1}{p}C(\eta_2)\|Q^{1/2}\|_{\mathcal{L}_2^1}
\int_0^t
\mathbbm{E}\Big[
\|u^{\varepsilon,R}(s)\|^{4p}
\Big]\,ds.
\end{align*}
Hence, choosing $\eta_2$ sufficiently small depending on $\|Q^{1/2}\|_{\mathcal{L}_2^1}$, $p$, and $\mu$, we arrive at
$$
\mathbbm{E}\Big[\sup_{s\in[0,t]}|V(u^{\varepsilon,R}(s))|^{2p}\Big]
\le
C\int_0^t
\mathbbm{E}\Big[\sup_{r\in[0,s]}|V(u^{\varepsilon,R}(r))|^{2p}\Big]\,ds
+C.
$$
An application of Gronwall's inequality then gives
\begin{equation}\label{VR_bound}
\mathbbm{E}\Big[\sup_{t\in[0,T]}|V(u^{\varepsilon,R}(t))|^{2p}\Big]
\le
C\bigl(
p,\mu,\|\psi\|_{L^\infty(0,T;H^2)},
\|Q^{1/2}\|_{\mathcal{L}_2^1},T
\bigr)
<\infty.
\end{equation}
Combining \eqref{grad_est}, \eqref{VR_bound}, and \eqref{L2es}, we further obtain
\begin{align*}
\mathbbm{E}\Big[\sup_{t\in[0,T]}\|\partial_xu^{\varepsilon,R}(t)\|^{2p}\Big]
&\le C\,\mathbbm{E}\Big[\sup_{t\in[0,T]}|V(u^{\varepsilon,R}(t))|^{2p}\Big]
+C\,\mathbbm{E}\Big[\sup_{t\in[0,T]}\|u^{\varepsilon,R}(t)\|^{20p/3}\Big]\le C,
\end{align*}
 Equivalently,
$$
\mathbbm{E}\Big[\sup_{t\in[0,\tau_R]}\|\partial_xu^\varepsilon(t)\|^{2p}\Big]\le C.
$$
Finally, since $\tau_R\to T$ almost surely as $R\to\infty$ and
$$
\sup_{t\in[0,\tau_R]}\|\partial_xu^\varepsilon(t)\|^{2p}
\to
\sup_{t\in[0,T]}\|\partial_xu^\varepsilon(t)\|^{2p}
\qquad\text{a.s.,}
$$
the monotone convergence theorem yields
$$
\mathbbm{E}\Big[\sup_{t\in[0,T]}\|\partial_xu^\varepsilon(t)\|^{2p}\Big]\le C.
$$
The constant $C$ is independent of $\varepsilon \in (0,1)$, so \eqref{H1es} holds. This completes the proof.

\subsection{Proof of Lemma \ref{time_reg_est}}
\label{app_time_reg}
We first note that, by Remark \ref{wp_stoc} and Lemma \ref{lem_H1b}, the assumptions on the initial data imply
\begin{equation}
\label{reg_n0}
\|\tilde \psi\|_{\mathbbm{L}_T^{\infty,2}}<\infty,
\qquad
\|\tilde{\chi}\|_{\mathcal{L}_T^{\infty,p,1}}<\infty,
\qquad p\in \mathbbm{N}.
\end{equation}
Next, using the mild formulation \eqref{mild_sol_twis_stoc}, we write
$$
\tilde{\chi}(t+s)-\tilde{\chi}(t)
=
2\mu\int_0^s {\rm e}^{(t+r)\partial_x^3}\partial_x
\Bigl(
{\rm e}^{-(t+r)\partial_x^3}\tilde\psi(t+r)\,
{\rm e}^{-(t+r)\partial_x^3}\tilde{\chi}(t+r)
\Bigr)\,dr
+\int_t^{t+s}{\rm e}^{r\partial_x^3}Q^{1/2}\,dW(r).
$$
By Assumption \eqref{ass_zero_mode} and then \eqref{hom_b1}, it is suffices to estimate $\mathbbm{E}\|\tilde{\chi}(t+s)-\tilde{\chi}(t)\|_{\dot{H}^1}^2$.

An application of It\^o isometry gives
\begin{align}
\mathbbm{E}\|\tilde{\chi}(t+s)-\tilde{\chi}(t)\|_{\dot{H}^1}^2
&\le C
\mathbbm{E}\Bigl\|
\int_0^s {\rm e}^{(t+r)\partial_x^3}\partial_x
\Bigl(
{\rm e}^{-(t+r)\partial_x^3}\tilde\psi(t+r)\,
{\rm e}^{-(t+r)\partial_x^3}\tilde{\chi}(t+r)
\Bigr)\,dr
\Bigr\|_{\dot{H}^1}^2\notag \\
&\quad
+
C\mathbbm{E}\Bigl\|
\int_t^{t+s}{\rm e}^{r\partial_x^3}Q^{1/2}\,dW(r)
\Bigr\|_{\dot{H}^1}^2
\notag\\
&\le C
\mathbbm{E}\Bigl\|
\int_0^s {\rm e}^{(t+r)\partial_x^3}\partial_x
\Bigl(
{\rm e}^{-(t+r)\partial_x^3}\tilde\psi(t+r)\,
{\rm e}^{-(t+r)\partial_x^3}\tilde{\chi}(t+r)
\Bigr)\,dr
\Bigr\|_{\dot{H}^1}^2
+
C\int_t^{t+s}\|{\rm e}^{r\partial_x^3}Q^{1/2}\|_{\mathcal{L}_2^1}^2\,dr
\notag\\
&\le C
\mathbbm{E}\Bigl\|
\int_0^s {\rm e}^{(t+r)\partial_x^3}\partial_x
\Bigl(
{\rm e}^{-(t+r)\partial_x^3}\tilde\psi(t+r)\,
{\rm e}^{-(t+r)\partial_x^3}\tilde{\chi}(t+r)
\Bigr)\,dr
\Bigr\|_{\dot{H}^1}^2
+ C
s\|Q^{1/2}\|_{\mathcal{L}_2^1}^2.
\label{n2}
\end{align}

We now have to bound the  $\dot{H}^1$-norm of the above integral term.
To this end, we exploit the oscillatory structure of the integrand and perform an integration by parts in Fourier space.
An integration by parts here yields 
\begin{align}
&\Bigl\|
\int_0^s {\rm e}^{(t+r)\partial_x^3}\partial_x
\Bigl(
{\rm e}^{-(t+r)\partial_x^3}\tilde\psi(t+r)\,
{\rm e}^{-(t+r)\partial_x^3}\tilde{\chi}(t+r)
\Bigr)\,dr
\Bigr\|_{\dot H^1}\notag
\\
&\qquad =
\frac{1}{2\pi}
\Biggl\|
\sum_{\ell_1,\ell_2\ne 0}
i(\ell_1+\ell_2)\,{\rm e}^{i(\ell_1+\ell_2)x}
\int_0^s
{\rm e}^{-3i(t+r)\ell_1\ell_2(\ell_1+\ell_2)}
\hat{\tilde \psi}_{\ell_1}(t+r)\hat{\tilde \chi}_{\ell_2}(t+r)\,dr
\Biggr\|_{\dot H^1}\notag
\\
&\qquad =
\frac{1}{2\pi}
\Biggl\|
\sum_{\substack{\ell_1,\ell_2\ne 0\\\ell_1+\ell_2\ne 0}}
i(\ell_1+\ell_2)\,{\rm e}^{i(\ell_1+\ell_2)x}
\int_0^s
\frac{d}{dr}
\Bigl(
\frac{{\rm e}^{-3i(t+r)\ell_1\ell_2(\ell_1+\ell_2)}}{-3i\ell_1\ell_2(\ell_1+\ell_2)}
\Bigr)
\hat{\tilde \psi}_{\ell_1}(t+r)\hat{\tilde \chi}_{\ell_2}(t+r)\,dr
\Biggr\|_{\dot H^1}\notag \\
&\qquad \le
\frac{1}{2\pi}
\Biggl\|
\sum_{\substack{\ell_1,\ell_2\ne 0\\\ell_1+\ell_2\ne 0}}
\frac{1}{3\ell_1\ell_2}
{\rm e}^{-3it\ell_1\ell_2(\ell_1+\ell_2)}
\bigl({\rm e}^{-3is\ell_1\ell_2(\ell_1+\ell_2)}-1\bigr)
\hat{\tilde\psi}_{\ell_1}(t+s)\hat{\tilde \chi}_{\ell_2}(t+s)
{\rm e}^{i(\ell_1+\ell_2)x}
\Biggr\|_{\dot H^1}
\notag\\
&\qquad\quad +
\frac{1}{2\pi}
\Biggl\|
\sum_{\substack{\ell_1,\ell_2\ne 0\\\ell_1+\ell_2\ne 0}}
\frac{1}{3\ell_1\ell_2}
{\rm e}^{-3it\ell_1\ell_2(\ell_1+\ell_2)}
\Bigl(
\hat{\tilde\psi}_{\ell_1}(t+s)\hat{\tilde \chi}_{\ell_2}(t+s)
-\hat{\tilde\psi}_{\ell_1}(t)\hat{\tilde \chi}_{\ell_2}(t)
\Bigr)
{\rm e}^{i(\ell_1+\ell_2)x}
\Biggr\|_{\dot H^1}
\label{t1}\\
&\qquad\quad +
\frac{1}{2\pi}
\Biggl\|
\int_0^s
\sum_{\substack{\ell_1,\ell_2\ne 0\\\ell_1+\ell_2\ne 0}}
\frac{1}{3\ell_1\ell_2}
{\rm e}^{-3i(t+r)\ell_1\ell_2(\ell_1+\ell_2)}
\,d\Bigl(
\hat{\tilde\psi}_{\ell_1}(t+r)\hat{\tilde \chi}_{\ell_2}(t+r)
\Bigr)
{\rm e}^{i(\ell_1+\ell_2)x}
\Biggr\|_{\dot H^1},\notag
\end{align}
where we also used the identity
$$
\int_0^s \dot f(r)\,g(r)\,dr
=
f(s)g(s)-f(0)g(0)-\int_0^s f(r)\,dg(r),
$$
with
$$
f(r)=\frac{{\rm e}^{-3i(t+r)\ell_1\ell_2(\ell_1+\ell_2)}}{-3i\ell_1\ell_2(\ell_1+\ell_2)},
\qquad
g(r)=\hat \tilde\psi_{\ell_1}(t+r)\hat{\tilde \chi}_{\ell_2}(t+r), \quad \ell_1,\ell_2\ne 0, \ell_1+\ell_2\ne 0.
$$

For the first term in \eqref{t1}, arguing as in the proof of Lemma 2.5 of \cite{kat0}, and using \eqref{bil_est} together with \eqref{reg_n0}, we claim
\begin{align}
\label{t4}
&\frac{1}{2\pi}
\Biggl\|
\sum_{\substack{\ell_1,\ell_2\ne 0 \\ \ell_1+\ell_2\ne 0}}
\frac{1}{3\ell_1\ell_2}
{\rm e}^{-3it\ell_1\ell_2(\ell_1+\ell_2)}
\bigl({\rm e}^{-3is\ell_1\ell_2(\ell_1+\ell_2)}-1\bigr)
\hat{\tilde\psi}_{\ell_1}(t+s)\hat{\tilde \chi}_{\ell_2}(t+s)
{\rm e}^{i(\ell_1+\ell_2)x}
\Biggr\|_{\dot H^1}\notag\\
\lesssim &
\sqrt{s}\, \|\tilde\psi(t+s)\|_{H^1}\,\|\tilde \chi(t+s)\|_{H^1}.
\end{align}
Indeed, by the definition of $\dot{H}^1$-norm, the unitary property of ${\rm e}^{t \partial_x^3}$ and a standard interpolation argument yield
\begin{align*}
   & \Bigg\|\sum_{\substack{\ell_1,\ell_2\ne 0 \\ \ell_1+\ell_2\ne 0}}
\frac{1}{3\ell_1\ell_2}
{\rm e}^{-3it\ell_1\ell_2(\ell_1+\ell_2)}
\bigl({\rm e}^{-3is\ell_1\ell_2(\ell_1+\ell_2)}-1\bigr)
\hat{\tilde\psi}_{\ell_1}(t+s)\hat{\tilde \chi}_{\ell_2}(t+s)
{\rm e}^{i(\ell_1+\ell_2)x}\Bigg\|_{\dot{H}^1}\\
&\quad = \Bigg\{\displaystyle\sum_{\ell \ne 0}\ell^2 \Bigg|\sum_{\substack{\ell_1,\ell_2,\ne 0\\ \ell_1+\ell_2=\ell}}\frac{1}{3\ell_1\ell_2}
{\rm e}^{-3it\ell_1\ell_2(\ell_1+\ell_2)}
\bigl({\rm e}^{-3is\ell_1\ell_2(\ell_1+\ell_2)}-1\bigr)
\hat{\tilde\psi}_{\ell_1}(t+s)\hat{\tilde \chi}_{\ell_2}(t+s)\Bigg|^2\Bigg\}^{1/2}\\
&\quad \lesssim\sqrt{s}\Bigg\{\displaystyle\sum_{\ell \ne 0}\ell^2 \Bigg(\sum_{\substack{\ell_1,\ell_2,\ne 0\\ \ell_1+\ell_2=\ell}}
\displaystyle\frac{|\ell_1+\ell_2|^{1/2}}{|\ell_1\ell_2|^{1/2}}
\big|\hat{\tilde\psi}_{\ell_1}(t+s)\big| \big|\hat{\tilde \chi}_{\ell_2}(t+s)\big|\Bigg)^2\Bigg\}^{1/2}\\
&\quad \lesssim \sqrt{s} \sup_{\ell_1,\ell_2\ne 0}\displaystyle\frac{|\ell_1+\ell_2|^{1/2}}{|\ell_1 \ell_2|^{1/2}}\big\|\tilde{\psi}^{(0)}(t+s) \tilde\chi^{(0)}(t+s) \big\|_{\dot{H}^1}.
\end{align*}
The claim \eqref{t4} follows because the $\dot{H}^1$-norm is controlled by the $H^1$-norm, together with the bilinear estimate \eqref{bil_est}. H\"older's inequality and \eqref{reg_n0} then give
\begin{equation}
    \label{e0}
    \mathbbm{E}\Bigg[\Bigg\|\sum_{\substack{\ell_1,\ell_2\ne 0\\\ell_1+\ell_2\ne 0}}
\frac{1}{3\ell_1\ell_2}
{\rm e}^{-3it\ell_1\ell_2(\ell_1+\ell_2)}
\bigl({\rm e}^{-3is\ell_1\ell_2(\ell_1+\ell_2)}-1\bigr)
\hat{\tilde\psi}_{\ell_1}(t+s)\hat{\tilde \chi}_{\ell_2}(t+s)
{\rm e}^{i(\ell_1+\ell_2)x}\Bigg\|^2_{H^1}\Bigg]\lesssim s.
\end{equation}
We next consider the second term in \eqref{t1}. In what follows, it is also important to recall that some of the computations below presented may be rigorously justified by considering standard Galerkin approximations of the linearized stochastic and deterministic equations and then proving that the resulting bounds are independent of the Galerkin truncation parameter. Since this is a standard procedure in the analysis of stochastic PDEs, we shall not discuss it further here; see, for example, \cite{deb_single}).

By the additivity of the noise and the chain rule, we have
\begin{align*}
\hat{\tilde \psi}_{\ell_1}(t+s)\hat{\tilde \chi}_{\ell_2}(t+s)
-\hat{\tilde \psi}_{\ell_1}(t)\hat{\tilde \chi}_{\ell_2}(t)
&=
\int_t^{t+s}\hat{\tilde \chi}_{\ell_2}(r)\frac{d}{dr}\hat{\tilde\psi}_{\ell_1}(r)\,dr
+
\int_t^{t+s}\hat{\tilde\psi}_{\ell_1}(r)\,d\hat{\tilde \chi}_{\ell_2}(r),
\end{align*}
where
\begin{align*}
\frac{d}{dt}\hat{\tilde\psi}_{\ell}(t)
&=
\frac{\mu}{\sqrt{2\pi}}
\sum_{\ell_1+\ell_2=\ell}
i\ell\,{\rm e}^{-it(\ell^3-\ell_1^3-\ell_2^3)}
\hat{\tilde\psi}_{\ell_1}(t)\hat{\tilde\psi}_{\ell_2}(t),
\\
d\hat{\tilde \chi}_{\ell}(t)
&=
\sqrt{\frac{2}{\pi}}\mu
\sum_{\ell_1+\ell_2=\ell}
i\ell\,{\rm e}^{-it(\ell^3-\ell_1^3-\ell_2^3)}
\hat{\tilde\psi}_{\ell_1}(t)\hat{\tilde \chi}_{\ell_2}(t)\,dt
+
\sum_k\big\langle {\rm e}^{t\partial_x^3}Q^{1/2}e_k,{\rm e}^{i\ell x}\big\rangle\,d\beta_k(t).
\end{align*}
Substituting these identities into the second term in \eqref{t1}, we obtain
\begin{align}
&
\frac{1}{2\pi}
\Biggl\|
\sum_{\substack{\ell_1,\ell_2 \ne 0 \\ \ell_1+\ell_2\ne 0}}
\frac{1}{3\ell_1\ell_2}
{\rm e}^{-3it\ell_1\ell_2(\ell_1+\ell_2)}
\Bigl(
\hat{\tilde\psi}_{\ell_1}(t+s)\hat{\tilde \chi}_{\ell_2}(t+s)
-\hat{\tilde\psi}_{\ell_1}(t)\hat{\tilde \chi}_{\ell_2}(t)
\Bigr)
{\rm e}^{i(\ell_1+\ell_2)x}
\Biggr\|_{\dot H^1}
\notag\\
&\hspace{3cm}\lesssim \mathcal{A}+\mathcal{B}+\mathcal{C},
\label{t3}
\end{align}
where
\begin{align*}
\mathcal{A}
&=
\frac{1}{2\pi}\frac{\mu}{\sqrt{2\pi}}
\Biggl\|
\sum_{\substack{\ell_1,\ell_2\ne 0 \\ \ell_1+\ell_2\ne 0}}
\frac{1}{3\ell_1\ell_2}
{\rm e}^{-3it\ell_1\ell_2(\ell_1+\ell_2)}
\int_t^{t+s}
\hat{\tilde \chi}_{\ell_2}(r)
\sum_{\ell_1'+\ell_2'=\ell_1}
i\ell_1\,{\rm e}^{-ir(\ell_1^3-\ell_1'^3-\ell_2'^3)}
\hat{\tilde\psi}_{\ell_1'}(r)\hat{\tilde\psi}_{\ell_2'}(r)\,dr\,
{\rm e}^{i(\ell_1+\ell_2)x}
\Biggr\|_{\dot H^1},\\
\mathcal{B}
&=
\frac{1}{2\pi}\sqrt{\frac{2}{\pi}}\mu
\Biggl\|
\sum_{\substack{\ell_1,\ell_2\ne 0 \\ \ell_1+\ell_2 \ne 0}}
\frac{1}{3\ell_1\ell_2}
{\rm e}^{-3it\ell_1\ell_2(\ell_1+\ell_2)}
\int_t^{t+s}
\hat{\tilde\psi}_{\ell_1}(r)
\sum_{\ell_1'+\ell_2'=\ell_2}
i\ell_2\,{\rm e}^{-ir(\ell_2^3-\ell_1'^3-\ell_2'^3)}
\hat{\tilde\psi}_{\ell_1'}(r)\hat{\tilde \chi}_{\ell_2'}(r)\,dr\,
{\rm e}^{i(\ell_1+\ell_2)x}
\Biggr\|_{\dot H^1},\\
\mathcal{C}
&=
\frac{1}{2\pi}
\Biggl\|
\sum_{\substack{\ell_1,\ell_2\ne 0 \\ \ell_1+\ell_2\ne 0}}
\frac{1}{3\ell_1\ell_2}
{\rm e}^{-3it\ell_1\ell_2(\ell_1+\ell_2)}
\int_t^{t+s}
\hat{\tilde\psi}_{\ell_1}(r)
\sum_{\ell'}
\big\langle {\rm e}^{r\partial_x^3}Q^{1/2}e_{\ell'},{\rm e}^{i\ell_2 x}\big\rangle
\,d\beta_{\ell'}(r)\,
{\rm e}^{i(\ell_1+\ell_2)x}
\Biggr\|_{\dot H^1}.
\end{align*}
For the term $\mathcal{A}$, the triangle inequality together with Cauchy--Schwarz in time yields
\begin{equation}
    \label{prov_b0}
\mathcal{A}^2
\le
\biggl(\frac{\mu}{2\pi}\biggr)^2\,
s\int_t^{t+s}
\big\| F(r)\big\|^2_{\dot{H}^1}\,dr,
\end{equation}
where
$$
F(r)
:=
\displaystyle\frac{i}{\sqrt{2\pi}} \sum_{\ell\ne 0}{\rm e}^{i\ell x}\sum_{\substack{\ell_1,\ell_2\ne 0 \\ \ell_1+\ell_2=\ell}}
\frac{1}{3\ell_1\ell_2}
{\rm e}^{-3it\ell_1\ell_2(\ell_1+\ell_2)}\ell_1
\hat{\tilde \chi}_{\ell_2}(r)
\sum_{\ell_1'+\ell_2'=\ell_1}
{\rm e}^{-ir(\ell_1^3-\ell_1'^3-\ell_2'^3)}
\hat{\tilde\psi}_{\ell_1'}(r)\hat{\tilde\psi}_{\ell_2'}(r).
$$
Thus, noting that $\big|\frac{\ell_1}{\ell_1 \ell_2}\big|\le 1$, for $\ell_1,\ell_2 \in\mathbbm{Z}, \ell_1,\ell_2\ne 0$, we get
\begin{align*}
    \big\| F(r)\big\|^2_{\dot{H}^1}& = \displaystyle\sum_{\ell \ne 0}\ell^2 \Bigg|\sum_{\substack{\ell_1,\ell_2\ne 0\\ \ell_1+\ell_2=\ell}}\frac{1}{3\ell_1\ell_2}
{\rm e}^{-3it\ell_1\ell_2(\ell_1+\ell_2)}\ell_1
\hat{\tilde \chi}_{\ell_2}(r)
\sum_{\ell_1'+\ell_2'=\ell_1}
{\rm e}^{-ir(\ell_1^3-\ell_1'^3-\ell_2'^3)}
\hat{\tilde\psi}_{\ell_1'}(r)\hat{\tilde\psi}_{\ell_2'}(r)\Bigg|^2\\
&\lesssim \sum_{\ell \in\mathbbm{Z}}\ell^2 \Bigg(\sum_{\ell_1+\ell_2=\ell}
\big|\hat{\tilde \chi}_{\ell_2}(r)\big|
\sum_{\ell_1'+\ell_2'=\ell_1}
\big|\hat{\tilde\psi}_{\ell_1'}(r)\big| \big|\hat{\tilde\psi}_{\ell_2'}(r)\big|\Bigg)^2.
\end{align*}
Recalling the definition in Section \ref{ass_section} for $\chi^{(0)}$ and $\psi^{(0)}$, we thus obtain
\begin{align*}
    \big\| F(r)\big\|^2_{\dot{H}^1}\lesssim \big\|\tilde \chi^{(0)} (r) \tilde\psi^{(0)}(r)\tilde\psi^{(0)}(r)\big\|^2_{\dot{H}^1}.
\end{align*}
Using \eqref{bil_est}, we finally achieve 
\begin{equation}
    \label{prov_b}
 \big\| F(r)\big\|^2_{\dot{H}^1}\lesssim \big\| \tilde \chi(r)\big\|^2_{H^1} \big\| \tilde \psi(r)\big\|^4_{H^1}.
\end{equation}
Inserting \eqref{prov_b} into \eqref{prov_b0}, taking expectations on both sides, using H\"older's inequality, and then applying \eqref{reg_n0}, we obtain
\begin{align}
    \mathbbm{E}\mathcal{A}^2 &\lesssim s \int_t^{t+s}\mathbbm{E}\big[\big\| \tilde \chi(r)\big\|^2_{H^1} \big\| \tilde \psi(r)\big\|^4_{H^1}\big]\,dr \notag \\
    &\label{tA}\lesssim s^2.
\end{align}
In a completely analogous way, we also obtain
\begin{equation}
\label{tB}
\mathbbm{E}\mathcal{B}^2\lesssim s^2.
\end{equation}
It remains to bound the term $\mathcal{C}$.

Set
$$
G(x)
:=\displaystyle\frac{1}{\sqrt{2\pi}}
\sum_{\substack{\ell_1,\ell_2\ne 0 \\ \ell_1+\ell_2\ne 0}}
\frac{1}{3\ell_1\ell_2}
{\rm e}^{-3it\ell_1\ell_2(\ell_1+\ell_2)}
\int_t^{t+s}
\hat{\tilde\psi}_{\ell_1}(r)
\sum_{\ell'}
\big\langle {\rm e}^{r\partial_x^3}Q^{1/2}e_{\ell'},{\rm e}^{i\ell_2 x}\big\rangle
\,d\beta_{\ell'}(r)\,
{\rm e}^{i(\ell_1+\ell_2)x}.
$$
Then
$$
\mathbbm{E}\mathcal C^2
=
\frac{1}{2\pi}
\mathbbm{E}\|G\|_{\dot H^1}^2.
$$
Note that $G(x)=\frac{1}{\sqrt{2\pi}}\sum_{\ell\ne 0}\hat{G}_\ell {\rm e}^{i\ell x}$, where
$$
\hat G_\ell
=
\sum_{\substack{\ell_1,\ell_2\ne 0\\ \ell_1+\ell_2=\ell}}
\frac{1}{3\ell_1\ell_2}
{\rm e}^{-3it\ell_1\ell_2\ell}
\int_t^{t+s}
\hat{\tilde\psi}_{\ell_1}(r)
\sum_{\ell'}
\big\langle {\rm e}^{r\partial_x^3}Q^{1/2}e_{\ell'},{\rm e}^{i\ell_2 x}\big\rangle
\,d\beta_{\ell'}(r).
$$
Hence, by the definition of the homogeneous Sobolev seminorm,
$
\mathbbm{E}\|G\|_{\dot H^1}^2
=
\mathbbm{E}\sum\limits_{\ell\in\mathbb Z\setminus\{0\}}
|\ell|^2\,|\hat G_\ell|^2.
$ 
Using It\^o's isometry and the independence of the Brownian motions, we obtain
\begin{align*}
\mathbbm{E}|\hat G_\ell|^2
&=
\mathbbm{E}\int_t^{t+s}
\sum_{\ell'}
\Biggl|
\sum_{\substack{\ell_1,\ell_2\ne 0\\ \ell_1+\ell_2=\ell}}
\frac{1}{3\ell_1\ell_2}
{\rm e}^{-3it\ell_1\ell_2\ell}
\hat{\tilde\psi}_{\ell_1}(r)
\big\langle {\rm e}^{r\partial_x^3}Q^{1/2}e_{\ell'},{\rm e}^{i\ell_2 x}\big\rangle
\Biggr|^2\,dr.
\end{align*}
Therefore,
\begin{align*}
\mathbbm{E}\mathcal C^2
&=
\frac{1}{2\pi}
\int_t^{t+s}
\sum_{\ell'}\sum_{\ell\in\mathbb Z\setminus\{0\}}
|\ell|^2\,
\mathbbm{E}
\Biggl|
\sum_{\substack{\ell_1,\ell_2\ne 0 \\ \ell_1+\ell_2=\ell}}
\frac{1}{3\ell_1\ell_2}
{\rm e}^{-3it\ell_1\ell_2\ell}
\hat{\tilde\psi}_{\ell_1}(r)
\big\langle {\rm e}^{r\partial_x^3}Q^{1/2}e_{\ell'},{\rm e}^{i\ell_2 x}\big\rangle
\Biggr|^2\,dr\\
&\lesssim \frac{1}{2\pi}
\int_t^{t+s}
\sum_{\ell'}\sum_{\ell\in\mathbb Z\setminus\{0\}}
|\ell|^2\,
\mathbbm{E}
\Biggl(
\sum_{\ell_1+\ell_2=\ell}
|\hat{\tilde\psi}_{\ell_1}(r)|
\Bigl|
\big\langle {\rm e}^{r\partial_x^3}Q^{1/2}e_{\ell'},{\rm e}^{i\ell_2 x}\big\rangle
\Bigr|
\Biggr)^2\,dr.
\end{align*}
Similarly as in the estimate of $\mathcal A$, for each $\ell'$ and $r$ we define
$$
f_{\ell'}(r)
:=
\frac{1}{\sqrt{2\pi}}
\sum_{\ell_2\in\mathbb Z}
\Bigl|
\big\langle {\rm e}^{r\partial_x^3}Q^{1/2}e_{\ell'},{\rm e}^{i\ell_2 x}\big\rangle
\Bigr|
{\rm e}^{i\ell_2 x}.
$$
Then, we get
$$
\tilde\psi^{(0)}(r) f_{\ell'}^{(0)}(r)= \displaystyle\frac{1}{\sqrt{2\pi}}\sum_{\ell\in\mathbbm{Z}}{\rm e}^{i\ell x}\displaystyle\frac{1}{\sqrt{2\pi}} \sum_{\ell_1+\ell_2=\ell} \big| \hat{\tilde\psi}_{\ell_1}(r)\big|\big|\big\langle {\rm e}^{r\partial_x^3}Q^{1/2}e_{\ell'},{\rm e}^{i\ell_2 x}\big\rangle\big| 
$$
Consequently,
$$
\displaystyle\frac{1}{2\pi}\sum_{\ell\in\mathbb Z\setminus\{0\}}
|\ell|^2
\Biggl(
\sum_{\ell_1+\ell_2=\ell}
|\hat{\tilde\psi}_{\ell_1}(r)|
\Bigl|
\big\langle {\rm e}^{r\partial_x^3}Q^{1/2}e_{\ell'},{\rm e}^{i\ell_2 x}\big\rangle
\Bigr|
\Biggr)^2
=
\|\tilde\psi^{(0)}(r)\,f_{\ell'}^{(0)}(r)\|_{\dot H^1}^2.
$$
Hence,
$$
\mathbbm{E}\mathcal C^2
\lesssim
\int_t^{t+s}
\sum_{\ell'}
\mathbbm{E}\|\tilde\psi^{(0)}(r)\,f_{\ell'}^{(0)}(r)\|_{\dot H^1}^2\,dr.
$$
Using \eqref{bil_est}, we further obtain
$$
\|\tilde\psi^{(0)}(r)\,f_{\ell'}^{(0)}(r)\|_{\dot H^1}
\le
\|\tilde\psi^{(0)}(r)\,f_{\ell'}^{(0)}(r)\|_{H^1}
\lesssim
\|\tilde\psi(r)\|_{H^1}\,\|f_{\ell'}(r)\|_{H^1}.
$$
Therefore,
\begin{align*}
\mathbbm{E}\mathcal C^2
&\lesssim
\int_t^{t+s}
\|\tilde\psi(r)\|_{H^1}^2
\sum_{\ell'}\|f_{\ell'}(r)\|_{H^1}^2\,dr\\
&\lesssim 
\int_t^{t+s}
\|\tilde\psi(r)\|_{H^1}^2
\sum_{\ell'}\|{\rm e}^{r\partial_x^3}Q^{1/2}e_{\ell'}\|_{H^1}^2\,dr\\
&\lesssim \big\| Q^{1/2}\big\|^2_{\mathcal{L}_2^1} \int_t^{t+s}
\|\tilde\psi(r)\|_{H^1}^2\,dr.
\end{align*}
Invoking \eqref{reg_n0}, we conclude that
\begin{equation}
\label{tC}
\mathbbm{E}\mathcal C^2\lesssim s.
\end{equation}
Inserting \eqref{tA}--\eqref{tC} into \eqref{t3}, we arrive at
\begin{equation}
\label{t5}
\mathbbm{E}
\Biggl[
\frac{1}{4\pi^2}
\Biggl\|
\sum_{\substack{\ell_1,\ell_2\ne 0\\ \ell_1+\ell_2\ne 0}}
\frac{1}{3\ell_1\ell_2}
{\rm e}^{-3it\ell_1\ell_2(\ell_1+\ell_2)}
\Bigl(
\hat {\tilde\psi}_{\ell_1}(t+s)\hat{\tilde \chi}_{\ell_2}(t+s)
-\hat {\tilde\psi}_{\ell_1}(t)\hat{\tilde \chi}_{\ell_2}(t)
\Bigr)
{\rm e}^{i(\ell_1+\ell_2)x}
\Biggr\|^2_{\dot H^1}
\Biggr]
\lesssim s.
\end{equation}

The third summand in \eqref{t1} can be bounded similarly, and we then omit tedious computations. 
Combining this bound and \eqref{e0}, we conclude that
\begin{equation}
\label{n1}
\mathbbm{E}\Bigl\|
\int_0^s {\rm e}^{(t+r)\partial_x^3}\partial_x
\Bigl(
{\rm e}^{-(t+r)\partial_x^3}\tilde\psi(t+r)\,
{\rm e}^{-(t+r)\partial_x^3}\tilde{\chi}(t+r)
\Bigr)\,dr
\Bigr\|_{\dot{H}^1}^2
\lesssim s.
\end{equation}
Finally, thanks to \eqref{hom_b1}, the bounds \eqref{n1} and \eqref{n2} give the desired result.

\bibliographystyle{amsplain}

\section*{Acknowledgements}
The authors thank Prof. Charles-Edouard Bréhier for his comments and suggestions.
The second and third authors are members of the INdAM Research group GNCS. This work has been supported by PRIN-MUR 2022 project 20229P2HEA “Stochastic numerical modelling for sustainable innovation” (CUP: E53C24002280006), granted by MUR within the
scrolling of the final rankings of the PRIN 2022 call. This work has also been supported by PRIN-PNRR project BAT-MEN (BATtery Modeling, Experiments \& Numerics) - Project code P20228C2PP, CUP E53D23017940001, funded by MUR (Italian Ministry of University and Research) and European Union – NextGenerationEU. This work is partially supported by MOST National Key R\&D Program No. 2024FA1015900, the Hong Kong Research Grant Council GRF grant 15302823, GRF grant 15301025, NSFC/RGC Joint Research Scheme N$\_$PolyU5141/24, NSFC grants (No. 12522119,  No. 12301526, No. 12471386 and No. 12461160278), and the
CAS AMSS-PolyU Joint Laboratory of Applied Mathematics.

\end{document}